\documentclass[reqno]{amsart}

\usepackage{latexsym} \usepackage{amssymb} \usepackage{amsfonts}
\usepackage{amscd} \usepackage{amsmath}
\usepackage[all]{xy}

\usepackage[dvipsnames,svgnames,x11names,hyperref]{xcolor}
\usepackage[colorlinks=true]{hyperref}
\hypersetup{urlcolor=blue,linkcolor=Maroon,citecolor=blue,colorlinks=true}

\newtheorem{prop}{Proposition}[section]
\newtheorem{lem}[prop]{Lemma}
\newtheorem{cor}[prop]{Corollary}
\newtheorem{thm}[prop]{Theorem}
\newtheorem{defn}[prop]{Definition}
\newtheorem{exmp}[prop]{Example}
\newtheorem{rem}[prop]{Remark}
\newtheorem{notation}[prop]{Notation}
\numberwithin{equation}{subsection}

 \def\mm{\mathfrak{m}} \def\nn{\mathfrak{n}} \def\aa{\mathfrak{A}}
  
\def\Disc{\mathrm{Disc}} \def\Res{\mathrm{Res}} \def\Dc{\mathcal{D}}
\def\ZZ{\mathbb{Z}}\def\NN{\mathbb{N}}  \def\aa{\mathfrak{A}} \def\TF{\mathrm{TF}}
\def\EE{\mathcal{E}} 
 \def\pp{\mathfrak{P}}
\def\supp{\mathrm{supp}}
\def\C{\mathrm{C}} \def\Proj{\mathrm{Proj}} \def\Spec{\mathrm{Spec}} \def\Red{\mathrm{Red}}
\def\Hess{\mathrm{Hess}} \def\HH{{\mathbb{H}}}

\def\dd{{\mathrm{d}}}
\def\Rees{{\mathrm{Rees}}}
\def\Frac{{\mathrm{Frac}}}
\def\CC{{\mathbb{C}}}
\newcommand\quotient[2]{
        \mathchoice
            {
                \text{\raise1ex\hbox{$#1$}\big/\lower1ex\hbox{$#2$}}%
            }
            {
                #1\,/\,#2
            }
            {
                #1\,/\,#2
            }
            {
                #1\,/\,#2
            }
    }

\title[The discriminant of homogeneous polynomials]{A computational approach to the discriminant of homogeneous polynomials}

\date{\today}

\author{Laurent Bus\'e}
\address{INRIA Sophia Antipolis - M\'editerran\'ee,
 2004 routes des Lucioles, B.P. 93, 
06902 Sophia Antipolis, France.}
\email{laurent.buse@inria.fr}
\urladdr{www-sop.inria.fr/members/Laurent.Buse/} 

\author{Jean-Pierre Jouanolou}
\address{Universit\'e Louis Pasteur, 7 rue Ren\'e Descartes, 67084 Strasbourg Cedex, France.}
\email{jean-pierre.jouanolou@math.unistra.fr}

\begin{document}

\maketitle

\begin{abstract}
In this paper, the discriminant of homogeneous polynomials is studied in two particular cases: a single homogeneous polynomial and a collection of $n-1$ homogeneous polynomials in $n\geq 2$ variables. In these two cases, the discriminant is defined over a large class of coefficient rings by means of the resultant. Many formal properties and computational rules are provided and the geometric interpretation of the discriminant is investigated over a general coefficient ring, typically a domain.
\end{abstract}




\section{Introduction}

The discriminant of a collection of polynomials gives information about the nature of the common roots of these polynomials. Following the example of the very classical discriminant of a single univariate polynomial, it is a fundamental tool in algebraic geometry which is very useful and has many applications. Several definitions of the discriminant can be found in the literature, but they are not all equivalent. Recall briefly the most standard one (\cite{GKZ}) for polynomials with coefficients in the field of complex numbers $\CC$: given integers $1\leq c \leq n$ and $1\leq d_1,\ldots,d_c$, denote by $S$ the set of all $c$-uples of homogeneous polynomials $f_1,\ldots,f_c$ in the polynomial ring $\CC[X_1,\ldots,X_n]$ of respective degrees $d_1,\ldots,d_c$. The subset $D$ of $S$ corresponding to the $c$-uples $f_1,\ldots,f_c$ such that $\{f_1=f_2=\ldots=f_c=0\}$ is not smooth and of codimension $c$ is called the discriminant locus. It is well-known that $D$ is an irreducible algebraic variety of codimension one providing $d_i\geq 2$ for some $i\in\{1,\ldots,c\}$ or $c=n$. The discriminant is then defined as an equation of $D$ (and set to be 1 if $D$ is not of codimension one).

\medskip

As far as we know, the literature on the theory of the discriminant goes back to an outstanding paper by Sylvester \cite{Sylvester,SylvesterBis} where among others, an explicit formula for the degree of the discriminant is given. Then, one find the works by Mertens \cite{Mertens1886,Mertens1892}, where the concept of inertia forms is already used, and some other works by K\"onig \cite{Koenig}, by Kronecker \cite{Kronecker}, by Ostrowki \cite{Ostrowski} and also by Henrici \cite{He1868}. There is also an important contribution by Krull \cite{Krull1,Krull2} who studied Jacobian ideals and some properties of the discriminant, especially  in the case $c=n-1$. An extensive study of the case $c=1$ can be found in a Bourbaki manuscript by Demazure \cite{Bou} that was unfortunately left unpublished until very recently \cite{Dem}.

For the past twenty years, one can observe a regain of interest, in particular regarding properties with respect to the shape (total degree, partial degrees, Newton polyhedron, etc) of the discriminant. Unlike the previously mentioned works, the techniques are here more advanced and uses homological methods. The book by Gelfand, Kapranov and Zelevinsky \cite{GKZ} was definitely a turning point  in this modern approach. One can also mention the paper by Scheja and Storch \cite{ScSt} and the more recent one by Esterov \cite{Esterov} that deals with more general grading of the polynomials (they correspond to anisotropic projective spaces and more general toric varieties respectively). It is as well worth mentioning the recent paper by Benoist \cite{Benoist} where the degree of the discriminant is carefully studied (see also \cite{Sylvester,SylvesterBis}).

\medskip

There are drawbacks to the modern above-mentioned definition of the discriminant. First, this definition is not stable under specialization (or change of basis). In other words, the discriminant is a polynomial in the coefficients of the polynomials $f_1,\ldots,f_c$ and a given specialization of these coefficients does not always commutes with this construction of the discriminant. Such a property is however a natural request for the discriminant. Notice that it is actually well satisfied when defining the discriminant of a single univariate polynomial $f$ as the determinant of the Sylvester matrix associated to $f$ and its first derivative. Second, the discriminant is defined up to multiplication by a nonzero constant. This is not satisfactory when the value of the discriminant is important, and not only its vanishing, as this is for instance the case for some applications in the fields of arithmetic geometry and number theory (see for instance the recent paper \cite{SaSe}). Finally, this definition is only valid under the hypothesis that the ground ring is a field, often assumed to be algebraically closed and of characteristic zero. But for many applications, it is very useful to understand the behavior of the discriminant under general ground rings. 

These three drawbacks are important obstructions that prevent the discriminant from having a well developed formalism, in particular some properties and formulas that allow to handle it as a computational tool. In many situations such a formalism is actually more important than the value of the discriminant itself, even more important because this value is often unreachable by direct computations. Moreover, the discriminant gives more insights if it is defined without ambiguity (in particular not up to a nonzero constant multiplicative factor) over a general coefficient ring (see for instance \cite{BM}). As a first stage, the goal of this paper is to provide such a theory of the discriminant in the two cases $c=1$ and $c=n-1$. To this aim, we will define the discriminant as a particular instance of the resultant. In this way, we will take advantage of the existing formalism of the resultant as developed by Jouanolou \cite{J91} and will be able to rigorously state that the discriminant is stable under a change of basis. As a consequence of this approach, we will provide a detailed analysis of the geometric behavior of our definition of the discriminant.

\medskip 

After some reminders and preliminaries on the resultant in Section \ref{prem}, Section \ref{dim0} deals with the case $c=n-1$, that is to say  the discriminant of a finite set of points in complete intersection in a projective space. Such a discriminant already appeared in two papers by Krull \cite{Krull1,Krull2}. Based on them, we give a general and universal definition of the discriminant and develop further its formalism. In particular, we provide a full description of the base change formula. Then, if the ground ring $k$ is assumed to be a domain, we show that the discriminant is a prime polynomial if $2\neq 0$ in $k$ and is the square of a prime polynomial otherwise. 

In Section \ref{codim1}, we will study the case $c=1$, that is to say the discriminant of a hypersurface in a projective space. This case is the more classical and it already appeared in \cite[chapter 12.B]{GKZ}, in \cite{Bou} and more recently in \cite{SaSe}. Our contribution is here on the extension of the study of the discriminant to an arbitrary commutative ground ring $k$, as well as several formal properties. If $k$ is a domain, we show that if $2\neq 0$ and $n$ is odd then the discriminant is a prime polynomial. Otherwise, if $2=0$ in $k$ and $n$ is even, the discriminant is the square of a prime polynomial. In addition, we also provide a detailed study of the birationality of the canonical projection of the incidence variety onto the discriminant locus.  

Finally, we end this paper with an appendix where we give rigorous proofs of two remarkable formulas that are due to F.~Mertens \cite{Mertens1892,Mertens1886}. We will use these formulas at some point in text, but these formulas are definitely interesting on their own. 

\medskip

All rings are assumed to be commutative and with unity.

\tableofcontents

\section{Preliminaries}\label{prem}

We recall here the basic definitions and properties of inertia forms
and the resultant that we will use in the rest of this paper to study the
discriminant of homogeneous polynomials. Our main source is  the
monograph \cite{J91}
where a detailed exposition can be found.

\medskip

Suppose given $r\geq 1$ homogeneous polynomials
of positive degrees $d_1,\ldots,d_{r}$, respectively, in the
variables $X_1,\ldots,X_n$, all assumed to have weight 1,
$$f_i(X_1,\ldots,X_n)=\sum_{|\alpha|=d_i}U_{i,\alpha}X^\alpha, \ \ 
i=1,\ldots,r.$$
Let $k$ be a commutative ring and set ${{}_kA}:=k[U_{i,\alpha}\,|\, i=1,\ldots,r, |\alpha|=d_i]$
the universal coefficient ring over $k$. Then $f_i \in
{{}_kC}:={{}_kA}[X_1,\ldots,X_n]$ for all $i=1,\ldots,r$. We define the ideal 
$I:=(f_1,\ldots,f_r) \subset {{}_kC}$ and the graded quotient ring ${{}_kB}:={{}_kC}/I$. The main purpose of
elimination theory is the study of the image of the canonical 
projection
$$\mathrm{Proj}({{}_kB}) \rightarrow \mathrm{Spec}({{}_kA})$$
which corresponds to the elimination of the variables $X_1,\ldots,X_n$ in the
polynomial system $f_1=\cdots=f_r=0$.  It turns out that this image is
closed (the latter  projection is a projective morphism) and its defining
ideal, that we will denote by ${{}_k\aa}$ and which is usually called the
\emph{resultant} (or \emph{eliminant}) \emph{ideal},
consists of the elements of
${{}_kA}$ which are contained in $I$ after multiplication by some power of
the maximal ideal $\mm:=(X_1,\ldots,X_n) \subset {{}_kC}$. In other words, ${{}_k\aa}$ is the
degree 0 part of the 0th local cohomology module of ${{}_kB}$ with respect to $\mm$,
i.e.~${{}_k\aa}=H^0_\mm({{}_kB})_0$. 

\subsection{Inertia forms}\label{in-form} First introduced by Hurwitz,  \emph{inertia
  forms} reveal a powerful tool to study
the resultant ideals, notably in the case $r=n$ corresponding
to the theory of the resultant, and more generally elimination theory.

\begin{defn} The ideal of inertia forms of the ideal $I$ with respect
  to the ideal $\mm$ is the ideal of ${{}_kC}$
  $${}_k\mathrm{TF}_\mm(I):=\pi^{-1}(H^0_\mm({{}_kB}))=\{f\in {{}_kC} : \exists \nu
  \in \mathbb{N} \ \mm^\nu f \subset I\} \subset {{}_kC}$$
  where $\pi$ denotes the canonical projection ${{}_kC}\rightarrow {{}_kB}={{}_kC}/I$.
\end{defn}
Observe that the inertia forms ideal is naturally graded and  that ${}_k\aa={{}_k\TF}_\mm(I)_0$. We recall two useful
other descriptions of this ideal.

Let us distinguish, for all $i=1,\ldots,r$, the particular coefficient
$\mathcal{E}_i:=U_{i,(0,\ldots,0,d_i)}$ of the polynomial $f_i$ which can 
be rewritten in ${{}_kC}[X_n^{-1}]$
$$f_i=X_n^{d_i} (\mathcal{E}_i + \sum_{ 
      \alpha\neq (0,\ldots,0,d_i)} U_{i,\alpha}X^\alpha X_n^{-d_i} ).$$
Then we get  the isomorphism of $k$-algebras 
\begin{eqnarray}\label{BX}
{{}_kB}_{X_n} & \xrightarrow{\sim} & k[U_{j,\alpha} : U_{j,\alpha}\neq
\mathcal{E}_i][X_1,\ldots,X_n][X_n^{-1}] \\
\mathcal{E}_i & \mapsto & \mathcal{E}_i - \frac{f_i}{X_n^{d_i}}=-\sum_{ 
      \alpha\neq (0,\ldots,0,d_i)} U_{i,\alpha}X^\alpha X_n^{-d_i} \nonumber
  \end{eqnarray}
and of course similar isomorphisms for all the ${{{}_kB}_{X_i}}'s$. They show
that $X_i$  is a nonzero divisor in ${{}_kB}_{X_j}$ for all couple $(i,j) \in
\{1,\ldots,n\}^2$, and by the way that, for all $i \in \{1,\ldots,n\}$,
\begin{align}\label{TF}
  {}_k\mathrm{TF}_\mm(I)=\{ f \in {{}_kC} : \exists \nu
  \in \mathbb{N} \ X_i^\nu f \subset I \}=\mathrm{Ker}({{}_kC}\rightarrow
  {{}_kB}_{X_i}).
\end{align}
In particular, if the commutative ring $k$ is a domain, it follows
that the ${{}_kB}_{X_i}$'s are also domains and thus that
${{}_k\TF}_\mm(I)$ is a \emph{prime} ideal of ${}_kC$, as well as
${}_k\aa$. 
Note also that, as a simple consequence, we obtain the equality   
\begin{equation}\label{form1}
{{}_k\aa}={}_k\mathrm{TF}_\mm(I)_0={{}_kA}\cap (\tilde{f}_1,\ldots,\tilde{f}_r)
\end{equation}
where $\tilde{f}_i(X_1,\ldots,X_{n-1})=f_i(X_1,\ldots,X_{n-1},1) \in
{{}_kA}[X_1,\ldots,X_{n-1}].$
 
The combination of  \eqref{TF} and \eqref{BX} also 
gives another  interesting description of ${{}_k\TF}_\mm(I)$. Indeed,
similarly  to \eqref{BX}, we define the morphism
$$\tau:{{}_kC} \rightarrow k[U_{i,\alpha} \,|\,  U_{i,\alpha}\neq
\mathcal{E}_i][X_1,\ldots,X_n][X_n^{-1}] : 
\mathcal{E}_i  \mapsto  \mathcal{E}_i - \frac{f_i}{X_n^{d_i}}$$
which is sometimes called the Kronecker substitution. Then, it follows 
that
\begin{equation}\label{perron}
 {}_k\mathrm{TF}_\mm(I)=\{
 f \in {{}_kC} : \tau(f) =0 \}.
\end{equation}
In other words, an inertia form is a polynomial in ${{}_kC}$ that vanishes
after  the substitution of  each $\mathcal{E}_i$ by  $\mathcal{E}_i -
f_i/X_n^{d_i}$ for all  $i=1,\ldots,r$. This property yields in
particular the following refinement of \eqref{form1}:
\begin{equation}\label{form2}
{{}_k\aa}={}_k\mathrm{TF}_\mm(I)_0={{}_kA}\cap \sum_{i=1}^r \tilde{f}_i.{{}_kA}[\tilde{f}_1,\ldots,\tilde{f}_r].
\end{equation}

\subsection{The resultant}
We now turn to the particular case $r=n$, usually
called the \emph{principal case of elimination}.
As we are going to recall, in this situation the resultant ideal ${{}_k\aa}$ is principal and the
resultant is one of its generator. We will need the 
\begin{notation}\label{not}
Let $k$ be a commutative ring. Suppose given a $k$-algebra $R$ and, for all integer $i\in
 \{1,\ldots,n\}$, a homogeneous polynomial of degree $d_i$ in the
    variables $X_1,\ldots,X_n$
    $$g_i=\sum_{|\alpha|=d_i}u_{i,\alpha}X^\alpha \in
    R[X_1,\ldots,X_n]_{d_i}.$$
    We denote by $\theta$ the k-algebra morphism $\theta:{{}_kA}\rightarrow R:
    U_{j,\alpha} \mapsto u_{j,\alpha}$ corresponding to the
    \emph{specialization} of the polynomials $f_i$ to the polynomials
    $g_i$.  Then, for any element $a \in
    {{}_kA}$ we set $a(g_1,\ldots,g_n):=\theta(a).$ In particular, if
    $R={{}_kA}$ and $\theta$ is the identity (i.e.~$g_i=f_i$ for all $i$),
    then $a=a(f_1,\ldots,f_n)$.  
\end{notation}

\begin{prop}[{\cite[\S2]{J91}}]\label{prop:resultant}
The ideal  ${{}_\ZZ\aa}$ of ${}_\ZZ A$ is principal and has a
\emph{unique generator}, denoted  ${}_\ZZ\Res$,
which satisfies 
  \begin{equation}\label{normres}
 {}_\ZZ \Res(X_1^{d_1},\ldots,X_n^{d_n})=1.
 \end{equation}
Moreover, for any commutative ring $k$, the 
  ideal ${{}_k\aa}$ of ${{}_kA}$ is also principal and generated by
  ${}_k\Res:=\lambda({}_\ZZ\Res)$, 
where $\lambda$ denotes the  canonical morphism
  $$\lambda : {{}_\ZZ A}:=\mathbb{Z}[U_{i,\alpha}]\rightarrow
{{}_kA}=k[U_{i,\alpha}] :    U_{j,\alpha} \mapsto U_{j,\alpha}.$$
In addition, ${}_k\Res$ is a nonzero divisor in ${{}_kA}$.
\end{prop}

In view of Notation \ref{not}, we have defined the resultant of
any set of homogeneous polynomials of positive degrees $f_1,\ldots,f_n 
\in  k[X_1,\ldots,X_n]$, where $k$ denotes any commutative ring; 
we will denote it by $\Res(f_1,\ldots,f_n)$ without any possible
confusion. Indeed, this resultant is by definition obtained as a specialization of
the corresponding resultant in the generic case over $\mathbb{Z}$,
that is to say ${}_\ZZ\Res$ (with the corresponding choice of degrees
for the input polynomials). Therefore, the resultant has the property to be stable under specialization whereas this is not the case of the inertia forms ideal in general. Nevertheless, we have the following property.

\begin{prop}\label{prop:resbasechange} The ideal of inertia forms is stable under specialization up to radical. More precisely, let $R$ be a commutative ring and $\rho:{}_\ZZ A \rightarrow R$ be a specialization morphism. Then, the ideals $\rho({}_\ZZ \TF_\mm(I)_0).R=(\rho({}_\ZZ\Res))$ and $\TF_\mm(\rho(I).R)_0$ are two ideals in $R$ that have the same radical.        \end{prop} 
\begin{proof} This result corresponds to a general property of proper morphisms under change of basis. As we already said, the canonical projection $\mathrm{Proj}({{}_\ZZ B}) \rightarrow \mathrm{Spec}({{}_\ZZ A})$ is a projective, hence proper, morphism whose image is closed and defined by the ideal ${}_\ZZ \TF_\mm(I)_0 \subset {}_\ZZ A$. The specialization $\rho$ corresponds to a change of basis from $\Spec(R)$ to $\mathrm{Spec}({{}_\ZZ A})$. Since the support of the closed image of a proper morphism is stable under change of basis, we deduce that, as claimed, the support of the inverse image of the closed image of $\mathrm{Proj}({{}_\ZZ B}) \rightarrow \mathrm{Spec}({{}_\ZZ A})$ is equal to the support of the closed image of 
	$$\mathrm{Proj}({{}_\ZZ B})\times_{\mathrm{Spec}({{}_\ZZ A})} \Spec(R) \rightarrow \Spec(R).$$  
	
We can give another proof, somehow more elementary, of this proposition. Indeed, by specialization it is clear that 
$$\rho({}_\ZZ \TF_\mm(I)_0).R=(\rho({}_\ZZ\Res))=(\Res(\rho(f_1),\ldots,\rho(f_n))) \subset \TF_\mm(\rho(I).R)_0.$$
Let $a\in \TF_\mm(\rho(I).R)_0$, so that there exists an integer $N$ such that for all $i=1,\ldots,n$  
$$X_i^{N}a \in (\rho(f_1),\ldots,\rho(f_n))\subset R[X_1,\ldots,X_n].$$ 
It follows that 
$$(X_1^{N}a, X_2^{N}a,\ldots, X_n^{N}a) \subset (\rho(f_1),\ldots,\rho(f_n)) \subset R[X_1,\ldots,X_n]$$
and hence that $\Res(\rho(f_1),\ldots,\rho(f_n))$ divides $\Res(X_1^{N}a, \ldots, X_n^{N}a)$ in $R$ by \cite[\S 5.6]{J91}. Now, using \cite[Proposition 2.3(ii)]{J91}, we obtain that
$$\Res(X_1^{N}a, \ldots, X_n^{N}a)=a^{nN^{n-1}}\Res(X_1^N,\ldots,X_n^N)=a^{nN^{n-1}} \in R.$$
Therefore, $\Res(\rho(f_1),\ldots,\rho(f_n))$ divides $a^{nN^{n-1}}$ in $R$ and hence $\TF_\mm(\rho(I).R)_0$ is contained in the radical of the ideal $(\Res(\rho(f_1),\ldots,\rho(f_n)))\subset R$.
\end{proof}

The resultant have a lot of interesting properties that we are going to
use all along this paper; we refer the reader to \cite[\S5]{J91} and 
 each time we will need one of these properties we will quote a precise reference from this source (as we have just done in the proof of the previous proposition).

\medskip

We end this paragraph by recalling the old-fashion way, still
very useful in some cases, to define the resultant (see for instance \cite{Zar37}). To do
this, let us introduce $n$ new indeterminates $T_1,\ldots,T_n$. From 
\eqref{form2} we deduce easily that 
\begin{multline*}
	{}_k\mathrm{TF}_\mm( (f_1-T_1X_n^{d_1},\ldots,f_n-T_nX_n^{d_n}) )_0 = \\ 
	\{ P(T_1,\ldots,T_n) \in {{}_kA}[T_1,\ldots,T_n]  :
	P(\tilde{f}_1,\ldots,\tilde{f}_n)=0\},	
\end{multline*}
equality which can be rephrased by saying that \emph{the kernel of the map
$$\phi : {{}_kA}[T_1,\ldots,T_n] \rightarrow {{}_kA}[X_1,\ldots,X_{n-1}]: T_i
\mapsto \tilde{f}_i$$
is a principal ideal generated by
$\Res(f_1-T_1X_n^{d_1},\ldots,f_n-T_nX_n^{d_n})$}.
Thus, we obtain an \emph{explicit} formulation of \eqref{form2} under the form
\begin{equation}\label{restilde}
  \Res(f_1-\tilde{f}_1X_n^{d_1},\ldots,f_n-\tilde{f}_nX_n^{d_n})=0.
  \end{equation}

\subsection{A generalized weight property}\label{subsec:ZarWeight}

When dealing with the discriminant of $n-1$ homogeneous polynomials in $n$ variables, we will need a property of homogeneity for the resultant that is due to Mertens \cite{Mertens1886} and that has been generalized by Zariski about fifty years later \cite[Theorem 6]{Zar37}. For the convenience of the reader, we provide a proof of this result. 

Suppose given $n$ integers $\mu_1,\ldots,\mu_n$ such that for all $i=1,\ldots,n$ we have $0\leq \mu_i\leq d_i$ and set $f_i=X_n^{\mu_i}g_i+h_i$ where all the monomials having a nonzero coefficient in the polynomial $h_i$ is not divisible by $X_n^{\mu_i}$, i.e.~is such that $\alpha_n<\mu_i$. Now, define the weight of each coefficient $U_{i,\alpha}$, $i=1,\ldots,n$, $|\alpha|=d_i$ by
\begin{equation}\label{eq:weight}
	\mathrm{weight}(U_{i,\alpha}):=
	\begin{cases}
	 0 & \textrm{if} \ \alpha_n<\mu_i \\
	 \alpha_n-\mu_i & \textrm{if} \ \alpha_n\geq \mu_i
	\end{cases}  	
\end{equation}
(we will refer to this grading as the Zariski grading) and set
$$\Res(f_1,\ldots,f_n)=H(f_1,\ldots,f_n)+N(f_1,\ldots,f_n) \in {}_kA$$
where $H$ is the homogeneous part of minimum degree of the resultant, using the above weights definition.
  
\begin{prop} With the above notation, there exists an element $$H_1(f_1,\ldots,f_n) \in {}_kA$$ which 
	is of degree zero and that satisfies
	$$H(f_1,\ldots,f_n)=\Res(g_1,\ldots,g_n)H_1(f_1,\ldots,f_n) \in {}_kA.$$
In particular, the degree of $H$ is equal to $\prod_{i=1}^n (d_i-\mu_i)$.	
\end{prop}

Here is an immediate corollary that is the form under which we will use this property later on. 

\begin{cor}\label{cor:mertenshom}
For all $i=1,\ldots,n$, define the polynomials $h_i$ and rename some coefficients $U_{i,\alpha}$ of $f_i$ so that
$f_i=X_n^{d_i-1}(\sum_{j=1}^n V_{i,j}X_j)+h_i$. Then, we have
$$\Res(f_1,\ldots,f_n)-\det((V_{i,j})_{i,j=1,\ldots,n})H_1 \in (V_{1,n},\ldots,V_{n,n})^2 \subset {}_kA.$$
\end{cor}

\begin{proof} Let $\phi \in \TF_\mm(f_1,\ldots,f_n)\cap A$, so that there exists an integer $N$ such that $X_n^N\phi \in (f_1,\ldots,f_n)$, and define $\phi_0\in A$ as the homogeneous part of minimum degree of $\phi$ with respect to the weights given in \eqref{eq:weight}. We begin by showing that $\phi_0 \in \TF_\mm(g_1,\ldots,g_n)\cap A$. 
	
	In addition of the weights \eqref{eq:weight}, we set $\mathrm{weight}(X_i)=1$ for all $i=1,\ldots,n-1$ and $\mathrm{weight}(X_n)=0$. In this way, for all $i=1,\ldots,n$ the terms in the decomposition $f_i=X_n^{\mu_i}g_i+h_i$ are such that $X_n^{\mu_i}g_i$ is homogeneous of degree $d_i-\mu_i$ whereas $h_i$ contains monomials that are homogeneous of degree strictly bigger than $d_i-\mu_i$. To emphasize this property, introduce a new indeterminate $t$ and consider the linear transformation 
\begin{equation*}
	\begin{cases}
	X_i \mapsto tX_i, \ i=1,\ldots,n-1 \\
	X_n\mapsto X_n \\
	U_{i,\alpha} \mapsto t^{\mathrm{weight}(U_i,\alpha)}U_{i,\alpha}, \ i=1,\ldots,n, \ |\alpha|=d_i.
	\end{cases}
\end{equation*}
Denoting by $\nu$ the degree of $\phi_0$ and applying the above transformation, we deduce that
\begin{multline}\label{eq:RHN1}
X_n^N t^\nu(\phi_0+t\omega_0) \in \\ (t^{d_1-\mu_1}(X_n^{\mu_1}g_1+t\omega_1), t^{d_2-\mu_2}(X_n^{\mu_2}g_2+t\omega_2),\ldots,t^{d_n-\mu_n}(X_n^{\mu_n} g_n+t\omega_n))	
\end{multline}
where $\omega_i \in A[X_1,\ldots,X_n,t]$ for all $i=0,\ldots,n$. Having in mind to use the characterization \eqref{perron} of inertia forms, for all $i=1,\ldots,n$ we set $g_i=\eta_iX_n^{d_i-\mu_i}+\varphi_i$, $\tilde{g_i}=g_i(X_1,\ldots,X_{n-1},1)$ and $\tilde{\varphi_i}=\varphi_i(X_1,\ldots,X_{n-1},1)$. Now, the specialization of $X_n$ to $1$
in \eqref{eq:RHN1} yields
\begin{equation*}
t^\nu(\phi_0+t\omega_0) \in (\tilde{g_1}+t\omega_1, \tilde{g_2}+t\omega_2,\ldots,\tilde{g_n}+t\omega_n)	
\end{equation*}
and then the specializations of $\eta_i$ to $-\tilde{\varphi_i}-t\omega_i$ for all $i=1,\ldots,n$ give
\begin{equation}\label{eq:RHN2}
t^\nu(\phi_0(-\tilde{\varphi_1}-t\omega_1,\ldots,-\tilde{\varphi_n}-t\omega_n)+t\omega_0(-\tilde{\varphi_1}-t\omega_1,\ldots,-\tilde{\varphi_n}-t\omega_n))=0	
\end{equation}
in $A[X_1,\ldots,X_n,t]$, where the quoted arguments of $\phi_0$ and $\phi_1$ are those corresponding to the coefficients $\eta_1,\ldots,\eta_n$ respectively. But since $t$ is a nonzero divisor, we can simplify \eqref{eq:RHN2} by $t^\nu$. Then, by specializing $t$ to $0$ we deduce that $\phi_0(-\tilde{\varphi_1},\ldots,-\tilde{\varphi_n})=0$ and hence that $\phi_0\in \TF_\mm(g_1,\ldots,g_n)$.  

Now, applying the above property to $\phi=\Res(f_1,\ldots,f_n)$ we deduce that there exists $H_1\in A$ such that $H=\Res(g_1,\ldots,g_n)H_1$. However, to conclude the proof it remains to show that $H_1$ is of degree zero, or equivalently that $H$ and $\Res(g_1,\ldots,g_n)$ have the same degree with respect to the weights \eqref{eq:weight}. Notice that we already know  that $\Res(g_1,\ldots,g_n)$ has degree $\prod_{i=1}^n(d_i-\mu_i)$ by the property \cite[\S 5.13.2]{J91} and hence, the degree of $H$ is greater or equal to $\prod_{i=1}^n(d_i-\mu_i)$. In order to show that it is actually an equality, we consider the following specialization
$$\left\{
\begin{array}{rcl}
f_1 & = & X_1^{d_1-\mu_1}X_n^{\mu_1} \\
f_2 & = & X_1^{d_2}+X_2^{d_2-\mu_2}X_n^{\mu_2} \\
f_3 & = & X_2^{d_3}+X_3^{d_3-\mu_3}X_n^{\mu_3} \\
 & \vdots & \\
f_{n-1} & = & X_{n-2}^{d_{n-1}}+X_{n-1}^{d_{n-1}-\mu_{n-1}}X_n^{\mu_{n-1}} \\
f_{n} & = & X_{n-1}^{d_{n-1}}+t^{d_n-\mu_n}X_{n}^{d_{n}} 
\end{array}\right.
$$
where, for all $i=1,\ldots,n$, the coefficient $U_{i,\alpha}$ of each monomial $X_1^{\alpha_1}\ldots X_n^{\alpha_n}$, $ |\alpha|=d_i$, of $f_i$ that appears in this specialization has been also specialized to $t^{\mathrm{weight}(U_i,\alpha)}$. Let us compute the resultant of $f_1,\ldots,f_n$. Applying the multiplicativity property of resultants \cite[\S 5.7]{J91}, we get
\begin{align*}
	\Res(f_1,\ldots,f_n) & = \Res(X_1^{d_1-\mu_1},f_2,\ldots,f_n)\Res(X_n^{\mu_1},f_2,\ldots,f_n) \\
	&= \Res(X_1,f_2,\ldots,f_n)^{d_1-\mu_1}\Res(X_n,X_1,X_2,\ldots,X_{n-1})^{\mu_1d_2d_3\ldots d_{n-1}}\\
	&= (-1)^{(n-1)\mu_1 d_2 d_3\ldots d_{n-1}}\Res(X_1,f_2,\ldots,f_n)^{d_1-\mu_1},
\end{align*}
then
\begin{align*}
	\lefteqn{\Res(X_1,f_2,\ldots,f_n)} \\
	&= \Res(X_1,X_2^{d_2-\mu_2},f_3,\ldots,f_n) \Res(X_1,X_n^{\mu_2},f_3,\ldots,f_n) \\
	&= \Res(X_1,X_2,f_3,\ldots,f_n)^{d_2-\mu_2}\Res(X_1,X_n,X_2,X_3,\ldots,X_{n-1})^{\mu_2d_3\ldots d_{n-1}}\\
	&= (-1)^{(n-2)\mu_2d_3\ldots d_{n-1}}\Res(X_1,X_2,f_3,\ldots,f_n)^{d_2-\mu_2}
\end{align*}
and continuing this way we arrive at the equality
$$ \Res(f_1,\ldots,f_n) = \pm \Res(X_1,\ldots,X_{n-1},f_n)^{(d_1-\mu_1)\ldots (d_{n-1}-\mu_{n-1})}.$$
But since $f_n$ is specialized to $X_{n-1}^{d_{n-1}}+t^{d_n-\mu_n}X_{n}^{d_{n}}$, we deduce that
\begin{align*}
\lefteqn{\Res(f_1,\ldots,f_n)} \\
&= \pm \Res(X_1,\ldots,X_{n-1},t^{d_n-\mu_n}X_{n}^{d_{n}})^{(d_1-\mu_1)\ldots (d_{n-1}-\mu_{n-1})}	\\
&= \pm t^{(d_1-\mu_1)\ldots (d_{n-1}-\mu_{n-1})(d_n-\mu_n)} \Res(X_1,\ldots,X_{n-1},X_{n})^{(d_1-\mu_1)\ldots (d_{n-1}-\mu_{n-1})d_n} \\
&= \pm t^{\prod_{i=1}^n(d_i-\mu_i)}.
\end{align*}
Therefore, for this particular specialization, we get that $\Res(f_1,\ldots,f_n)$ is of degree $\prod_{i=1}^n(d_i-\mu_i)$, and hence that, in the generic context, the degree of $H$ can not be greater than $\prod_{i=1}^n(d_i-\mu_i)$ which concludes the proof.
\end{proof}

Mention that from an historical point of view, the above result is the beginning of the theory of the \emph{reduced resultant}. Indeed, Zariski proved \cite{Zar37} that the factor $H_1$ is a generator of a principal ideal whose geometric interpretation is that the polynomials $h_1,\ldots,h_n$ have a common root in addition of the root $X_1=\ldots=X_{n-1}=0$ that they already have in common. It is called the \emph{reduced resultant}. We refer the interested reader to \cite{Zar37} and \cite{PhDthesisRed} for more details.

\subsection{The Dedekind-Mertens Lemma} We end this section of preliminaries by recalling the Dedekind-Mertens Lemma and give an important corollary that we will use several times in this text (sometimes even implicitly).

Let $A$ be a commutative ring and $\underline{X}:=(X_1,\ldots,X_n)$ be 
a sequence of $n\geq 1$
indeterminates. Given a $A$-module $M$ and an element
$$m=\sum_{\alpha}c_\alpha X^\alpha \in M[\underline{X}]:=M[X_1,\ldots,X_n]$$
we define the \emph{support} of $m$ as
$$\supp(m)=\{ \alpha \in \NN^n : c_\alpha \neq 0\}$$
and the \emph{length} of $m$, denoted $l(m)$, as the cardinal of
$\supp(m)$. Observe that $l(m)=0$ if and only if  $m=0$.
Moreover, for any subring $R$ of $A$, we define the 
  $R$-content of $m$ as the $R$-submodule of $M$:
$$\C_R(m):=\sum_{\alpha \in \supp(m)}c_\alpha R.$$

\begin{lem}[Dedekind-Mertens]  Let $M$ be a $A$-module, $f$ be a
  polynomial in $A[\underline{X}]$ and $m$ a polynomial in
  $M[\underline{X}]$. Then, for all subring $R$ of $A$ we have
  $$\C_R(f)^{l(m)}\C_R(m)=\C_R(f)^{l(m)-1}\C_R(fm)$$
  where we set, by convention, $\C_R(f)^{-1}=R$.
\end{lem}

\begin{cor} Let $M$ be a $A$-module and $f\in
  A[\underline{X}]$ a polynomial. Then, the following are equivalent:
  \begin{itemize}
  \item[\rm (i)] The polynomial $f$ is a nonzero divisor in the
    $A[\underline{X}]$-module $M[\underline{X}]$.
    \item[\rm (ii)] The ideal $\C_A(f)$ does not divide zero in $M$
      (there does not exists $m \in M$ such that $m\neq 0$ and $\C_A(f)m=0$). 
  \end{itemize}
\end{cor}
\begin{proof} Assume that (i) holds and that  there exists $m \in M$ such that
  $\C_A(f)m=0$. Then $(m.1_{A[\underline{X}]})f=0$ in
  $M[\underline{X}]$ and hence $m=0$, which proves that (i)
  implies (ii).

  Now, assume that (ii) holds and that there exists $m \in
  M[\underline{X}]$ such that $mf=0$. Then, by the Dedekind-Mertens
  lemma, we deduce that $\C_A(f)^{l(m)}\C_A(m)=0$ and from (ii) that
  $\C_A(m)=0$. It follows that $m=0$ and the corollary is proved. 
\end{proof} 

Finally, recall that a polynomial $f \in A[\underline{X}]$ is said to be \emph{primitive} if $C_A(f)=A$.

\section{The discriminant of a finite set of points}\label{dim0}

\subsection{Definition and first properties}\label{sec:def}

In this section, we give the definition of the discriminant of $n-1$
homogeneous polynomials in $n$ variables.  We begin section with some properties on Jacobian determinants. 
Then, we provide computational rules for handling this discriminant and we show that its definition have the expected geometric property:
its vanishing corresponds to the detection of a singular locus. 

\medskip

Hereafter, we suppose given $n-1$, with $n\geq 2$, homogeneous polynomials
$f_1,\ldots,f_{n-1}$ of positive degree $d_1,\ldots,d_{n-1}$,
respectively, 
$$f_i(X_1,\ldots,X_n)=\sum_{|\alpha|=d_i}U_{i,\alpha}X^\alpha, \ \ 
i=1,\ldots,n-1.$$
We denote by $k$ an arbitrary commutative ring and set ${}_k A:=k[U_{i,\alpha}]$
the universal coefficient ring over $k$. Thus, $f_i \in
{}_kA[X_1,\ldots,X_n]_{d_i}$ for all $i=1,\ldots,n-1$.

\subsubsection{Jacobian determinants} 
For all $i=1,\ldots,n$, consider the Jacobian determinant
\begin{multline}\label{Ji}
J_i(f_1,\ldots,f_{n-1}):= \\
(-1)^{n-i}
\left|
\begin{array}{cccccc}
\partial_{X_1}f_1 & \cdots& \partial_{X_{i-1}}f_1 & \partial_{X_{i+1}}f_1 & \cdots & \partial_{X_{n}}f_1 \\
\partial_{X_1}f_2 & \cdots& \partial_{X_{i-1}}f_2 & \partial_{X_{i+1}}f_2 & \cdots & \partial_{X_{n}}f_2 \\
\vdots & & \vdots & \vdots & & \vdots \\
\partial_{X_1}f_{n-1} & \cdots& \partial_{X_{i-1}}f_{n-1} & \partial_{X_{i+1}}f_{n-1} & \cdots & \partial_{X_{n}}f_{n-1}
\end{array}\right|
\end{multline}
that is obviously a homogeneous polynomial in the
variables $X_1,\ldots,X_n$ of degree 
$\deg(J_i)=\sum_{j=1}^{n-1}(d_j-1)$. Notice that this degree is independent on $i\in \{1,\ldots,n\}$. 
\begin{lem}\label{Jirred}
  For all integer $i\in \{1,\ldots,n\}$, we have:
\begin{itemize}
	\item[i)] the Jacobian determinant
	  $J_i:=J_i(f_1,\ldots,f_{n-1})$ is irreducible in
	  the polynomial ring ${}_\ZZ A[X_1,\ldots,X_n]$,
	\item[ii)] the polynomial $J_i(X_1,\ldots,X_{i-1},1,X_{i+1},\ldots,X_n)$ is  primitive, hence a non\-zero divisor, in ${}_k A[X_1,\ldots,X_{i-1},X_{i+1},\ldots,X_n]$,
	  \item[iii)] if $k$ is a domain then $J_i(X_1,\ldots,X_{i-1},1,X_{i+1},\ldots,X_n)$ is prime in the polynomial ring ${}_k A[X_1,\ldots,X_{i-1},X_{i+1},\ldots,X_n]$.
	
\end{itemize}
\end{lem}
\begin{proof} It is sufficient to prove this result for
  $J_n:=J_n(f_1,\ldots,f_{n-1})$. Observe first that $J_n$ is homogeneous of degree 1 in each
  set of variables $(U_{i,\alpha})_{|\alpha|=d_i}$ with 
  $i\in \{1,\ldots,n-1\}$. Now, consider the specialization $\rho$ that sends each polynomial $f_i$,
  $i=1,\ldots,n-1$, to
  $$f_i \mapsto U_{i,1}X_1X_n^{d_i-1}+ U_{i,2}X_2X_n^{d_i-1} + \cdots
  +U_{i,n-1}X_{n-1}X_n^{d_i-1}.$$
We have
\begin{equation}\label{eq:speJn}
	\rho(J_n)=X_n^{\deg(J_n)} \left| \begin{array} {ccc}
	      U_{1,1} & \cdots & U_{1,n-1} \\
	      \vdots & & \vdots \\
	      U_{n-1,1} & \cdots & U_{n-1,n-1} \end{array}
	  \right|.
\end{equation}

Let us assume first that $k$ is a UFD. Then the determinant in \eqref{eq:speJn} is known to be irreducible in $k[U_{i,j} | i,j=1,\ldots,n-1]$. Since $\rho$ preserves the homogeneity with respect to each set of variables $(U_{i,\alpha})_{|\alpha|=d_i}$, $i\in \{1,\ldots,n-1\}$, we deduce that $\mathrm{iii)}$ holds (under the assumption that $k$ is a UFD).
 
Moreover, assuming that $k=\ZZ$, \eqref{eq:speJn} implies that $J_n$ decomposes as a product $P.Q$ 
where $P$ is irreducible and depends on the
$U_{i,\alpha}$'s, $Q$ does not depend on the $U_{i,\alpha}$'s. 
  Moreover $Q \in \ZZ[X_1,\ldots,X_n]$ so that it must divide $X_n^{\deg(J_n)}$. Now, if we specialize each polynomial $f_i$ to $X_i^{d_i}$, then  $J_n$
  specializes to $\prod_{i=1}^{n-1}d_iX_i^{d_i-1}$. It follows that  $Q$ must
  also divide this latter polynomial and we deduce that $Q$ is equal to 
  $\pm1 \in \mathbb{Z}$. This proves i).

\medskip

Now, we prove that iii) holds under the weaker assumption that $k$ is a domain. For that purpose, consider the quotient ring 
$${}_k Q:=\quotient{{}_kA[X_1,\ldots,X_{n-1}]}{(J_n(X_1,\ldots,X_{n-1},1))}$$ 
and set $Q:={}_\ZZ Q$ for simplicity in the notation. We have already proved that ${}_k Q$ is a domain as soon as $k$ is a UFD. In particular $Q$ is a domain. Since $Q$ contains $\ZZ$, $Q$ is a torsion-free abelian group and hence it is flat. It follows that the canonical inclusion of rings $k\subset  K:=\Frac(k)$ gives rise to an injective map	
$${}_k Q = k \otimes_\ZZ Q \rightarrow K\otimes_\ZZ Q = {}_{K}Q.$$
But we have proved that ${}_{K}Q$ is a domain, so we deduce that ${}_k Q$ is also a domain and hence that $J_n(X_1,\ldots,X_{n-1},1)$ is a prime element in ${}_kA[X_1,\ldots,X_{n-1}]$ as claimed.

\medskip

Finally, from i) we deduce that $J_n(X_1,\ldots,X_{n-1},1)$ is a primitive polynomial in ${}_\ZZ A[X_1,\ldots,X_{n-1}]$. It follows that it is also primitive over any commutative ring $k$, hence a nonzero-divisor by the Dedekind-Mertens Lemma.
\end{proof}

\begin{rem} Notice that the Jacobian determinant $J_i \in {}_k A[X_1,\ldots,X_n]$ is not irreducible in general. Indeed, take for instance $n=2$ and set $f_1(X_1,X_2)=\sum_{i=0}^dU_iX_1^iX_2^{d-i}$. Then
	$$J_2=\frac{\partial f_1}{\partial X}=dU_{d}X_1^{d-1}+(d-1)U_{d-1}X_1^{d-2}X_2+\cdots+U_1X_2^{d-1}$$ 
and hence $X_2$ divides $J_2$ as soon as $d=0$ in $k$.	

Similarly, the Jacobian determinant of $n$ homogeneous polynomials in $n$ homogeneous variables is not irreducible in general. For instance, the Jacobian of the polynomials
$$ f_1(X_1,X_2)=aX_1^2+bX_1X_2+cX_2^2, \ \ f_2(X_1,X_2)=uX_1^2+vX_1X_2+wX_2^2$$
is equal to the determinant
$$\left|\begin{array}{cc}
bX_2 & b X_1 \\
vX_2 & vX_1	
\end{array}\right|$$ which is identically zero in $k[a,b,c,u,v,w][X_1,X_2]$ as soon as $2=0$ in $k$.
\end{rem}

Now, introduce the generic homogeneous polynomial of degree $d\geq 1$ in the set of variables 
$X_1,\ldots,X_n$
$$F(X_1,\ldots,X_n):=\sum_{|\alpha|=d} U_\alpha X^\alpha$$
and set 
${}_kA':={}_k A[U_\alpha : |\alpha|=d].$ The Jacobian
determinant
\begin{equation}\label{J}
J(f_1,\ldots,f_{n-1},F):=
\begin{vmatrix}
	\partial_{X_1}f_1 & \partial_{X_{2}}f_1 &  \cdots & \partial_{X_{n}}f_1 \\
	\partial_{X_1}f_2 & \partial_{X_{2}}f_2 &  \cdots & \partial_{X_{n}}f_2 \\
	\vdots & \vdots & & \vdots \\
		\partial_{X_1}f_{n-1} & \partial_{X_{2}}f_{n-1} &  \cdots & \partial_{X_{n}}f_{n-1} \\
			\partial_{X_1}F & \partial_{X_{2}}F &  \cdots & \partial_{X_{n}}F \\
\end{vmatrix}
\end{equation}
is a homogeneous polynomial of degree $\deg(J)=(d-1)+\sum_{i=1}^{n-1}(d_i-1)$ in the set of variables
$X_1,\ldots,X_n$. 
By developing the determinant \eqref{J} with 
respect to its last row, we obtain the equality 
$$J(f_1,\ldots,f_{n-1},F)=\sum_{i=1}^{n} \frac{\partial F}{\partial
  X_i} J_i(f_1,\ldots,f_{n-1})$$
that holds in the ring ${}_kA'[X_1,\ldots,X_n]$.

\begin{lem}\label{lem1} With the above notation, we have:
  \begin{itemize}
  \item[\rm i)] for all integer $i \in \{1,\ldots,n\}$
\begin{multline*}
    X_iJ(f_1,\ldots,f_{n-1},F) - dFJ_i(f_1,\ldots,f_{n-1}) \\ \in
    (d_1f_1,\ldots,d_{n-1}f_{n-1})\subset {}_kA'[X_1,\ldots,X_n].	
\end{multline*}
  \item[\rm ii)] for all couple $(i,j)$ of distinct integers in
    $\{1,\ldots,n \}$
\begin{multline*}
    X_iJ_j(f_1,\ldots,f_{n-1})-X_jJ_i(f_1,\ldots,f_{n-1}) \\ \in
    (d_1f_1,\ldots,d_{n-1}f_{n-1})\subset {}_kA'[X_1,\ldots,X_n].	
\end{multline*}
    \end{itemize}
\end{lem}
\begin{proof} These properties follow straightforwardly by using 
   Euler's identities $$\sum_{j=1}^n X_j\frac{\partial
    f_i}{\partial X_j}=d_if_i, \ \ i=1,\ldots,n$$ in the determinants \eqref{Ji} and \eqref{J}.
\end{proof}

\subsubsection{Definition of the discriminant}\label{defsubsec} The definition of the discriminant of the
homogeneous polynomials $f_1,\ldots,f_{n-1}$ is based on the

\begin{prop}\label{lemdiv} With the previous notation, 
 $$d^{d_1\ldots d_{n-1}}\Res(f_1,\ldots,f_{n-1},F) \text{ divides }
 \Res(f_1,\ldots,f_{n-1},J(f_1,\ldots,f_{n-1},F))$$
 in ${}_k A'$.
  Moreover,  for all $i \in \{1,\ldots,n\}$, 
  we have the equality 
 \begin{multline*}
\Res(f_1,\ldots,f_{n-1},J(f_1,\ldots,f_{n-1},F)) \Res(f_1,\ldots,f_{n-1},X_i) = \\
         d^{d_1\ldots d_{n-1}}\Res(f_1,\ldots,f_{n-1},F)\Res(f_1,\ldots,f_{n-1},J_i)
\end{multline*}
\end{prop}

\begin{proof} By specialization, it is sufficient to prove this proposition over the integers, that is to say by assuming that $k=\ZZ$. 
	
	By Lemma \ref{lem1} we know that
  $X_iJ(f_1,\ldots,f_{n-1},F)$ and $dJ_i(f_1,\ldots,f_{n-1})F$ are
  homogeneous polynomials of the same degree in the variables
  $X_1,\ldots,X_n$ that are equal modulo the ideal
  $(f_1,\ldots,f_{n-1})$. It follows that, in ${}_k A'$,
  $$\Res(f_1,\ldots,f_{n-1},X_iJ(f_1,\ldots,f_{n-1},F))=
  \Res(f_1,\ldots,f_{n-1},dJ_i(f_1,\ldots,f_{n-1})F).$$ 
  The result then follows from standard properties of 
  resultants \cite[\S5]{J91}.
\end{proof}

We are now ready  to state the definition of the discriminant of the
polynomials $f_1,\ldots,f_{n-1}$.

\begin{defn}\label{defdisc} If $\sum_{i=1}^{n-1}(d_i-1) \geq 1$ then the
  \emph{discriminant}  of the polynomials
  $f_1,\ldots,f_{n-1}$, denoted $\Disc(f_1,\ldots,f_{n-1})$, is
  defined as the \emph{unique} non-zero
  element in ${}_\ZZ A$ such that 
\begin{align}\label{defeq}
  \Disc(f_1,\ldots,f_{n-1})\Res(f_1,\ldots,f_{n-1},X_i)=\Res(f_1,\ldots,f_{n-1},J_i) 
\end{align}
for all $i\in \{1,\ldots,n\}$.  If $\sum_{i=1}^{n-1}(d_i-1)=0$, or
equivalently if $d_1=\cdots=d_{n-1}=1$, we set
$\Disc(f_1,\ldots,f_{n-1})=1 \in {}_\ZZ A$.

Let $R$ be a commutative ring and suppose given $n-1$ homogeneous polynomials 
    $$g_i=\sum_{|\alpha|=d_i}u_{i,\alpha}X^\alpha \in
    R[X_1,\ldots,X_n], \ \ i=1,\ldots,n-1,$$
of degree $d_1,\ldots,d_{n-1}$ respectively. As in \S\ref{not}, denote by 	$\theta$ the ring morphism $\theta:{{}_\ZZ A}\rightarrow R:
    U_{j,\alpha} \mapsto u_{j,\alpha}$ corresponding to the
    \emph{specialization} of the polynomial $f_i$ to the polynomial
    $g_i$ for all $i=1,\ldots,n-1$. Then, the \emph{discriminant} of $g_1,\ldots,g_{n-1}$ is defined as 
$$\Disc(g_1,\ldots,g_{n-1}):=  \theta(\Disc(f_1,\ldots,f_{n-1})) \in R.$$
\end{defn}

\begin{rem}\label{fbarre} We recall that, for all integer  $i\in \{1,\ldots,n\}$,  
$$\Res(f_1,\ldots,f_{n-1},X_i)=\Res(f_1^{(i)},\ldots,f_{n-1}^{(i)}) \in {}_kA$$
where $f_1^{(i)},\ldots,f_{n-1}^{(i)}$ are the polynomials obtained
from $f_1,\ldots,f_{n-1}$, respectively, by substituting $X_i$ for 0
(see \cite[Lemma 4.8.9]{J91}).  It is a nonzero divisor in ${}_kA$ (see Proposition \ref{prop:resultant}).
\end{rem}

A direct consequence of the definition of the
discriminant is the following. From Proposition \ref{lemdiv},  it follows
immediately that, in ${}_k A'$,
\begin{multline}\label{J(F)}
\Res (f_1,\ldots,f_{n-1},J(f_1,\ldots,f_{n-1},F) )
   = \\ d^{d_1\ldots d_{n-1}} \Disc(f_1,\ldots,f_{n-1}) \Res (f_1,\ldots,f_{n-1},F).  
 \end{multline}
 Moreover, if $\deg(F)=d=1$ then $J(f_1,\ldots,f_{n-1},F)$ can be replaced by the polynomial
 $F(J_1,\ldots,J_n)$ in this formula and we get 
\begin{multline*}
 \Res(f_1,\ldots,f_{n-1},U_1J_1+\cdots+U_nJ_n)= \\
 \Disc(f_1,\ldots,f_{n-1}) \Res(f_1,\ldots,f_{n-1},U_1X_1+\cdots+U_nX_n).	
\end{multline*}
 More generally, we have the 
\begin{prop}\label{F-J} For all $d\geq 1$ the following equality holds in ${}_k A'$:
 $$\Res(f_1,\ldots,f_{n-1},F(J_1,\ldots,J_n)) 
  =\Disc(f_1,\ldots,f_{n-1})^d\, \Res(f_1,\ldots,f_{n-1},F).$$
  \end{prop}
  \begin{proof}
  Indeed, Lemma \ref{lem1} shows that both polynomial
  $X_i^dF(J_1,\ldots,J_n)$ and polynomial $J_i^dF(X_1,\ldots,X_n)$ are
  homogeneous of the same degree in the variables $X_1,\ldots,X_n$ and
  equal up to an element in the ideal $(f_1,\ldots,f_{n-1}).$ It
  follows that
  $$\Res(f_1,\ldots,f_{n-1},X_i^dF(J_1,\ldots,J_n))=\Res(f_1,\ldots,f_{n-1},J_i^dF(X_1,\ldots,X_n))$$
  and the claimed formula is obtained using the multiplicativity
  property of the resultants \cite[\S5.7]{J91}.
  \end{proof}

An important property of the \emph{generic} discriminant is that, similarly to the generic resultant, it is universally a nonzero divisor.

\begin{prop} The discriminant $\Disc(f_1,\ldots,f_{n-1}) \in {}_kA$ is a nonzero divisor.
\end{prop} 
\begin{proof} By specializing each polynomial $f_i$ to a product of generic linear form, the discriminant specialize to a primitive polynomial (the ideal generated by its coefficients is equal to $k$) by Corollary \ref{det2}.  It follows that $\Disc(f_1,\ldots,f_{n-1}) \in {}_kA$ itself a primitive polynomial in ${}_kA$. Therefore, the claimed result follows by  Dedekind-Mertens Lemma. 
\end{proof}

\subsubsection{The degree of the discriminant}

 The discriminant is multi-homogeneous, as inheritance from the resultant:
 it is homogeneous with respect to the coefficients of \emph{each}
 polynomial $f_1,\ldots,f_{n-1}$.  The following result gives the
 precise multi-degree of the discriminant.

 \begin{prop}\label{degree-prop} With the notation of \S \ref{defsubsec}, $\Disc(f_1,\ldots,f_{n-1})$ is a
   homogeneous polynomial in ${}_k A$ of total degree
   $$(n-1)\prod_{i=1}^{n-1} d_i +
   (d_1+\cdots+d_{n-1}-n)\,\left(\sum_{i=1}^{n-1}\frac{d_1 \cdots
     d_{n-1}}{d_i}\right).$$
   Moreover, it is homogeneous with respect to the
   coefficients of each polynomial $f_i$, $i\in \{1,\ldots,n-1\}$, of
   degree
\begin{equation}\label{partial-degree}
  \frac{d_1 \cdots d_{n-1}}{d_i} \left((d_i-1)+\sum_{j=1}^{n-1}(d_j-1)\right).\end{equation}
  \end{prop}
  \begin{proof} 
    Let us fix an integer $i\in \{1,\ldots,n-1\}$ and introduce a new
    variable $t$. We know that the Jacobian polynomial $J_n$ is homogeneous
    in the variables $X_1,\ldots,X_n$ of degree
    $\sum_{i=1}^{n-1}(d_i-1)$. It also obviously  satisfies
    \begin{equation}\label{J-f_i}
      J_n(f_1,\ldots,tf_i,\ldots,f_{n-1})=tJ_n(f_1,\ldots,f_i,\ldots,f_{n-1}).
      \end{equation}
      Therefore, by multi-homogeneity property of the resultant
      \cite[2.3(ii)]{J91}, we deduce that
\begin{align}
  &
  \Res(f_1,\ldots,tf_i,\ldots,f_{n-1},J_n(f_1,\ldots,tf_i,\ldots,f_{n-1}))
  \notag \\
  & \qquad \qquad = t^{ \frac{d_1 \cdots d_{n-1}}{d_i} \sum_{
      j=1}^{n-1}(d_j-1)}
  \Res(f_1,\ldots,f_{n-1},J_n(f_1,\ldots,tf_i,\ldots,f_{n-1}))
  \notag  \\
  & \qquad \qquad = t^{ \frac{d_1 \cdots d_{n-1}}{d_i} \sum_{
      j=1}^{n-1}(d_j-1)}
  \Res(f_1,\ldots,f_{n-1},tJ_n(f_1,\ldots,f_n)) \notag   \\
  & \qquad \qquad = t^{ \frac{d_1 \cdots d_{n-1}}{d_i} \sum_{
      j=1}^{n-1}(d_j-1)+\prod_{i=1}^{n-1}d_i}
  \Res(f_1,\ldots,f_{n-1},J_n(f_1,\ldots,f_n)) \notag
\end{align}
and
\begin{align}
  & \Res(f_1,\ldots,tf_i,\ldots,f_{n-1},X_n)= t^{\frac{d_1 \cdots
      d_{n-1}}{d_i}}\Res(f_1,\ldots,f_i,\ldots,f_{n-1},X_n) .  \notag
\end{align}
From Definition \ref{defdisc} of the discriminant, it follows that
$$\Disc(f_1,\ldots,tf_i,\ldots,f_{n-1})=t^{\frac{d_1 \cdots
    d_{n-1}}{d_i}
  \left((d_i-1)+\sum_{j=1}^{n-1}(d_j-1)\right)}\Disc(f_1,\ldots,f_{n-1})$$
as claimed. The total degree is obtained by adding all these partial
degrees.
\end{proof}

\begin{rem} Observe that the integers \eqref{partial-degree} are
    always even. This is expected because, as we will see later on, in characteristic 2 it turns out that the discriminant is the square of an irreducible polynomial.    
\end{rem}

\subsubsection{The classical case $n=2$}\label{n=2} Let us show that our definition
of the discriminant coincides with the classical case $n=2$.

Let $f$ be a polynomial homogeneous in the variable $X,Y$ of degree $d
\geq 2$
$$f:=V_{d}X^d+V_{d-1}X^{d-1}Y+V_{d-2}X^{d-2}Y^2+\cdots+V_1X^1Y^{d-1}+V_0Y^d.$$
According to \eqref{defeq} we have
$$\Res(f,J_2(f))= \Disc(f)\Res(f,Y) \in k[V_0,\ldots,V_d].$$
But it is easy to see that $\Res(f,Y)=V_d$ and that
$J_2(f)=\frac{\partial f}{\partial X}$. Therefore we recover the usual definition 
$V_d \Disc(f)=\Res(f,\frac{\partial f}{\partial X})$. Moreover, from Proposition \ref{degree-prop} we also obtain that it is a homogeneous
polynomial in the coefficients of $f$, i.e.~$V_0,\ldots,V_d$, of degree $2d-2$.

A lot of properties are known
for this discriminant (see e.g.~\cite{ApJo} or \cite[chapter 12.B]{GKZ}) and we will generalize most of them to the case of $n-1$ homogeneous
polynomials in $n$ variables  in the sequel.

\subsubsection{Vanishing of the discriminant} 
Assume that
$k$ is an algebraically closed field and let $f_1,\ldots,f_{n-1}$ be
$n-1$ homogeneous polynomials in $k[X_1,\ldots,X_n]$ such that
the variety $Y:=V(f_1,\ldots,f_{n-1})\subset \mathbb{P}^{n-1}_k$ is \emph{finite}. 
The following
proposition says that the discriminant of $f_1,\ldots,f_{n-1}$
vanishes if and only if the polynomial system $f_1=\cdots=f_{n-1}=0$
has a multiple root.

\begin{prop}\label{poi} With the above notation, $\Disc(f_1,\ldots,f_{n-1})=0$ if
  and only if there exists a point $\xi \in Y$ such that $Y$ is
  singular at $\xi$.
\end{prop}
\begin{proof} First, without loss of generality we can assume
  $Y\cap V(X_n)=\emptyset$, so that 
  $\Res(f_1,\ldots,f_{n-1},X_n)$ is not equal to zero in $k$ and
  \begin{align}
    \Disc(f_1,\ldots,f_{n-1})=\frac{\Res(f_1,\ldots,f_{n-1},J_n)}
    {\Res(f_1,\ldots,f_{n-1},X_n)} \in k. \notag
  \end{align}
 By the Poisson's formula \cite[Proposition
  2.7]{J91}, we have the equality
   \begin{align}
     \frac{\Res(f_1,\ldots,f_{n-1},J_n)}{\Res(f_1,\ldots,f_{n-1},X_n)^{\deg(J_n)}}=
     \prod_{\xi \in Y}J_n(\xi)^{\mu_\xi} \notag
  \end{align}
  where $\mu_\xi$ denotes the multiplicity of $\xi \in Y$. It follows
  that $\Disc(f_1,\ldots,f_{n-1})=0$ if and only if there exists a point
  $\xi \in Y$ such that $J_n(\xi)=0$. 

Now, a classical necessary and sufficient condition for $\xi \in Y$ to be a
  singular point of $Y$ is that $J_i(\xi)=0$ for all $i=1,\ldots,n$ (see e.g.~\cite[Chapter I, Theorem 5.1]{H}). But from Lemma \ref{lem1}, ii), we have $J_i(\xi)=\xi_iJ_n(\xi)$ for all $\xi \in Y$ and all $i=1,\ldots,n-1$, where $\xi=(\xi_1:\xi_2:\cdots:\xi_{n-1}:1)\in Y \subset \mathbb{P}^{n-1}_k$. Therefore, we deduce that $\xi \in Y$ is a
  singular point of $Y$ if and only if $J_n(\xi)=0$.
\end{proof}

This proposition gives a geometric interpretation of the discriminant
of $n-1$ homogeneous polynomials in $n$ variables. We will give a more
precise description of its geometry in Section \ref{inform}.


\subsection{Formulas and formal properties}
  
  In this section we give some properties of the discriminant. Thanks
  to the definition we gave of the discriminant in terms of the resultant, it
  turns out that most of these properties can be derived  from
  the known ones of the resultant.

  \medskip

  Hereafter $R$ will denote an arbitrary commutative ring.

\subsubsection{Elementary transformations} The discriminant of $n-1$
homogeneous polynomials $f_1,\ldots,f_{n-1}$ is invariant under a
permutation of the $f_i$'s. It is also invariant if one adds to one of
the $f_i$'s an element in the ideal generated by the others.

\begin{prop}\label{perm} for all $j=1,\ldots,n-1$, let  $f_j$
be a homogeneous polynomial of degree $d_j\geq 1$
  in $R[X_1,\ldots,X_n]$. Then,
\begin{itemize}
\item[\rm i)] for any permutation $\sigma$ of the set $\{1,\ldots,n-1
  \}$ we have

  $$\Disc(f_{\sigma(1)},\ldots,f_{\sigma(n-1)})=\Disc(f_1,\ldots,f_{n-1})
  \text{ in } R.$$
\item[\rm ii)] for all $i \in \{1,\ldots,n-1\}$ we have
  $$\Disc(f_1,\ldots,f_i+\sum_{j\neq
    i}h_{i,j}f_j,\ldots,f_{n-1})=\Disc(f_1,\ldots,f_{n-1}) \text{ in }
  R,$$
  where the $h_{i,j}$'s are arbitrary homogeneous polynomials in
  $R[X_1,\ldots,X_n]$ of respective degrees $d_i-d_j$ (therefore $h_{i,j}=0$ if $d_i<d_j$).
\end{itemize}
\end{prop}
\begin{proof} Of course, it is sufficient to prove these properties in the
  generic case. The property ii) is an immediate consequence of
  \cite[\S 5.9]{J91}.

  To prove i), we first remark that
  $$J_n^{\sigma}:=J_n(f_{\sigma(1)},\ldots,f_{\sigma(n-1)})=\epsilon(\sigma)J_n(f_1,\ldots,f_{n-1}).$$
  Then, using  \cite[\S 5.8]{J91} we deduce that  
  \begin{eqnarray*}
\Res(f_{\sigma(1)},\ldots,f_{\sigma(n-1)},J_n^\sigma) & = &
\epsilon(\sigma)^{d_1\ldots d_{n-1}}\Res(f_{\sigma(1)},\ldots,f_{\sigma(n-1)},J_n)\\
 & = &
 \epsilon(\sigma)^{d_1\ldots d_{n-1}}\epsilon(\sigma)^{d_1\ldots d_{n-1}\deg(J_n)} \Res(f_1,\ldots,f_n,J_n),
    \end{eqnarray*}
and
\begin{eqnarray*}
  \Res(f_{\sigma(1)},\ldots,f_{\sigma(n-1)},X_n) & = &
  \epsilon(\sigma)^{d_1\ldots
    d_{n-1}}\Res(f_1,\ldots,f_{n-1},X_n).
 \end{eqnarray*}
From here the claimed result follows from \eqref{defeq} (with $i=n$) and
the fact that 
$$d_1\ldots d_{n-1}\deg(J_n)=d_1\ldots
d_{n-1}\sum_{i=1}^{n-1}(d_i-1)$$ 
is always an even integer.
\end{proof}

\subsubsection{Reduction on the variables} 
Hereafter, for any
polynomial $f \in R[X_1,\ldots,X_n]$ we 
denote by $f^{(j)}$ the polynomial obtained by substituting
$X_j$ with $0$ in $f$. Notice that 
$f^{(j)} \in R[X_1,\ldots,X_{j-1},X_{j+1},\ldots,X_n]$.

\begin{prop}[$n\geq 3$]\label{prop:red} For all $i=1,\ldots,n-2$, let $f_i$
be a homogeneous polynomial of degree $d_i\geq 1$
  in $R[X_1,\ldots,X_n]$. The following equality holds in $R$: 
  $$\Disc(f_1,\ldots,f_{n-2},X_n)=(-1)^{d_1\ldots
    d_{n-2}}\Disc(f_1^{(n)},\ldots,f_{n-2}^{(n)}).$$
\end{prop}
\begin{proof} It is sufficient to prove this formula in the
  generic context. From the definition of the discriminant we thus
  have the equality
\begin{multline*}
	\Res(f_1,\ldots,f_{n-2},X_n,J_{n-1}(f_1,\ldots,f_{n-2},X_n))=\\
	\Disc(f_1,\ldots,f_{n-2},X_n) \Res(f_1,\ldots,f_{n-2},X_n,X_{n-1}).	
\end{multline*}
But, from \eqref{Ji} we deduce that
$$J_{n-1}(f_1,\ldots,f_{n-2},X_n)=-\left|\frac{\partial
    (f_1,\ldots,f_{n-2},X_n)}{\partial
    (X_1,\ldots,X_{n-2},X_n)}\right|=
(-1)^n\left|\frac{\partial
    (f_1,\ldots,f_{n-2})}{\partial
    (X_1,\ldots,X_{n-2})}\right|.
$$
And since
\begin{multline}
  \Res(f_1,\ldots,f_{n-2},X_n,J_{n-1}(f_1,\ldots,f_{n-2},X_n))  =
  (-1)^{d_1\ldots d_{n-2}\sum_{i=1}^{n-2}(d_i-1)} \\ \notag
  \Res(f_1,\ldots,f_{n-2},J_{n-1}(f_1,\ldots,f_{n-2},X_n),X_n)
\end{multline}
where $  d_1\ldots d_{n-2}\sum_{i=1}^{n-2}(d_i-1)$ is even, it comes
\begin{multline}
  \Res(f_1,\ldots,f_{n-2},X_n,J_{n-1}(f_1,\ldots,f_{n-2},X_n))  = \\
 \Res(f_1^{(n)},\ldots,f_{n-2}^{(n)},J_{n-1}(f_1^{(n)},\ldots,f_{n-2}^{(n)})).
 \end{multline} 
Moreover, we also have
\begin{eqnarray*}
  \Res(f_1,\ldots,f_{n-2},X_n,X_{n-1}) & = & (-1)^{d_1\ldots d_{n-2}} 
\Res(f_1,\ldots,f_{n-2},X_{n-1},X_n) \\
 & = & (-1)^{d_1\ldots d_{n-2}} \Res(f_1^{(n)},\ldots,f_{n-2}^{(n)},X_{n-1}). 
\end{eqnarray*}
Now taking the ratio of both previous quantities we obtain,
\begin{multline*}
	 (-1)^{d_1\ldots d_{n-2}}
	  \Res(f_1^{(n)},\ldots,f_{n-2}^{(n)},J_{n-1}(f_1^{(n)},\ldots,f_{n-2}^{(n)}))
	=   \\
	\Disc(f_1,\ldots,f_{n-2},X_n)\Res(f_1^{(n)},\ldots,f_{n-2}^{(n)},X_{n-1})	
\end{multline*}
so that, as claimed,
$$ \Disc(f_1,\ldots,f_{n-2},X_n) = (-1)^{d_1\ldots d_{n-2}} \Disc(f_1^{(n)},\ldots,f_{n-2}^{(n)}).
$$
\end{proof}

The following proposition and corollary give reductions of 
the discriminant in cases where certain polynomials $f_1,\ldots,f_{n-1}$ do not depend on all the variables
$X_1,\ldots,X_n$.  

\begin{prop}[$n\geq 3$] Let $k\in \{2,\ldots,n-1\}$ and for all $i=1,\ldots,n-1$ let $f_i$ 
  be a homogeneous polynomial of degree $d_i\geq 1$
  in $R[X_1,\ldots,X_n]$ such that $\sum_{i=1}^{k-1}(d_i-1)\geq 1$.
Assume moreover that
$f_1,\ldots,f_{k-1}$ only depend on the variables
  $X_1,\ldots,X_k$.
  Then, denoting for all integer $i=k,\ldots,n-1$
  $$\hat{f}_i=f_i(0,\ldots,0,X_{k+1},\ldots,X_n) \in
  R[X_{k+1},\ldots,X_n],$$
  we have the equality 
  \begin{multline}
    \Disc(f_1,\ldots,f_{n-1})=
    (-1)^{(n-k)\prod_{i=1}^{n-1}d_i} \Disc(f_1,\ldots,f_{k-1})^{\prod_{i=k}^{n-1}d_i} \\ \notag    
    \Res(\hat{f}_k,\ldots,\hat{f}_{n-1})^{(\prod_{i=1}^{k-1}d_i)(\sum_{i=1}^{k-1}d_i-k)}
    \Res\left(f_1,\ldots,f_{n-1},\left|\frac{\partial
    (f_{k},\ldots,f_{n-1})}{\partial (X_{k+1},\ldots,X_n)}\right|\right).
    \end{multline}
\end{prop}
\begin{proof} As always, it is sufficient to prove this formula in the
  generic case. By definition we have
  $$\Res(f_1,\ldots,f_{n-1},X_1) \Disc(f_1,\ldots,f_{n-1})=\Res(f_1,\ldots,f_{n-1},J_1(f_1,\ldots,f_{n-1})).$$
  From the hypothesis, the Jacobian determinant involved in this
  formula decomposes into four square blocks and one of them is
  identically zero. More precisely, one has
  $$J_1(f_1,\ldots,f_{n-1})=\left| \frac{\partial (f_1,\ldots,f_{k-1})  }
    {\partial (X_2,\ldots,X_k)  } \right|
  \left| \frac{\partial (f_k,\ldots,f_{n-1})  }
    {\partial (X_{k+1},\ldots,X_n)  } \right|$$
  and by multiplicativity of the resultant \cite[\S 5.7]{J91} we deduce
\begin{multline}
    \Res(f_1,\ldots,f_{n-1},J_1(f_1,\ldots,f_{n-1}))=
    \Res\left(f_1,\ldots,f_{n-1},\left| \frac{\partial (f_1,\ldots,f_{k-1})  }
    {\partial (X_2,\ldots,X_k)  } \right| \right) \\ \notag  
    \times \Res\left(f_1,\ldots,f_{n-1},\left| \frac{\partial (f_k,\ldots,f_{n-1})  }
    {\partial (X_{k+1},\ldots,X_n)  } \right|\right).
\end{multline}
Now, permuting polynomials in the resultant \cite[\S 5.8]{J91}, 
\begin{multline}
  \Res\left(f_1,\ldots,f_{n-1},\left| \frac{\partial (f_1,\ldots,f_{k-1})  }
    {\partial (X_2,\ldots,X_k)  } \right| \right)=  \\
(-1)^{\nu}
  \Res\left(f_1,\ldots,f_{k-1},\left| \frac{\partial (f_1,\ldots,f_{k-1})  }
    {\partial (X_2,\ldots,X_k)  } \right| ,f_{k},\ldots,f_{n-1}\right) 
\end{multline}
where $\nu:=(n-k)(\prod_{i=1}^{n-1}d_i)(\sum_{i=1}^{k-1}(d_i-1)) \geq
1$  and is even,
and using Laplace's formula
\cite[\S5.10]{J91} this latter resultant is equal to 
\begin{equation*}
 \Res\left(f_1,\ldots,f_{k-1},\left| \frac{\partial (f_1,\ldots,f_{k-1})  }
    {\partial (X_2,\ldots,X_k)  } \right| \right)^{\prod_{i=k}^{n-1}d_i} 
  \Res(\hat{f}_k,\ldots,\hat{f}_{n-1})^{(\prod_{i=1}^{k-1}d_i)\sum_{i=1}^{k-1}(d_i-1)}.
\end{equation*}
Similarly, we have
\begin{multline}
\Res(f_1,\ldots,f_{n-1},X_1)= \\
(-1)^{(n-k)(\prod_{i=1}^{n-1}d_i)}
\Res(f_1,\ldots,f_{k-1},X_1)^{\prod_{i=k}^{n-1}d_i} 
  \Res(\hat{f}_k,\ldots,\hat{f}_{n-1})^{(\prod_{i=1}^{k-1}d_i)}
\end{multline}
and the claimed formula follows easily by gathering these computations. 
\end{proof}

\begin{cor} Let $k\in \{1,\ldots,n-1\}$ and for all $i=1,\ldots,n-1$, let $f_i$
  be a homogeneous polynomial of degree $d_i\geq 1$
  in $R[X_1,\ldots,X_n]$. Assume moreover that $d_1\geq 2$. If 
  the polynomials $f_1,\ldots,f_{k}$ only depend on the variables
  $X_1,\ldots,X_k$ then $\Disc(f_1,\ldots,f_{n-1})=0$.
\end{cor}
\begin{proof} First assume that $k\geq 2$; since $d_1 \geq 2$ we have
  $\sum_{i=1}^{n-1}(d_i-1)\geq 1$.  Since $f_k$ only depends on the
  variables $X_1,\ldots,X_k$ we deduce that, according to the notation
  of the previous proposition, $\hat{f}_k=0$. Consequently, using the
  formula of this proposition we immediately get that
  $\Disc(f_1,\ldots,f_{n-1})=0$.

  Now assume that $k=1$; thus $f_1=U_{1}X_1^{d_1}$. One may also
  assume that the polynomials $f_2,\ldots,f_{n-1}$ are generic in all the variables
  $X_1,\ldots,X_n$. It follows that $\Res(f_1,\ldots,f_{n-1},X_n)$ is
  nonzero and we know that
  $$\Res(f_1,\ldots,f_{n-1},X_n)\Disc(f_1,\ldots,f_{n-1})=\Res(f_1,\ldots,f_{n-1},J_n(f_1,\ldots,f_{n-1})).$$
But since $f_1=U_{1}X_1^{d_1}$ we deduce that $X_1^{d_1-1}$ divides
$J_n$ and consequently that $\Res(f_1,\ldots,f_{n-1},J_n)$ vanishes.
\end{proof}

\subsubsection{Multiplicativity} 
We now describe the multiplicativity property of the discriminant, property that was already known to Sylvester \cite{Sylvester}. Recall that the discriminant of $n-1$  homogeneous polynomials of degree 1 equals 1 (the unit of the ground ring) by convention.

\begin{prop}\label{mult} Let 
  $f_1'$, $f_1''$, $f_2,\ldots,f_{n-1}$ be $n$ homogeneous polynomials in $R[X_1,\ldots,X_n]$
  of  positive degree
  $d_1',d_1''$, $d_2,\ldots$, $d_{n-1}\geq 1$, respectively. Then,
  \begin{multline}
    \Disc(f_1'f_1'',f_2,\ldots,f_{n-1})= \notag \\
    (-1)^s \Disc(f_1',f_2,\ldots,f_{n-1})
    \Disc(f_1'',f_2,\ldots,f_{n-1})
    \Res(f_1',f_1'',f_2,\ldots,f_{n-1})^2, \notag
  \end{multline}
  where $s:=d_1'd_1''d_2 \ldots d_{n-1}$.
\end{prop}
\begin{proof} It is sufficient to prove this result in the generic
  case, so let us assume that $f_1',
  f_1'',f_2,\ldots,f_n$ are generic polynomials. It is easy to see
  that 
  $$J_n(f_1'f_1'',f_2,\ldots,f_n)=f_1'J_n(f_1'',f_2,\ldots,f_n)+f_1''J_n(f_1',f_2,\ldots,f_n)
  =:f_1'J_n''+f_1''J_n'.$$
  Assume first that $\deg(J_n'')\geq 1$ and $\deg(J_n')\geq 1$. Using
  \cite[\S5.7 \& \S5.8]{J91} we obtain
\begin{align*}
	   \lefteqn{\Res(f_1'f_1'',f_2,\ldots,f_{n-1},J_n)} \\
	& =\Res(f_1',f_2,\ldots,f_{n-1},f_1''J_n')\Res(f_1'',f_2,\ldots,f_{n-1},f_1'J_n'')  \\
	& =(-1)^s\Res(f_1',f_1'',f_2,\ldots,f_{n-1})^2\Res(f_1',f_2,\ldots,f_{n-1},J_n') 
    \Res(f_1'',f_2,\ldots,f_{n-1},J_n'')
\end{align*}
where $s:=d_1'd_1''d_2\ldots d_{n-1}.$ And since
$$\Res(f_1'f_1'',f_2,\ldots,f_{n-1},X_n)=\Res(f_1',f_2,\ldots,f_{n-1},X_n)\Res(f_1'',f_2,\ldots,f_{n-1},X_n),$$
we deduce the expected formula by applying \eqref{defeq}.

 Assume now that $\deg(J_n')=0$ and $\deg(J_n'')\geq 1$.  Then, in the
 previous computations, the resultant $\Res(f_1',f_2,\ldots,f_{n-1},J_n')$ must
 be replaced by $J_n'$ (under our hypothesis $d_1'd_2\ldots
 d_{n-1}=1$). But it turns out that, always since $\deg(J_n')=0$, 
 $J_n'=\Res(f_1',f_2,\ldots,f_{n-1},X_n)$  and consequently  the
 whole formula remains exact. A similar argument shows that this
 formula is also exact if $\deg(J_n')=1$ and $\deg(J_n'')\geq 0$, and
 if $\deg(J_n')=\deg(J_n'')\geq 0$.
\end{proof}


\begin{cor}\label{det2}  Let $d_1,\ldots,d_{n-1}$ be $n-1$ integers greater or
  equal to 2 and  let
   $l_{i,j}$, for $1\leq i \leq
  n-1$ and $1\leq j \leq d_i$, be linear forms in
  $R[X_1,\ldots,X_n]$. Then 
  \begin{equation*}
    \Disc\left(\prod_{j=1}^{d_1}l_{1,j},\ldots,\prod_{j=1}^{d_{n-1}}l_{n-1,j}\right)=(-1)^{s} \prod_{I}
    \det(l_{1,j_1},l_{2,j_2},\ldots,l_{n-1,j_{n-1}},l_{i,j})^2
    \end{equation*}
    where  $s:=\frac{1}{2}\prod_{i=1}^{n-1}d_i\sum_{i=1}^{n-1}(d_i-1)$
  and  
    the product runs over the set 
\begin{multline}
 I:=\{ (j_1,\ldots,j_{n-1},i,j) \ | \  1\leq j_1 \leq d_1, 1\leq j_2 \leq
 d_2, \ldots, 1\leq j_{n-1} \leq d_{n-1}, \\ \notag
 1 \leq i \leq n-1 \text{ and } 1\leq j \leq d_i \text{ such that } j\neq j_i \}.
\end{multline}
\end{cor}

\subsubsection{Covariance}

Assume that $n\geq 2$ and suppose given a sequence of $n-1$ positive integers $d_1,\ldots,d_{n-1}$ such that $\sum_{i=1}^{n-1}(d_i-1)\geq 1$. For all $d \in \NN$ set $I_d:=\{i\in\{1,\ldots,n\} | d_i=d\}$ and define $L:=\{d\in \NN | I_d\neq \emptyset\}$. In this way, the set $\{1,\ldots,n\}$ is the disjoint union of $I_d$ with $d\in L$.

Let $\varphi$ be a square matrix of size $n-1$ with coefficients in $R$
$$\varphi=\left[
\begin{array}{ccc}
u_{1,1} & \cdots & u_{1,n-1} \\
\vdots & & \vdots\\
u_{n-1,1} & \cdots & u_{n-1,n-1} 
\end{array}
\right].$$
 We will say that $\varphi$ is adapted to the sequence $d_1,\ldots,d_n$ if and only if 
$$u_{i,j}\neq 0 \Rightarrow d_i=d_j.$$
Equivalently, $\varphi$ is adapted to the sequence $d_1,\ldots,d_n$ if and only if $\varphi$ can be transformed by row and column permutations into a block diagonal matrix whose diagonal blocs are given by $\varphi_d:=\varphi_{|I_d\times I_d}$ for all $d\in L$; in particular $\det(\varphi)=\prod_{d\in L}\det(\varphi_d) \in R$.

\begin{prop}\label{prop:covariance} Assume that $n\geq 2$ and suppose given a sequence of $n-1$ positive integers $d_1,\ldots,d_{n-1}$ such that $\sum_{i=1}^{n-1}(d_i-1)\geq 1$ and a sequence of $n-1$ homogeneous polynomials $f_1,\ldots,f_{n-1}$ in $R[X_1,\ldots,X_n]$ of degree $d_1,\ldots,d_{n-1}$ respectively. Then, for all $i=1,\ldots,n-1$ and all matrix $\varphi=\left( u_{i,j} \right)_{1\leq i,j\leq n-1}$ with coefficients in $R$ adapted to $d_1,\ldots,d_{n-1}$, the polynomial $\sum_{j=1}^{n-1}u_{i,j}f_j \in R$ is homogeneous of degree $d_i$ and we have
\begin{multline*}
	\Disc\left(\sum_{j=1}^{n-1}u_{1,j}f_j,\ldots,\sum_{j=1}^{n-1}u_{n-1,j}f_j\right)= \\
	\left(\prod_{d\in L}  \det(\varphi_d)^{\frac{d_1\ldots d_{n-1}\left((d-1)+\sum_{i=1}^{n-1}(d_i-1)\right)}{d}}   \right)
	\Disc(f_1,\ldots,f_{n-1}).	
\end{multline*}
\end{prop}
\begin{proof} By specialization, we can assume that the coefficients of the polynomials $f_1,\ldots,f_{n-1}$ and all the $u_{i,j}$ are distinct indeterminates so that $R$ is the polynomial ring of these indeterminates over the integers.

By definition of the discriminant we have 
\begin{equation}\label{eq:cov1}
	\Res\left(f_1,\ldots,f_{n-1},J_n(f_1,\ldots,f_{n-1})\right)=\Disc(f_1,\ldots,f_{n-1})\Res(f_1,\ldots,f_{n-1},X_n)
\end{equation}
and
\begin{multline}\label{eq:cov2}
	\Res\left(\sum_{j=1}^{n-1}u_{1,j}f_j,\ldots,\sum_{j=1}^{n-1}u_{n-1,j}f_j,J_n\left(\sum_{j=1}^{n-1}u_{1,j}f_j,\ldots,\sum_{j=1}^{n-1}u_{n-1,j}f_j\right)\right)\\
	=\Disc\left(\sum_{j=1}^{n-1}u_{1,j}f_j,\ldots,\sum_{j=1}^{n-1}u_{n-1,j}f_j\right)\Res\left(\sum_{j=1}^{n-1}u_{1,j}f_j,\ldots,\sum_{j=1}^{n-1}u_{n-1,j}f_j,X_n\right).
\end{multline}
Now, it is not hard to check that
$$J_n\left(\sum_{j=1}^{n-1}u_{1,j}f_j,\ldots,\sum_{j=1}^{n-1}u_{n-1,j}f_j\right)=\det(\varphi)J_n(f_1,\ldots,f_{n-1})$$
so that
\begin{multline*}
	\Res\left(\sum_{j=1}^{n-1}u_{1,j}f_j,\ldots,\sum_{j=1}^{n-1}u_{n-1,j}f_j,J_n\left(\sum_{j=1}^{n-1}u_{1,j}f_j,\ldots,\sum_{j=1}^{n-1}u_{n-1,j}f_j\right)\right)\\
= \det(\varphi)^{d_1\ldots d_{n-1}}\Res\left(\sum_{j=1}^{n-1}u_{1,j}f_j,\ldots,\sum_{j=1}^{n-1}u_{n-1,j}f_j,J_n(f_1,\ldots,f_{n-1})\right).
\end{multline*}
But since $J_n(f_1,\ldots,f_{n-1})$ is a polynomial of degree $\sum_{i=1}^{n-1}(d_i-1)\geq 1$, the covariance property of the resultant \cite[\S 5.11]{J91} yields 
\begin{multline*}
	\Res\left(\sum_{j=1}^{n-1}u_{1,j}f_j,\ldots,\sum_{j=1}^{n-1}u_{n-1,j}f_j,J_n(f_1,\ldots,f_{n-1})\right) \\
	= \left(  \prod_{d\in L}  \det(\varphi_d)^{\frac{d_1\ldots d_{n-1}\sum_{i=1}^{n-1}(d_i-1)}{d}} \right)
	\Res(f_1,\ldots,f_n,J_n(f_1,\ldots,f_{n-1}))
\end{multline*}
and we deduce that
\begin{multline}\label{eq:cov3}
	\Res\left(\sum_{j=1}^{n-1}u_{1,j}f_j,\ldots,\sum_{j=1}^{n-1}u_{n-1,j}f_j,J_n\left(\sum_{j=1}^{n-1}u_{1,j}f_j,\ldots,\sum_{j=1}^{n-1}u_{n-1,j}f_j\right)\right)\\
	= \det(\varphi)^{d_1\ldots d_{n-1}}\left(  \prod_{d\in L}  \det(\varphi_d)^{\frac{d_1\ldots d_{n-1}\sum_{i=1}^{n-1}(d_i-1)}{d}} \right)
	\Res(f_1,\ldots,f_n,J_n).
\end{multline}
Again by the covariance formula for resultants, we have
\begin{multline}\label{eq:cov4}
	\Res\left(\sum_{j=1}^{n-1}u_{1,j}f_j,\ldots,\sum_{j=1}^{n-1}u_{n-1,j}f_j,X_n\right)=\\ 
	\left(  \prod_{d\in L}  \det(\varphi_d)^{\frac{d_1\ldots d_{n-1}}{d}} \right)\Res(f_1,\ldots,f_{n-1},X_n)
\end{multline}
and therefore, since $\det(\varphi)=\prod_{d\in L}\det(\varphi_d)$, the comparison of \eqref{eq:cov1}, \eqref{eq:cov2}, \eqref{eq:cov3} and \eqref{eq:cov4} gives the claimed formula.
\end{proof}

\subsubsection{Reduction modulo $\delta$} Recall from Lemma \ref{lem1} that, for all $1\leq i,j\leq n$ we have
\begin{equation}\label{eq:discRes}
	X_iJ_j(f_1,\ldots,f_{n-1}) - X_jJ_i(f_1,\ldots,f_{n-1}) \in \delta.(f_1,\ldots,f_{n-1}) \subset \delta.A[X_1,\ldots,X_n]	
\end{equation}
where $\delta:=gcd(d_1,\ldots,d_{n-1})$.
Considering the (cohomological) Koszul complex associated to the sequence $X_1,\ldots,X_n$ in the ring $A/\delta.A[X_1,\ldots,X_n]$
$$ 0 \rightarrow \frac{A}{\delta.A}[X_1,\ldots,X_n] \xrightarrow{\mathbf{d}_1={}^t[X_1,\cdots,X_n]} \bigoplus_{i=1}^n \frac{A}{\delta.A}[X_1,\ldots,X_n] \xrightarrow{\mathbf{d}_2} \cdots,$$
we notice that since $n\geq 2$, its cohomology groups $H^0$ and $H^1$ are both equal to $0$. In addition, the equations \eqref{eq:discRes} imply that $(J_1,\ldots,J_n)$ belongs to the kernel of $\mathbf{d}_2$. Therefore, we deduce that there exists a polynomial $\Delta \in A[X_1,\ldots,X_n]$ whose residue class in $A/\delta.A[X_1,\ldots,X_n]$ is unique and such that
\begin{equation}\label{eq:discRes2}
	J_i(f_1,\ldots,f_{n-1})=X_i \Delta \mod \delta.A[X_1,\ldots,X_n], \ 1\leq i \leq n.	
\end{equation}
From here, we get the following property.
\begin{prop}
	With the above notation, we have the following equality in ${}_k A$:
	$$\Disc(f_1,\ldots,f_{n-1})=\Res(f_1,\ldots,f_{n-1},\Delta) \mod \delta.$$
\end{prop}
\begin{proof} From \eqref{eq:discRes2} and the multiplicativity of the resultant, we obtain that
	$$\Res(f_1,\ldots,f_{n-1},J_n)=\Res(f_1,\ldots,f_{n-1},X_n)\Res(f_1,\ldots,f_{n-1},\Delta) \mod \delta.$$
By definition of the discriminant, it follows that
\begin{multline*}
\Res(f_1,\ldots,f_{n-1},X_n)\Disc(f_1,\ldots,f_{n-1})= \\ 
\Res(f_1,\ldots,f_{n-1},X_n)\Res(f_1,\ldots,f_{n-1},\Delta) \mod \delta	
\end{multline*}
from we deduce the claimed equality since $\Res(f_1,\ldots,f_{n-1},X_n)$ is a nonzero divisor in $A/\delta.A[X_1,\ldots,X_n]$ by Proposition \ref{prop:resultant}.
\end{proof}

Obviously, this result is useless if $\delta=1$, but as soon as $\delta>1$ it allows to explicit the discriminant as a single resultant modulo $\delta$. For instance, suppose given the two quadrics
\begin{align*}
	f_1 & := a_0X_1^2+a_1X_1X_2+a_2X_1X_3+a_3X_2^2+a_4X_2X_3+a_5X_3^2, \\
	f_2 & := b_0X_1^2+b_1X_1X_2+b_2X_1X_3+b_3X_2^2+b_4X_2X_3+b_5X_3^2.
\end{align*}
We have $\delta=2$ and it is not hard to see that $J_i=X_i\Delta \mod 2$, $i=1,2,3$,  where
$$\Delta=X_1
\left| 
\begin{array}{cc}
	a_1 & a_2 \\
	b_1 & b_2
\end{array}
\right| + X_2
\left|
\begin{array}{cc}
	a_1 & a_4 \\
	b_1 & b_4
\end{array}
\right| + X_3
\left|
\begin{array}{cc}
	a_2 & a_4 \\
	b_2 & b_4
\end{array}
\right|.
$$
It follows that
$$\Disc(f_1,f_2)=\Res(f_1,f_2,\Delta) \mod 2.\ZZ[a_0,\ldots,a_5,b_0,\ldots,b_5].$$

\subsection{Inertia forms and the discriminant}\label{inform}

The resultant was originally built to provide a condition
for the existence of a common root to a polynomial system. For its
part, the discriminant was introduced to give a condition for the
existence of a singular root in such a polynomial system.  The aim of
this section is to show that the definition we gave of the discriminant
of $n-1$ homogeneous polynomials in $n$ variables (i.e.~Definition \ref{defdisc}) fits this goal.

\medskip

Hereafter we take again the notation of Section
\ref{sec:def}: $k$ is a commutative ring and for all $i=1,\ldots,n-1$, $n\geq 2$, we set
$$f_i(X_1,\ldots,X_n):=\sum_{|\alpha|=d_i\geq 1}U_{i,\alpha}X^\alpha \in {}_kA[X_1,\ldots,X_n]_{d_i}$$
where ${}_k A:=k[U_{i,\alpha}\,|\, |\alpha|=d_i, \, i=1,\ldots,n-1]$. Notice that we will often omit the subscript $k$ to not overload the notation, but we will print it whenever there is a confusion or a need to emphasis it.

Now, we define the ideals of $C=A[X_1,\ldots,X_n]$
$$\Dc=(f_1,\ldots,f_{n-1},J_1,\ldots,J_n), \ \ \mm=(X_1,\ldots,X_n)$$
and set $B:=C/\Dc$.
The ring $B$ is graded (setting $\mathrm{weight}(X_i)=1$) and we can thus
consider the projective scheme $\mathrm{Proj}(B)\subset
\mathbb{P}_{A}^{n-1}$ that corresponds set-theoretically to the
points $((u_{i,\alpha})_{i,\alpha},x) \in \mathrm{Spec}(A)\times
\mathbb{P}_{k}^{n-1}$ such that the $f_i$'s and the $J_i$'s
vanish simultaneously. The scheme-theoretic image of the projection
$$\mathrm{Proj}(B) \rightarrow \mathrm{Spec}(A)$$
is a closed subscheme of $\mathrm{Spec}(A)$ whose defining ideal is exactly
$$\pp:=H^0_\mm(B)_0=\TF_\mm(\Dc)_0$$ 
where we recall that
\begin{equation}\label{TFD}
\TF_\mm(\Dc)=\ker \left( C \rightarrow \prod_{i=1}^n
B_{X_i}\right).\end{equation}
 
\begin{prop}\label{prop:Bdomain} If $k$ is a domain then for all $i=1,\ldots,n$ the ring $B_{X_i}$ is a domain. 
\end{prop}
\begin{proof} For simplicity, we prove the claim for $i=n$; the
  other cases can be treated exactly in the same way. 
  
  Let $h_1,h_2$ be two elements in $C$ such that their product 
  $h_1h_2$ vanishes in $B_{X_n}$ (recall that we have the canonical
  projection $C \rightarrow B=C/\Dc$). 
  This means that, up to multiplication by some power of
  $X_n$, this product is in the ideal $\Dc$. Thus, 
  using Lemma \ref{lem1}, ii), we deduce that there
  exists $\nu \in \mathbb{N}$ such that
  $$X_n^{\nu} h_1h_2 \in (f_1,\ldots,f_{n-1},J_n).$$
  Now, taking the additional notation of the subsection \ref{in-form},
  we substitute each $\EE_i$ by $\EE_i-\tilde{f}_i$ and obtain that
  $h_1h_2(\EE_i-\tilde{f}_i) \in (\tilde{J}_n)$ in
  $A[X_1,\ldots,X_{n-1}]$ (since $f_i(\EE_i-\tilde{f}_i)=0$). But by Lemma \ref{Jirred} 
  $\tilde{J}_n$ is prime in $A[X_1,\ldots,X_{n-1}]$ 
 and it follows that it divides $h_1(\EE_i-\tilde{f}_i)$ or $h_2(\EE_i-\tilde{f}_i)$, say $h_1(\EE_i-\tilde{f}_i)$. Therefore  
  there exists $\mu \in \mathbb{N}$ such that
  $$ X_n^{\mu}h_1 \in (f_1,\ldots,f_{n-1},J_n) \subset \Dc,$$
  that is to say $h_1$ equals 0 in $B_{X_n}$, and the claim is
  proved. 
\end{proof}
\begin{cor}\label{cor:TFprimedim1} 
	Moreover, for all $i=1,\ldots,n$ we have 
	$$\TF_\mm(\Dc)=\TF_{(X_i)}(\Dc)=\ker(C \rightarrow B_{X_i}), \ H^0_\mm(B)=H^0_{(X_i)}(B).$$
	In particular,  $$\pp=A\cap (\tilde{f}_1,\ldots,\tilde{f}_{n-1},\tilde{J}_n) \subset
	A[X_1,\ldots,X_{n-1}].$$
As a consequence, if $k$ is a domain then $\TF_\mm(\Dc)$ and $\pp$ are prime ideals
of ${}_k C$ and ${}_k A$ respectively. 
\end{cor}
\begin{proof} The only thing to prove in that for all couple of integers $(i,j) \in \{1,\ldots,n \}^2$ the variable $X_i$ is a nonzero divisor in the ring $B_{X_j}$. Indeed, this property implies immediately the equalities given in this corollary (similarly to \eqref{TF} and \eqref{form1} for the case of the resultant). From here, assuming moreover that $k$ is a domain we deduce that $\TF_\mm(\Dc)$ and $\pp$ are prime ideals by Proposition \ref{prop:Bdomain}.

\medskip
	
So let us fix a couple of integer $(i,j) \in \{1,\ldots,n \}^2$ and prove that $X_i$ is a nonzero divisor in ${}_k B_{X_j}$ (for any commutative ring $k$). 
By Proposition \ref{prop:Bdomain}, this property holds if $k$ is a domain. On the one hand, this implies that ${}_\ZZ B_{X_j}$ is a torsion-free abelian group, hence flat (as a $\ZZ$-module). On the other hand, this implies that the multiplications by $X_i$ in ${}_\ZZ B_{X_j}$ and ${}_{\ZZ/p\ZZ} B_{X_j}$, $p$ a prime integer, are all injective maps. Denoting by ${}_\ZZ Q$ the quotient abelian group of the multiplication by $X_i$ in ${}_\ZZ B_{X_j}$, we deduce that ${}_\ZZ Q$ is a torsion-free, hence flat, abelian group. Indeed, the exact sequence of abelian groups
\begin{equation}\label{eq:flatres}
	0 \rightarrow {}_\ZZ B_{X_j} \xrightarrow{\times  X_i} {}_\ZZ B_{X_j} \rightarrow {}_\ZZ Q \rightarrow 0
\end{equation}
is a flat resolution of ${}_\ZZ Q$ and it remains exact after tensorization by $\ZZ/p\ZZ$ over $\ZZ$ for all prime integer $p$. Therefore $\mathrm{Tor}^\ZZ_1(\ZZ/p\ZZ,{}_\ZZ Q)=0$ and hence ${}_\ZZ Q$ is torsion-free, hence flat. As a consequence, 
for any commutative ring $k$ we have $\mathrm{Tor}_1^{\ZZ}({}_\ZZ Q,k)=0$ and therefore the multiplication by $X_i$ in ${}_k B_{X_j}$ is an injective map, i.e.~ $X_i$ is a nonzero divisor in ${}_k B_{X_j}$. 
\end{proof}

\begin{lem}\label{lemcor} $\Disc(f_1,\ldots,f_{n-1})$ belongs to the ideal
  $\pp \subset {}_kA$.
\end{lem}
\begin{proof} By specialization, it is sufficient to prove this property under the assumption that $k=\ZZ$.
	Denote $\rho:=\Res(f_1,\ldots,f_{n-1},X_n)$. From 
	Definition \ref{defdisc} and Lemma  \ref{lem1}, ii) we deduce that
   there exists $\nu$ such that
  $$X_n^{\nu}\rho \Disc(f_1,\ldots,f_{n-1}) \in
  (f_1,\ldots,f_{n-1},J_n).$$
  Now, taking again the notation of subsection \ref{in-form} and 
  substituting each $\EE_i$ by $\EE_i-\tilde{f}_i$ we deduce that
  $\rho\Disc(f_1,\ldots,f_{n-1})(\EE_i-\tilde{f}_i) \in (\tilde{J}_n)$ in
  $A[X_1,\ldots,X_{n-1}]$. But   $\tilde{J}_n$ is prime in $A[X_1,\ldots,X_{n-1}]$ by Lemma
  \ref{Jirred}, and it is coprime
  with $\rho$ since $\rho$ does not depend on the variables $X_1,\ldots,X_n$
  and is also prime. Therefore $\tilde{J}_n$ 
   must divide $\Disc(f_1,\ldots,f_{n-1})(\EE_i-\tilde{f}_i)$ and we obtain that there exists $\mu \in \mathbb{N}$ such that
\begin{equation}\label{eq:firstrelation}
	X_n^{\mu}\Disc(f_1,\ldots,f_{n-1}) \in (f_1,\ldots,f_{n-1},J_n)
	  \subset \Dc.
\end{equation}
 In other words, $\Disc(f_1,\ldots,f_{n-1}) \in \TF_{(X_n)}(\Dc)=\TF_\mm(\Dc)$.  
\end{proof}
 
\begin{thm}\label{thm:2<>0} If $2$ is a nonzero divisor in $k$ then $\pp$ is generated by the discriminant $\Disc(f_1,\ldots,f_{n-1})$. In particular, if $k$ is moreover assumed to a domain then $\Disc(f_1,\ldots,f_{n-1})$ is a prime polynomial in ${}_kA$.
\end{thm}
\begin{proof} We first prove this theorem under the assumption that $k$ is a UFD. 
So assume that $k$ is a UFD and
  let $a \in \pp=\TF_\mm(\Dc)\cap A$. Then
  there exists $\nu \in \mathbb{N}$ such that $X_n^\nu a \in
  (f_1,\ldots,f_{n-1},J_n).$ Therefore we have the inclusion
  $$(f_1,\ldots,f_{n-1},X_n^\nu a) \subset (f_1,\ldots,f_{n-1},J_n)$$
  from we deduce, using the divisibility property of the resultant
  \cite[\S 5.6]{J91}, that
  $$\Res(f_1,\ldots,f_{n-1},J_n) \text{ divides }
  \Res(f_1,\ldots,f_{n-1},X_n^\nu a).$$
Let us denote by $\rho:=
\Res(f_1,\ldots,f_{n-1},X_n)=\Res(f_1^{(n)},\ldots,f_{n-1}^{(n)})$
(see Remark \ref{fbarre}). 
From Definition
  \ref{defdisc} and the multiplicativity property of the resultant
  \cite[\S 5.7]{J91} we obtain that
  \begin{equation}\label{div1}
  \Disc(f_1,\ldots,f_{n-1}) \text{ divides } a^{d_1\ldots
    d_{n-1}}\rho^{\nu-1}
  \end{equation}
for all $a \in \pp$. But it turns out that 
 $\Disc(f_1,\ldots,f_{n-1})$ and
  $\rho$ are coprime in $A$.
  Indeed, since $\rho$ is irreducible, if
  $D:=\Disc(f_1,\ldots,f_{n-1})$ and $\rho$ are not coprime then
  $\rho$ must divide $D$. Consider the specialization where each
  polynomial $f_i$ is specialized to a product of generic linear
  forms. Then,  $\rho$ specializes to a product of determinants where
  each determinant is a prime polynomial (see for instance \cite[Theorem 2.10]{BrVe88}) in the coefficients of
  these linear forms except the ones of the variables $X_n$. On the
  other hand, $D$ specializes to a product of square of determinants
  (see Corollary \ref{det2}), where each determinant is a prime
  polynomial in the coefficients of the generic linear form and does 
  depend on the ones of the variable $X_n$. We thus obtain a
  contradiction and deduce that $\rho$ and $D$ are coprime. \cite[Theorem 2.10]{BrVe88}
Therefore, from \eqref{div1} and the fact that $\rho$ is prime in $A$ we deduce that  for all $a \in
\pp$ the discriminant $D$ divides  $a^{d_1\ldots
  d_{n-1}}$ and hence that
$$\pp^{d_1\ldots d_{n-1}} \subset (D) \subset \pp.$$
 Since $\pp$ is prime, we 
deduce that $D=c.P^p$ where $c$ is an invertible element in $k$, $p$ is a positive integer and $P$ is an irreducible element in $A$ such that $\pp$ is a principal ideal generated by $P$.

Now, always under the assumption that $k$ is UFD, we will prove that $p=1$ if $2\neq 0$ in $k$. Notice that we can assume $d_1\geq 2$ because if $d_1=\cdots=d_{n-1}=1$ then $\pp=(D)=A$ and we can permute polynomials by Proposition \ref{perm}, i). To begin with,  consider the specialization of the polynomial $f_1$ to a product of a
generic linear form $l$ and a generic polynomial $f_1'$ of degree
$d_1-1$. By Proposition \ref{mult}, $D$ specializes, up to sign, to the product
\begin{equation}\label{eq:productDDR2}
	\Disc(l,f_2,\ldots,f_{n-1})
    \Disc(f_1',f_2,\ldots,f_{n-1})
    \Res(l,f_1',f_2,\ldots,f_{n-1})^2.
\end{equation}
Since all the polynomials $l,f_1',f_2,\ldots,f_{n-1}$ are generic of positive degree, 
this product is nonzero. Moreover, the factor $\Res(l,f_1',f_2,\ldots,f_{n-1})$ is irreducible and is clearly coprime with the two discri\-mi\-nants appearing in \eqref{eq:productDDR2}. It follows that necessarily $p\leq 2$, i.e.~$p=1$ or $p=2$.

To prove that $p=1$, equivalently that $D$ is irreducible, we proceed by induction on the integer $r:=\sum_{i=1}^{n-1}d_i$. The intricate point is actually the initialization step. Indeed, assume that $D$ is irreducible for $r=n$ (observe that $D=1$ if $r=n-1$). Then, using the specialization \eqref{eq:productDDR2}, we deduce immediately by induction that both discriminants in \eqref{eq:productDDR2} are irreducible and  coprime, and consequently that $D$ is also irreducible. Therefore, we have to show that if $d_1=2$ and $d_2=\cdots=d_{n-1}=1$ then $D$ is irreducible. For that purpose, we consider the  specialization
$$\left\{\begin{array}{rcl} 
	f_1 &=& U_{1,1}X_1^2+U_{1,2}X_1X_2+U_{2,2}X_2^2+\sum_{i=3}^n U_{i,i}X_i^2 \\
	f_2 &=& X_3-V_{3}X_1 \\
	\vdots \\
	f_{n-1} &=& X_n-V_{n}X_1
\end{array}
\right.$$
and the matrix 
$$\varphi=\left[
\begin{array}{ccccc}
	1 & 0 & 0 & \cdots & 0 \\
	0 & 1 & 0 & \cdots & 0 \\
	-V_3 & 0 & 1 & \ddots & \vdots\\
	\vdots & \vdots &\ddots & \ddots & 0 \\
	-V_n & 0 & \cdots & 0 & 1
\end{array}\right]
$$
that corresponds to a linear change of coordinates such that $f_i=X_{i+1}\circ \varphi$ for all $i=2,\ldots,n-1$. Applying Proposition \eqref{chgcoor} then Proposition \ref{prop:red}, we get
\begin{align*}
\Disc(f_1,f_2,\ldots,f_{n-1}) &= \Disc(f_1\circ \varphi^{-1},X_3,\ldots,X_n) \\
&= \Disc\left(U_{1,1}X_1^2+U_{1,2}X_1X_2+U_{2,2}X_2^2+\sum_{i=3}^n U_{i,i}V_i^2X_1^2 \right) \\
&= \Disc\left( \left(U_{1,1}+ \sum_{i=3}^n U_{i,i}V_i^2\right)X_1^2 +U_{1,2}X_1X_2+U_{2,2}X_2^2 \right) \\
&= U_{1,2}^2 - 4 U_{2,2}\left(U_{1,1}+ \sum_{i=3}^n U_{i,i}V_i^2\right).
\end{align*}
Since $2\neq 0$ in $k$, this is an irreducible polynomial. Therefore, we deduce that necessarily $p=1$, i.e.~that $D=c.P$. Since $c.P$ also generates $\pp=(P)$, we conclude that $D$ is an irreducible polynomial that generates $\pp$. This concludes the proof of the theorem under the assumptions that $k$ is a UFD and $2\neq 0$ in $k$.

\medskip

It remains to show that this theorem holds with the single assumption that $2$ is a nonzero divisor in $k$, $k$ being an arbitrary commutative ring. For that purpose, consider the exact sequence of abelian groups
\begin{equation}\label{eq:exactseqE}
	0 \rightarrow {}_\ZZ A \xrightarrow{\times D} {}_\ZZ A \rightarrow {}_\ZZ B_{X_n} \rightarrow E \rightarrow 0
\end{equation}
where the map on the left is the multiplication by $D$, the map on the middle is the canonical one and where $E$ is the cokernel of this latter. By what we have just proved above, this sequence is exact and remains exact after tensorization by $\ZZ/p\ZZ$ over $\ZZ$ for all prime integer $p\neq 2$ (they are all UFD). Since ${}_\ZZ A$ and ${}_\ZZ B_{X_n}$ are torsion-free, the exact sequence \eqref{eq:exactseqE} is a flat resolution of $E$ and therefore for all integer $i\geq 2$ the abelian group $\mathrm{Tor}_i^{\ZZ}(-,E)$ is supported on $V((2))$. Now, let $M$ be an abelian group without $2$-torsion. The abelian group $M_{(2)}$ is a flat $\ZZ_{(2)}$-module and hence for all $i\geq 1$ we have
$$\mathrm{Tor}_i^{\ZZ}(M,E)_{(2)}\simeq \mathrm{Tor}_i^{\ZZ_{(2)}}(\ZZ_{(2)}\otimes M, \ZZ_{(2)}\otimes E) \simeq \mathrm{Tor}_i^{\ZZ_{(2)}}(M_{(2)},\ZZ_{(2)}\otimes E)= 0.$$
It follows that $\mathrm{Tor}_i^{\ZZ}(M,E)_{(p)}=0$ for all $i\geq 2$ and all prime integer $p$, so that $\mathrm{Tor}_i^{\ZZ}(M,E)=0$ for all $i\geq 2$. Consequently, since $2$ is a nonzero divisor in $k$, $k$ has no $2$-torsion and we deduce that the sequence obtained by tensorization of \eqref{eq:exactseqE} by $k$ over $\ZZ$
\begin{equation*}
	0 \rightarrow {}_k A \xrightarrow{\times D} {}_k A \rightarrow {}_k B_{X_n} \rightarrow k\otimes E \rightarrow 0
\end{equation*}
is exact and the theorem is proved.
\end{proof}

It is reasonable to ask what happens if $2$ is a zero divisor in $k$. As shown in \cite[\S 8.5.2]{ApJo}, one can not expect in this case that the discriminant generates $\pp$, nor even that $\pp$ is a principal ideal. Indeed, in \emph{loc.~cit.~}the authors exhibit an example where $\pp$ is not a principal ideal with the settings $n=2$, $d_1=2$ and $k=\ZZ/2^r\ZZ$ with $r\geq 2$. Nevertheless, we will show in the following theorem that the situation is not so bad if $k$ is assumed to be a domain.

\begin{thm}\label{thm:2=0} Assume that $k$ is a domain and that $2=0$ in $k$. Then 
	$$\Disc(f_1,\ldots,f_{n-1})=P^2$$ 
where $P$ is a prime polynomial that generates $\pp$.
\end{thm}
\begin{proof} We first prove this theorem under the stronger assumption that $k$ is a UFD such that $2=0$ in $k$. To begin with, recall that in the proof of Theorem \ref{thm:2<>0} it is shown that there exists a prime element $P\in A$ and an integer $p\leq 2$ such that the discriminant $D:=\Disc(f_1,\ldots,f_{n-1})$ satisfies $D=c.P^p$, $P$ being a generator of the prime and principal ideal $\pp$. We will show that $p=2$ under our assumptions. Our strategy is based on the use of a Mertens' formula that allows to rely on a discriminant of a unique bivariate and homogeneous polynomial. Indeed, in this case (i.e.~$n=2$) it is known that the claimed result holds \cite[Proposition 60]{ApJo} (see also Theorem \ref{thm:disc-irred} in the case $n=2$ for a self-contained reference).
	
	Introduce some notation related to the Mertens' formulae given in the appendix at the end of this paper. Let $U_1,\ldots,U_n$ be new indeterminates and define
	$$\theta(U_1,\ldots,U_n):=\Res(f_1,\ldots,f_{n-1},\sum_{i=1}^n U_iX_i) \in A[U_1,\ldots,U_n]$$
	and $\theta_i(U_1,\ldots,U_n):=\partial \theta/\partial U_i \in A[U_1,\ldots,U_n]$ for all $i=1,\ldots,n$. In addition, let $V_1,\ldots, V_n$, $W_1,\ldots,W_n$, $X,Y$ be a collection of some other new indeterminates and consider the ring morphisms
	\begin{eqnarray*}
		\rho : A[U_1,\ldots,U_n] & \rightarrow & A[V_1,\ldots,V_n,W_1,\ldots,W_n][X_1,\ldots,X_n] \\
		U_i & \mapsto & V_i(\sum_{j=1}^nW_jX_j)-W_i(\sum_{j=1}^nV_jX_j)
	\end{eqnarray*}
	and 
	\begin{eqnarray*}
		\overline{\rho} : A[U_1,\ldots,U_n] & \rightarrow & A[V_1,\ldots,V_n,W_1,\ldots,W_n][X,Y] \\
		U_i & \mapsto & V_iX+W_iY.
	\end{eqnarray*}
To not overload the notation, we will sometimes denote a collection of variable with its corresponding letter underlined. For instance,  $V_1,\ldots,V_n$ will be shortcut by $\underline{V}$.

Our aim is to show that the multivariate discriminant $\Disc(f_1,\ldots,f_n) \in A$ divides the bivariate discriminant $\Disc_{X,Y}(\overline{\rho}) \in  A[\underline{V},\underline{W}]$. To begin with, introduce two collections of new indeterminates $t_1,\ldots,t_n$ and $Z_1,\ldots,Z_n$, and define the matrix  
$$\varphi:=\left[\begin{array}{ccccc}
 t_n & 0 &  \cdots & 0 & Z_1 \\
 0 & t_n & \vdots & \vdots & Z_2 \\
\vdots  & \ddots & \ddots & 0  & \vdots \\
0 &\cdots & 0 & t_n & Z_{n-1} \\
-t_1 & -t_2& \cdots &-t_{n-1} & Z_n 
\end{array}\right].$$
Applying the base change formula for the resultant \cite[\S 5.12]{J91}, we get 
\begin{multline} \label{eq:thetaZ}
\theta_Z:=	\Res(f_1\circ \varphi, \ldots, f_{n-1}\circ \varphi, (\sum_{i=1}^n U_iX_i)\circ\varphi) \\ 
= \det(\varphi)^{d_1\ldots d_{n-1}}\theta(\underline{U}) = t_n^{(n-2)d_1\ldots d_{n-1}}\left(\sum_{i=1}^n t_iZ_i\right)^{d_1\ldots d_{n-1}}\theta(\underline{U})
\end{multline}
in the extended ring $A[\underline{U},\underline{t},\underline{Z}]$. Now, set $f_n:=\sum_{i=1}^n U_iX_i$. Having in mind to use Corollary \ref{cor:mertenshom}, we need to identify for all $i,j=1,\ldots,n$ the coefficient, say $V_{i,j}$, of the monomial $X_jX_n^{d_i-1}$ in the polynomial $f_i\circ\varphi$. The coefficients $V_{i,n}$ are easily seen to be equal to $f_i(Z_1,\ldots,Z_n)$ since one only has to evaluate $f_i\circ\varphi$ at $X_1=\cdots=X_{n-1}=0$ and $X_n=1$. Then, to get the coefficients $V_{i,j}$ with $j\neq n$, we have to differentiate $f_i\circ\varphi$ with respect to $X_j$ and finally evaluate the result at $X_1=\cdots=X_{n-1}=0$ and $X_n=1$; we find
$$V_{i,j}=t_n\frac{\partial f_i}{\partial X_j}(Z_1,\ldots,Z_n) - t_j\frac{\partial f_i}{\partial X_n}(Z_1,\ldots,Z_n), \ j\neq n.$$
We claim that
\begin{equation}\label{eq:Dc}
\mathcal{D}:=\det\left(V_{i,j} \right)_{i,j=1,\ldots n}=\left(\sum_{i=1}^n U_iZ_i\right) t_n^{n-2}\Delta_t \textrm{ mod } (f_1(\underline{Z}),\ldots,f_{n-1}(\underline{Z}))	
\end{equation}
in $A[\underline{U},\underline{t},\underline{Z}]$, where $\Delta_t$ stands for the Jacobian matrix 
$$\Delta_t:=\frac{\partial (f_1,\ldots,f_{n-1},\sum_{i=1}^{n}t_iX_i)}{\partial (X_1,\ldots,X_n)}(Z_1,\ldots,Z_n) \in A[\underline{t},\underline{Z}].$$
Indeed, from the definition, it is easy to see that
\begin{multline*}
	\mathcal{D}=(\sum_{i=1}^n U_iZ_i) \left| 
	\begin{array}{ccc}
	t_n\frac{\partial f_1}{\partial X_1}(\underline{Z})-t_1\frac{\partial f_1}{\partial X_n}(\underline{Z}) & \cdots & t_n\frac{\partial f_1}{\partial X_{n-1}}(\underline{Z})-t_{n-1}\frac{\partial f_{1}}{\partial X_n}(\underline{Z})	\\
	\vdots & & \vdots \\
	t_n\frac{\partial f_{n-1}}{\partial X_1}(\underline{Z})-t_1\frac{\partial f_{n-1}}{\partial X_n}(\underline{Z}) & \cdots & t_n\frac{\partial f_{n-1}}{\partial X_{n-1}}(\underline{Z})-t_{n-1}\frac{\partial f_{n-1}}{\partial X_n}(\underline{Z})
	\end{array}
	\right| \\
	\textrm{ mod } (f_1(\underline{Z}),\ldots,f_{n-1}(\underline{Z})).	
\end{multline*}
Denote by $M$ the determinant appearing in this equality. Then, it is clear that
$$t_nM= 
\left| 
\begin{array}{cccc}
t_n\frac{\partial f_1}{\partial X_1}(\underline{Z}) & \cdots & t_n\frac{\partial f_1}{\partial X_{n-1}}(\underline{Z}) & \frac{\partial f_{1}}{\partial X_n}(\underline{Z})	\\
\vdots & & \vdots & \vdots \\
t_n\frac{\partial f_{n-1}}{\partial X_1}(\underline{Z}) & \cdots & t_n\frac{\partial f_{n-1}}{\partial X_{n-1}}(\underline{Z}) & \frac{\partial f_{n-1}}{\partial X_n}(\underline{Z}) \\
t_nt_1 & \cdots & t_nt_{n-1} & t_n
\end{array}
\right|
=t^{n-1}\Delta_t$$  
and \eqref{eq:Dc} is proved. Therefore, by Corollary \ref{cor:mertenshom} there exists $H_1\in A[\underline{U},\underline{t},\underline{Z}]$ such that 
$$\theta_Z-\mathcal{D}H_1 \in (f_1(\underline{Z}),\ldots,f_{n-1}(\underline{Z}),\sum_{i=1}^n U_iZ_i)^2$$
and hence, using \eqref{eq:thetaZ} and \eqref{eq:Dc}, we obtain that
\begin{multline*}
	t_n^{(n-2)d_1\ldots d_{n-1}}\left(\sum_{i=1}^n t_iZ_i\right)^{d_1\ldots d_{n-1}}\theta(\underline{U}) \in \\
	\left(f_1(\underline{Z}),\ldots,f_{n-1}(\underline{Z}),\left(\sum_{i=1}^n U_iZ_i\right) t_n^{n-2}\Delta_t,\left(\sum_{i=1}^n U_iZ_i\right)^2\right).
\end{multline*}
Applying the operator $\sum_{i=1}^n t_i\partial(-)/\partial U_i$, we get
\begin{multline*}
	t_n^{(n-2)d_1\ldots d_{n-1}}\left(\sum_{i=1}^n t_iZ_i\right)^{d_1\ldots d_{n-1}}\sum_{i=1}^n t_i \theta_i(\underline{U}) \in \\
	\left(f_1(\underline{Z}),\ldots,f_{n-1}(\underline{Z}),\left(\sum_{i=1}^n t_iZ_i\right) t_n^{n-2}\Delta_t,\left(\sum_{i=1}^n U_iZ_i\right)\right).	
\end{multline*}
Now, we send this relation through the morphism $\rho$ and substitute $\underline{X}$ to $\underline{Z}$. It turns out that $\sum_{i=1}^n U_iZ_i$ is sent to zero and hence we obtain that
$$t_n^{(n-2)d_1\ldots d_{n-1}}\left(\sum_{i=1}^n t_iZ_i\right)^{d_1\ldots d_{n-1}}\sum_{i=1}^n t_i\rho(\theta_i)(\underline{Z}) \in
\left(f_1(\underline{Z}),\ldots,f_{n-1}(\underline{Z}),\Delta_t\right).$$
By the divisibility property of the resultant \cite[\S 5.6]{J91}, we deduce that the resultant $\Res(f_1,\ldots,f_{n-1},\Delta_t)$ divides
$$\Res\left(f_1,\ldots,f_{n-1},t_n^{(n-2)d_1\ldots d_{n-1}}\left(\sum_{i=1}^n t_iZ_i\right)^{d_1\ldots d_{n-1}}\sum_{i=1}^n t_i\rho( \theta_i)(\underline{Z})\right).$$
But by definition, 
$$\Res(f_1,\ldots,f_{n-1},\Delta_t)=\Disc(f_1,\ldots,f_{n-1})\Res(f_1,\ldots,f_{n-1},\sum_{i=1}^n t_iX_i)$$
and by the second Mertens' formula and the multiplicativity property of the resultant we have
\begin{multline*}
	\Res\left(f_1,\ldots,f_{n-1},t_n^{(n-2)d_1\ldots d_{n-1}}\left(\sum_{i=1}^n t_iZ_i\right)^{d_1\ldots d_{n-1}}\sum_{i=1}^n t_i\rho(\theta_i)(\underline{Z})\right) = \\
	(-1)^{d_1\ldots d_{n-1}}t_n^{(n-2)d_1^2\ldots d_{n-1}^2} \Disc_{X,Y}(\overline{\rho}(\theta))\Res(f_1,\ldots,f_{n-1},\sum_{i=1}^n t_iX_i)^{d_1\ldots d_{n-1}+1}.
\end{multline*}
Since $\Disc(f_1,\ldots,f_{n-1})$ and $\Res(f_1,\ldots,f_{n-1},\sum_{i=1}^n t_iX_i)$ are coprime (the latter is irreducible and depends on $\underline{t}$ which is not the case of the discriminant) we deduce that there exists $H\in A[\underline{V},\underline{W}]$ such that
$$\Disc_{X,Y}(\overline{\rho}(\theta))=H\Disc(f_1,\ldots,f_n).$$

To finish the proof, we will show that $H$ and $\Disc(f_1,\ldots,f_n)$ are coprime, so that $p$ must be equal to 2 since $\Disc(\overline{\rho}(\theta))$ is a square, as a specialization of a square. For that purpose, we proceed as in the proof 
of Lemma A (in the appendix): we specialize each polynomial $f_i$, $i=1,\ldots,n-1$ to the product of $d_i$ generic linear forms 
	$$l_{i,j}:=U_{i,j,1}X_1+U_{i,j,1}X_2+\cdots+U_{i,j,n}X_n=\sum_{r=1}^{d_i}U_{i,j,r}X_r, \ \ i=1,\ldots,n, \ j=1,\ldots,d_i.$$ 
	After this specialization, we get (see the proof of Lemma A)
\begin{multline*}
	\Disc(\overline{\rho}(\theta))= \\ \pm\prod_{\lambda<\mu} \left(
	\Delta_\lambda(V_1,\ldots,V_n) \Delta_\mu(W_1,\ldots,W_n)-\Delta_\lambda(W_1,\ldots,W_n)\Delta_\mu(V_1,\ldots,V_n)
	\right)^2.		
\end{multline*}
On the other hand, Proposition \ref{det2} yields
\begin{equation*}
    \Disc\left(\prod_{j=1}^{d_1}l_{1,j},\ldots,\prod_{j=1}^{d_{n-1}}l_{n-1,j}\right)=\pm \prod_{I}
    \det(l_{1,j_1},l_{2,j_2},\ldots,l_{n-1,j_{n-1}},l_{i,j})^2.
    \end{equation*}
	Moreover, if $\lambda=(j_1,\ldots,j_{n-2},j_{n-1}')$ and $\mu=(j_1,\ldots,j_{n-2},j_{n-1}'')$ then we have the equality
	\begin{multline*}
	\Delta_\lambda(V_1,\ldots,V_n) \Delta_\mu(W_1,\ldots,W_n)-\Delta_\lambda(W_1,\ldots,W_n)\Delta_\mu(V_1,\ldots,V_n) = \\
		\det(l_{1,j_1},l_{2,j_2},\ldots,l_{n-1,j_{n-1}'},l_{n-1,j_{n-1}''})\times \det(l_{1,j_1},l_{2,j_2},\ldots,V,W)
	\end{multline*}
(it is easy to check this formula in the case $n=2$; then the general case can be deduced from this by developing each determinant in this equality with respect to their two last columns). 	
	Therefore, $H$ and $\Disc(f_1,\ldots,f_n)$ are coprime. So we have proved that $D=c.P^2$ under the assumptions $k$ is a UFD and $2=0$ in $k$.
	
\medskip

Now, assume that $k$ is a domain such that $2=0$, and set $F:=\ZZ/2\ZZ$ for simplicity. The injective map $F\hookrightarrow k$  is flat 
for $k$ is a torsion-free $F$-module ($k$ is not the trivial ring). 
Therefore, the canonical exact sequence (see Corollary \ref{cor:TFprimedim1}) 
$$0 \rightarrow {}_F\TF_\mm(\Dc) \rightarrow {}_F C \rightarrow {}_F B_{X_n}$$ 
remains exact after tensorization by $k$ over $F$. Since ${}_F C\otimes_F k\simeq {}_kC$ and ${}_F B_{X_n}\otimes_F k \simeq {}_k B_{X_n}$ we deduce that
\begin{equation*}
	{}_k\TF_\mm(\Dc)\simeq {}_F\TF_\mm(\Dc)\otimes_F k
\end{equation*}
and hence that ${}_k \pp\simeq {}_F \pp \otimes_F k$. Moreover, $F$ is a UFD and hence we have proved that ${}_F D=P^2$ where $P$ is 
a prime element that generates ${}_F \pp$ (observe that the unit $c$ is necessarily equal to $1$ in $F$). Considering the specialization $\rho : {}_F A \rightarrow {}_k A$, it follows that $\rho(P)$ generates ${}_k \pp$ and ${}_k D=\rho({}_F D)=\rho(P)^2$ (by definition of the discriminant) and this concludes the proof of this theorem.
\end{proof}


Before closing this section, we give a refined relationship for the discriminant. Let $R$ be a commutative ring and suppose given $f_1,\ldots,f_{n-1}$
homogeneous polynomials in $R[X_1,\ldots,X_n]$ of respective positive degree
$d_1,\ldots,d_{n-1}$. Recall the notation $\tilde{f}_i(X_1,\ldots,X_{n-1}):=f_i(X_1,\ldots,X_{n-1},1) \in
R[X_,\ldots,X_{n-1}]$ (and similarly for $\tilde{J}_n$). An immediate consequence of the proof of Lemma  \ref{lemcor} (see 
\eqref{eq:firstrelation}) is that 
\begin{equation*}
  \Disc(f_1,\ldots,f_{n-1}) \in
  (\tilde{f}_1,\ldots,\tilde{f}_{n-1},\tilde{J}_n) \subset
  A[X_1,\ldots,X_{n-1}]. \end{equation*} 
The following theorem, which appears in \cite{ApJo} for the case $n=2$, improves this result.

\begin{thm}
  With the above notation we have 
  $$\Disc(f_1,\ldots,f_{n-1}) \in R \cap 
  \left(\tilde{f}_1,\ldots,\tilde{f}_{n-1},
 \tilde{J_n}^2\right) \subset R[X_1,\ldots,X_{n-1}].$$
\end{thm}
\begin{proof} As always, it is sufficient to prove this theorem in the
  generic case of Section \ref{sec:def}; $f_1,\ldots,f_{n-1}$ are
  supposed to be homogeneous polynomials in $A[X_1,\ldots,X_n]$, where
  $A$ is the universal coefficient ring, of respective positive degree
  $d_1,\ldots,d_{n-1}$. We recall that $J_n$ denotes the Jacobian
  determinant $|\frac{\partial (f_1,\ldots,f_{n-1}) }{\partial
    (X_1,\ldots,X_{n-1}) }|$ and that for any polynomial $P$ in
  $X_1,\ldots,X_n$ we denote by $\tilde{P}$ (resp.~$\overline{P}$) the
  polynomial in $X_1,\ldots,X_{n-1}$ obtained by substituting $X_n$ by
  1 (resp.~0) in $P$.

  Let us introduce the new indeterminates $T_1,\ldots,T_n$ and,
  setting $\delta:=\deg(J_n)=\sum_{i=1}^{n-1}(d_i-1)$, 
  consider both resultants
 $$\rho:=\Res(f_1,\ldots,f_{n-1},X_n)=\Res(\overline{f}_1,\ldots,\overline{f}_{n-1}) \in A,$$
  $$\mathcal{R}:=\Res(f_1-T_1X_n^{d_1},\ldots,f_{n-1}-T_{n-1}X_n^{d_{n-1}},J_n-T_nX_n^\delta)
  \in A[T_1,\ldots,T_n].$$
Since the $f_i$'s are generic polynomials, we know that $\rho$ is an
irreducible element in $A$ generating the inertia forms ideal 
$$\mathcal{T}:=\TF_{(X_1,\ldots,X_{n-1})}(\overline{f}_1,\ldots\overline{f}_{n-1})_0=\TF_{(X_1,\ldots,X_n)}(f_1,\ldots,f_{n-1},X_n)_0
\subset A.$$
From Lemma \ref{lem1}, ii) (take $i=1,\ldots,n-1$ and $j=n$), we deduce
that $J_n \in \mathcal{T}$. Consequently, polynomials $
f_1-T_1X_n^{d_1},\ldots,f_{n-1}-T_{n-1}X_n^{d_{n-1}}$ and $J_n-T_nX_n^\delta$
are in $\mathcal{T}\otimes_AA[T_1,\ldots,T_n]$ and it follows that
$\mathcal{R}$ itself is in $\mathcal{T}\otimes_AA[T_1,\ldots,T_n]$. This implies that 
$\rho$ divides $\mathcal{R}$: there exists $H(T_1,\ldots,T_n)\in
A[T_1,\ldots,T_n]$ such that 
$$\mathcal{R}=\rho H(T_1,\ldots,T_n) \in A[T_1,\ldots,T_n].$$
This polynomial $H$ have the two following important properties:
\begin{itemize}
\item $H(0,\ldots,0)=\Disc(f_1,\ldots,f_{n-1}) \in A$  (by
  \eqref{defeq}),
  \item $H(\tilde{f}_1,\ldots,\tilde{f}_{n-1},\tilde{J}_n)=0 \in
    A[X_1,\ldots,X_{n-1}]$ (by \eqref{restilde}). 
  \end{itemize}
Therefore $H(T_1,\ldots,T_n)$ gives (similarly to
\eqref{restilde} for the resultant)
an \emph{explicit expression} of the
  discriminant of $f_1,\ldots f_{n-1}$ as a polynomial in
  $\tilde{f}_1,\ldots,\tilde{f}_{n-1},\tilde{J}_n$ with coefficients
  in $A$ and without constant term; in other words as an element in
  $\tilde{J}_nA[\tilde{f}_1,\ldots,\tilde{f}_{n-1},\tilde{J}_n] + \sum_{i=1}^{n-1}
  \tilde{f}_iA[\tilde{f}_1,\ldots,\tilde{f}_{n-1},\tilde{J}_n].$
  We claim that the coefficient of $H$ (seen as a polynomial in the
  $T_i$'s) of the monomial $T_n$ is zero, and this implies our theorem.

  To prove this claim, it is sufficient to prove the same claim for
  $\mathcal{R} \in A[T_1,\ldots,T_n]$, and even, by performing the
  specialization (which leaves $J_n$ invariant)
  $$f_i \mapsto f_i + T_iX_n^{d_i} \text{ for all } i=1,\ldots,n-1,$$
  for the resultant 
  $$\Res(f_1,\ldots,f_{n-1},J_n-T_nX_n^\delta) \in A[T_n].$$
Let $K$ be the quotient field of $A$ and $\overline{K}$ its algebraic
closure. Then the $f_i$'s have $d_1\ldots d_{n-1}$ simple roots, none
at infinity, in $\mathbb{P}_{\overline{K}}^{n-1}$. As in
the proof of Proposition \ref{poi}, the Poisson's formula gives
$$\frac{\Res(f_1,\ldots,f_{n-1},J_n-T_nX_n^\delta)}
{\Res(\overline{f}_1,\ldots,\overline{f}_{n-1})^\delta}=\prod_{\xi \in I}(\tilde{J}_n(\xi)-T_n),$$
where $I:=\{\xi \in \mathbb{A}^{n-1}_{\overline{K}} :
  f_1(\xi)=\cdots=f_{n-1}(\xi)=0\}$.
But the coefficient of $T_n$, up to a nonzero multiplicative constant,
equals
$$\left(\prod_{\xi \in I}\tilde{J}_n(\xi)\right).\left(\sum_{\xi \in I}\frac{1}{\tilde{J}_n(\xi)}\right).$$
This latter quantity vanishes since its second factor is zero by the well known Jacobi formula.
\end{proof}

\begin{rem}
	Observe that we actually proved that
	  $$\Disc(f_1,\ldots,f_{n-1}) \in \tilde{J_n}^2 A[\tilde{f}_1,\ldots,\tilde{f}_{n-1},\tilde{J}_n] + \sum_{i=1}^{n-1}
	  \tilde{f}_iA[\tilde{f}_1,\ldots,\tilde{f}_{n-1},\tilde{J}_n].$$	
\end{rem}


\subsection{The base change formula}

In this section, we  investigate
 the behavior of the discriminant
  of $n-1$ homogeneous polynomials in $n$ variables
  under polynomial compositions. Although the situation is much more
  involved compared to the case of the resultant \cite[\S 5.12]{J91}, we provide a detailed 
  base change formula. We begin with the case of a linear change of coordinates.   

  \begin{prop}\label{chgcoor} Let $R$ be a commutative ring and  $f_i$
  ($i=1,\ldots,n-1$) be a homogeneous polynomial of degree $d_i\geq 1$
  in $R[X_1,\ldots,X_n]$. Given a matrix $\varphi=\left[c_{i,j} \right]_{1\leq
        i,j \leq n}$ with entries in $R$ and denoting, for all $f\in R[X_1,\ldots,X_n]$,
      $$f\circ
      \varphi(X_1,\ldots,X_n):=f\left(c_{1,1}X_1+\cdots+c_{1,n}X_n,\ldots,\sum_{j=1}^nc_{i,j}X_j,\ldots,\sum_{j=1}^nc_{n,j}X_n\right),$$ 
      we have
      $$\Disc(f_1\circ\varphi,\ldots,f_{n-1}\circ\varphi)=\det(\varphi)^{d_1\ldots
        d_{n-1}(\sum_{i=1}^{n-1}(d_i-1))} \Disc(f_1,\ldots,f_{n-1}).$$
    \end{prop}

    \begin{proof} We prove this proposition in the generic case. By
      Definition \ref{defdisc}, we have
\begin{multline*}
	     \Res(f_1\circ\varphi,\ldots,f_{n-1}\circ\varphi,X_n\circ\varphi)\Disc(f_1\circ\varphi,\ldots,f_{n-1}\circ\varphi)\\
	=\Res(f_1\circ\varphi,\ldots,f_{n-1}\circ\varphi,J_n(f\circ \varphi)).	
\end{multline*}
Now, since $J_n(f\circ
\varphi)=J_n(f_1,\ldots,f_{n-1})\circ [\varphi].\det(\varphi)$ (the
classical formula for changing variables), we deduce from \cite[\S 5.12]{J91} and the homogeneity of the resultant that the numerator of the previous display is equal to
$$\det(\varphi)^{d_1\ldots d_{n-1}}\Res(f_1,\ldots,f_{n-1},J_n)
\det(\varphi)^{d_1\ldots d_{n-1}\sum_{i=1}^{n-1}(d_i-1)},$$
and the denominator is equal to
$$\Res(f_1,\ldots,f_{n-1},X_n)\det(\varphi)^{d_1,\ldots,d_{n-1}}.$$
The result follows by simplifying $
\det(\varphi)^{d_1,\ldots,d_{n-1}}$ in both previous equalities.
\end{proof}

    \begin{cor} Take again the notation of \S
      \ref{defsubsec}.  Let $m$ be a fixed integer in $\{1,\ldots,n \}$ and
       define a grading on the ring ${}_k A=k[U_{i,\alpha} \,|\, |\alpha|=d_i]$ by
      $$\mathrm{weight}(U_{i,\alpha_1,\ldots,\alpha_n}):=\alpha_m.$$
      Then $\Disc(f_1,\ldots,f_{n-1}) \in {}_k A$ is homogeneous of total weight
$$d_1\ldots d_{n-1}\sum_{i=1}^{n-1} (d_i-1).$$      
    \end{cor}
    \begin{proof} It is an immediate corollary of Proposition \ref{chgcoor} by
      taking  the diagonal matrix
      $\varphi=[c_{i,j}]$ where $c_{m,m}=t$,     where   $t$ be a new
      indeterminate,
      and $c_{i,i}=1$ if $i\neq m$. 
    \end{proof}

We now turn to the general situation.

  \begin{prop}\label{prop:basechange} For all $i=1,\ldots,n-1$, let  $f_i$
  be a homogeneous polynomial of degree $d_i\geq 1$
  in $R[X_1,\ldots,X_n]$, where $R$ is a commutative ring.
  If  $g_1,\ldots,g_n$ are  $n$ homogeneous polynomials of the
    same degree $d\geq 2$ in $R[X_1,\ldots,X_n]$ then, denoting 
    $f_i\circ g:=f_i(g_1,\ldots,g_n)$ for
    all $i=1,\ldots,n-1$,  we have
    \begin{multline}
      d^{d^{n-1}\prod_{i=1}^{n-1}d_i}
      \Disc(f_1\circ g,\ldots,f_{n-1}\circ g)
      = \Disc(f_1,\ldots,f_{n-1})^{d^{n-1}} \notag \\
      \Res(g_1,\ldots,g_n)^{d_1\ldots d_{n-1}
        ((\sum_{i=1}^{n-1}(d_i-1))-1)}
      \Res(f_1\circ g,\ldots,f_{n-1}\circ g,J(g_1,\ldots,g_n)).
      \end{multline}
  \end{prop}
  \begin{proof} As always, we assume that we are in the generic
    situation over the integers, which is sufficient to prove this formula. Let us
    introduce the polynomials $F:=U_1g_1+\cdots +U_ng_n$ which is
    homogeneous of degree $d$ in the variables $X_1,\ldots,X_n$. Then
    by \eqref{J(F)} we get
    $$d^{d^{n-1}\prod_{i=1}^{n-1}d_i}\Disc(f_1\circ
    g,\ldots,f_{n-1}\circ g)=
    \frac{\Res(f_1\circ g,\ldots, f_{n-1}\circ g,J(f\circ g,F))}
    {\Res(f_1\circ g,\ldots, f_{n-1}\circ g,F)}.$$
    But $J(f_1\circ g,\ldots,f_{n-1}\circ g,
    F)=J(f_1,\ldots,f_{n-1},\sum_{i=1}^n U_iX_i)\circ g \times
    J(g_1,\ldots,g_n)$ and $\deg(J(g_1,\ldots,g_n))=n(d-1)\geq 1$. By
    the  base change formula for the resultant \cite[\S 5.12]{J91} we deduce
    that, denoting $l:=\sum_{i=1}^n U_iX_i$ and using obvious notation,
    \begin{multline*}
      \Res(f\circ g,J(f\circ g,F))=\Res(f\circ 
        g,J(f_1,\ldots,f_{n-1},l)\circ g)\Res(f\circ g,J(g))= \\ 
        \Res(f,J(f_1,\ldots,f_{n-1},l))^{d^{n-1}}\Res(g_1,\ldots,g_n)^{d_1\ldots d_{n-1}(\sum_{i=1}^{n-1}(d_i-1))}
        \Res(f\circ g,J(g))  
\end{multline*}
and
\begin{align}
 \Res(f_1\circ g,\ldots, f_{n-1}\circ g,F) =
 \Res(f_1,\ldots,f_{n-1},l)^{d^{n-1}}
 \Res(g_1,\ldots,g_{n})^{d_1,\ldots d_{n-1}}. \notag
      \end{align}
Therefore the claimed formula follows.    
  \end{proof}

This first base change formula is not completely
factorized. Indeed, it is not hard to see   
    that $\Res(g_1,\ldots,g_n)^{d_1\ldots d_{n-1}}$ divides
    $\Res(f\circ g, J(g_1,\ldots, g_n))$ and this latter must contain other
    factors by degree evidence. Let us state this property more precisely.

\begin{lem}\label{lem:basechange} There exists a polynomial in the coefficients of the $f_i$'s and
	 the $g_i$'s, denoted $K(f,g)$, such that
	\begin{equation*}
		\Res(f\circ g, J(g_1,\ldots, g_n))=d^{d^{n-1}\prod_{i=1}^{n}d_i}\Res(g_1,\ldots,g_n)^{\prod_{i=1}^{n}d_i}K(f,g).
	\end{equation*}
\end{lem}
\begin{proof} As always, we assume that we are in the generic
    situation over the integers, which is sufficient to prove this formula. For all $i=1,\ldots,n-1$, it is clear that $f_i\circ g \in (g_1,\ldots,g_n)^{d_i}$. Moreover, we also have that $X_nJ(g_1,\ldots, g_n) \in (g_1,\ldots,g_n)$. Therefore, applying the general divisibility lemma for the resultant \cite[Proposition 6.2.1]{J91}, we deduce that
$\Res(g_1,\ldots,g_n)^{\prod_{i=1}^{n}d_i}$ divides 
$$\Res(f\circ g, X_nJ(g_1,\ldots, g_n))=\Res(f\circ g, J(g_1,\ldots, g_n))\Res(f\circ g, X_n).$$
Now, we claim that $\Res(g_1,\ldots,g_n)$ and $\Res(f\circ g, X_n)$	are relatively prime, which concludes the proof. Indeed, $\Res(g_1,\ldots,g_n)$ being irreducible, if it divides $\Res(f\circ g, X_n)$, then it must divides any specialization of this latter resultant where the $g_i$'s are left generic. So, if we specialize each polynomial $f_i$ to $X_i^{d_i}$ then this resultant specialize to $\Res(g_1,\ldots,g_{n-1},X_n)$ which is irreducible and independent of the polynomial $g_n$. Therefore, we obtain a contradiction.
\end{proof}

By gathering Proposition \ref{prop:basechange} and Lemma \ref{lem:basechange}, we are ready to give a base change formula which is completely factorized.

 \begin{thm}\label{thm:bcf} With the notation of Proposition \ref{prop:basechange} and Lemma \ref{lem:basechange}, we have
	\begin{multline}
      \Disc(f_1\circ g,\ldots,f_{n-1}\circ g)
      = \Disc(f_1,\ldots,f_{n-1})^{d^{n-1}} \notag \\
      \Res(g_1,\ldots,g_n)^{d_1\ldots d_{n-1}
        \sum_{i=1}^{n-1}(d_i-1)}
      K(f_1,\ldots,f_{n-1},g_1,\ldots,g_n).
      \end{multline}
The polynomial $K(f,g)$ is homogeneous with respect to the coefficients of the polynomials $g_1,\ldots,g_n$ of degree $$n(n-1)(d-1)d^{n-2}\prod_{i=1}^nd_i$$
and, for all $î=1,\ldots,n-1$, it is homogeneous with respect to the coefficients of the polynomial $f_i$ of degree
$$n(d-1)d^{n-2}\left(\frac{d_1\ldots d_n}{d_i}\right).$$ 
Moreover, if $k$ is a domain then $K(f_1,\ldots,f_{n-1},g_1,\ldots,g_n) \in {}_kA$ satisfies to the following properties:
\begin{itemize}
 \item[i)] $K(f,g)$ is irreducible if $2\neq 0$ in $k$,
\item[ii)] $K(f,g)$ is the square of an irreducible polynomial if $2=0$ in $k$.
 \end{itemize}
 \end{thm}
\begin{proof} The first equality follows directly from Proposition \ref{prop:basechange} and Lemma \ref{lem:basechange}. The computations of the degrees of $K$ can be deduced from this formula and the  degrees for the discriminant and the resultant. Indeed, since for all $i=1,\ldots,n-1$ the polynomial $f_i\circ g$ is homogeneous of degree $dd_i$ in the $X_i$'s, by Proposition \ref{degree-prop} we deduce that $\Disc(f_1\circ g,\ldots,f_{n-1}\circ g)$ is homogeneous of degree 
$$D_i:=d^{n-2}\frac{\prod_{j=1}^{n-1}d_j}{d_i}\left(
(dd_i-1)+\sum_{j=1}^{n-1}(dd_j-1)
\right)$$
with respect to the coefficients of the polynomial $f_i$ and of degree
$$D:=\sum_{i=1}^{n-1}d_iD_i=nd^{n-2}\left(\prod_{j=1}^{n-1}d_j\right)\sum_{i=1}^{n-1}(dd_i-1)$$
with respect to the coefficients of the polynomials $g_1,\ldots,g_n$. Therefore, it follows that $K$ is homogeneous with respect to the coefficients of the polynomial $f_i$ of degree
$$D_i-d^{n-1}\frac{\prod_{j=1}^{n-1}d_j}{d_i}\left(
(d_i-1)+\sum_{j=1}^{n-1}(d_j-1)
\right)=n(d-1)d^{n-2}\left(\frac{d_1\ldots d_n}{d_i}\right)$$
and is homogeneous with respect to the coefficients of the polynomials $g_1,\ldots,g_n$ of degree
$$D-nd^{n-1}\left(\prod_{i=1}^{n-1}d_i\right)\sum_{i=1}^{n-1}(d_i-1)=n(n-1)(d-1)d^{n-2}\prod_{i=1}^nd_i$$
since $\Res(g_1,\ldots,g_n)$ is homogeneous of degree $nd^{n-1}$ with respect to the coefficients of the polynomials $g_1,\ldots,g_n$.

\medskip

Now, we turn to the proof of the irreducibility of $K$. First we observe that it is sufficient to prove the claimed properties in $\Frac(k)$ so that we will always work in a UFD. We begin with the case where $2\neq 0$ in $k$. We will proceed by induction on the integer $r=d_1+d_2+\cdots+d_{n-1}$. The difficult point is actually to prove this irreducibility property for $r=n-1$, that is to say for the case $d_1=\cdots=d_{n-1}=1$. Indeed, let us assume this for a moment and suppose that $r>n-1$. Then,  at least one of the degree $d_i$ is greater or equal to 2 and we can assume without loss of generality that it is $d_1$ by permuting the $f_i$'s if necessary. Consider the specialization that sends $f_1$ to the product of a generic form $l$ and a generic polynomial $f_1'$ of degree $d_1-1$. Lemma \ref{lem:basechange} implies that $K$ has a multiplicativity property with respect to the polynomial $f_1,\ldots,f_n$, so that this specialization sends $K(f_1,\ldots,f_n)$ (we omit the $g_i$'s in the notation for simplicity) to the product $K(l,f_2,\ldots,f_n)K(f_1',f_2,\ldots,f_n).$  Now, if $K$ is reducible then all its irreducible factors depending on the polynomial $f_1$ must depend on $l$ and $f_1'$ after the above specialization. Therefore, since $K(l,f_2,\ldots,f_n)$ are $K(f_1',f_2,\ldots,f_n)$ are both irreducible by our inductive hypothesis and distinct, we deduce that $K(f_1,\ldots,f_n)$ is also irreducible. 

So, it remains to prove that $K$ is irreducible in the case $d_1=\cdots=d_{n-1}=1$. Set $f_i=\sum_{j=1}^n U_{i,j}X_j$ for all $i=1,\ldots,n-1$, introduce new indeterminates $W_1,\ldots,W_n$ and define the determinant
$$\Lambda:=\left| 
\begin{array}{cccc}
	U_{1,1} & U_{1,2} & \cdots & U_{1,n} \\
	U_{2,1} & U_{2,2} & \cdots & U_{2,n} \\
	\vdots  & \vdots &  & \vdots \\
	U_{n-1,1} & U_{n-1,2} & \cdots & U_{n-1,n} \\
	W_1 & W_2 & \cdots & W_n
\end{array}
\right|.$$
By \eqref{J(F)} and the covariance property of resultants \cite[\S 5.11]{J91}, we have 
\begin{multline*}
	\Res\left(f_1\circ g, \ldots, f_{n-1}\circ g,J\left(f_1\circ g, \ldots, f_{n-1}\circ g,\sum_{i=1}^n W_ig_i\right)\right) \\
	=d^{d^{n-1}}	\Res\left(f_1\circ g, \ldots, f_{n-1}\circ g,\sum_{i=1}^n W_ig_i\right)\Disc(f_1\circ g, \ldots, f_{n-1}\circ g) \\
	= d^{d^{n-1}} \Lambda^{d^{n-1}}\Res(g_1,\ldots,g_n) \Disc(f_1\circ g, \ldots, f_{n-1}\circ g).
\end{multline*}
On the other hand, since $\Lambda.J(g_1,\ldots,g_n)=J(f_1\circ g, \ldots, f_{n-1}\circ g,\sum_{i=1}^n W_ig_i)$ we obtain that
\begin{align*}
	\lefteqn{\Res\left(f_1\circ g, \ldots, f_{n-1}\circ g,J\left(f_1\circ g, \ldots, f_{n-1}\circ g,\sum_{i=1}^n W_ig_i\right)\right)} \\
	& =\Res\left(f_1\circ g, \ldots, f_{n-1}\circ g,\Lambda.J(g_1,\ldots,g_n)\right) \\
	& = \Lambda^{d^{n-1}}\Res(f_1\circ g, \ldots, f_{n-1}\circ g,J(g_1,\ldots,g_n)) \\
	& =  \Lambda^{d^{n-1}}d^{d^{n-1}}\Res(g_1,\ldots,g_n)K(f_1,\ldots,f_n,g_1,\ldots,g_n).
 \end{align*}
where the last equality follows from Lemma \ref{lem:basechange}. Therefore, by comparison of these two computations (in the generic case over the integers and then by specialization) we deduce that
\begin{equation}\label{eq:KDisc}
	K(f_1,\ldots,f_n,g_1,\ldots,g_n)=\Disc(f_1\circ g, \ldots, f_{n-1}\circ g)
\end{equation}
under our assumption $d_1=\cdots=d_{n-1}=1$. In order to show that this discriminant is irreducible, we will compare several specializations. 

We begin with the specialization of the polynomials $f_1,\ldots,f_{n-1}$ to $X_1,\ldots,X_{n-1}$ respectively. Under this specialization, the polynomial $f_i\circ g$ is sent to $g_i$ for all $i=1,\ldots,n-1$ and hence $\Disc(f_1\circ g,\ldots,f_{n-1}\circ g)$ is sent to $\Disc(g_1,\ldots,g_{n-1})$ which is known to be an irreducible polynomial in the coefficients of the polynomials $g_1,\ldots,g_{n-1}$ by Theorem \ref{thm:2<>0}.  It follows that if $\Disc(f_1\circ g,\ldots,f_{n-1}\circ g)$ is reducible, then necessarily there exists a non constant and irreducible polynomial $P(U_{i,j})$ which is independent of the coefficients of the polynomials $g_1,\ldots,g_n$ and that divides $\Disc(f_1\circ g,\ldots,f_{n-1}\circ g)$.

Now, consider the specialization that sends the polynomial $g_n$ to $0$. Then, $\Disc(f_1\circ g,\ldots,f_{n-1}\circ g)$ is sent to 
\begin{multline*}
	\Disc\left(\sum_{i=1}^{n-1}U_{1,j}g_j,\ldots,\sum_{i=1}^{n-1}U_{n-1,j}g_j \right)=\\
	\left| 
	\begin{array}{ccc}
		U_{1,1} & \cdots & U_{1,n-1} \\
		\vdots & & \vdots \\
		U_{n-1,1} & \cdots & U_{n-1,n-1}
	\end{array}
	\right|^{n(d-1)d^{n-2}}\times \Disc(g_1,\ldots,g_{n-1})	
\end{multline*}
where the equality holds by the covariance property given in Proposition \ref{prop:covariance}. We deduce that $P(U_{i,j})$ is equal to the determinant of the matrix $(U_{i,j})_{1\leq i,j\leq n-1}$ up to multiplication by an invertible element in $k$. But if we consider the specialization that sends the polynomial $g_1$ to $0$, then by a similar argument we get that 
$\Disc(f_1\circ g,\ldots,f_{n-1}\circ g)$ is sent to 
\begin{multline*}
	\Disc\left(\sum_{i=2}^{n}U_{1,j}g_j,\ldots,\sum_{i=2}^{n}U_{n-1,j}g_j \right)=\\
	\left| 
	\begin{array}{ccc}
		U_{1,2} & \cdots & U_{1,n} \\
		\vdots & & \vdots \\
		U_{n-1,2} & \cdots & U_{n-1,n}
	\end{array}
	\right|^{n(d-1)d^{n-2}}\times \Disc(g_2,\ldots,g_{n}).	
\end{multline*}
Therefore, we deduce that $P(U_{i,j})$ should also be equal to the determinant of the matrix $(U_{i,j})_{1\leq i,j\leq n-1}$ up to multiplication by an invertible element in $k$ and hence we get a contradiction. This concludes the proof of the irreducibility of $K$ when $2\neq 0$ in $k$. 

\medskip

Now, we turn to the proof that $K$ is the square of an irreducible polynomial under the assumption $2= 0$ in $k$.
By Theorem \ref{thm:2=0}, the discriminant is the square of a polynomial, irreducible in the generic case, that we will denote by $\Delta$. Now, define the polynomial $\chi$ by the equality
  	\begin{multline}
    \Delta(f_1\circ g,\ldots,f_{n-1}\circ g)
    = \Delta(f_1,\ldots,f_{n-1})^{d^{n-1}} \notag \\
    \Res(g_1,\ldots,g_n)^{\frac{1}{2}d_1\ldots d_{n-1}
      \sum_{i=1}^{n-1}(d_i-1)}
    \chi(f_1,\ldots,f_{n-1},g_1,\ldots,g_n)
    \end{multline}
so that $K(f_1,\ldots,f_{n-1},g_1,\ldots,g_n)=\chi(f_1,\ldots,f_{n-1},g_1,\ldots,g_n)^2$. To prove that $\chi$ is an irreducible polynomial we can proceed similarly to the case where $2\neq 0$ in $k$: we proceed by induction on the integer $r=d_1+\cdots+d_{n-1}\geq n-1$. Assuming for a moment that the statement holds for $r=n-1$, then the reasoning is exactly the same: $\chi$ inherits of a multiplicative property from $K$ and hence by specializing one polynomial of degree $\geq 2$, say $f_1$, to the product of a linear form $l$ and a polynomial $f_1'$ of degree $d_1-1$ then we conclude that $\chi$ is irreducible. 

To prove that $\chi$ is indeed irreducible when $d_1=\cdots=d_{n-1}=1$, we also proceed similarly to the case where $2\neq 0$. Using \eqref{eq:KDisc} that holds in the generic case other the integers, we deduce that there exists $\epsilon \in k$ such that $\epsilon^2=1$ and
$$\chi(f_1,\ldots,f_{n-1},g_1,\ldots,g_n)=\epsilon \Delta(f_1\circ g,\ldots,f_{n-1}\circ g)$$
in $k$ under the assumption $d_1=\cdots=d_{n-1}=1$. From here, we conclude that $\chi$ is irreducible by exploiting, as in the case $2\neq 0$, the three specializations $f_i\mapsto X_i$ for all $i=1,\ldots,n-1$, then $g_n\mapsto 0$ and finally $g_1\mapsto 0$, the argumentation being the same.
\end{proof}

\section{The discriminant of a hypersurface}\label{codim1}

In this section we study the discriminant of a single homogeneous polynomial in several
variables.
 
\medskip

Let $k$ be a commutative ring and $f$ be a
homogeneous polynomial of degree $d \geq 2$ in the polynomial ring
$k[X_1,\ldots,X_n]$  $(n\geq 1$). We will denote by $\partial_i f$ the partial
derivative of  $f$ with respect to the variable $X_i$. Recall the classical Euler identity  
$$df=\sum_{i=1}^n X_i\partial_i f.$$
We will also often denote by $\overline{f}$, respectively $\tilde{f}$, the polynomial
$f(X_1,\ldots,X_{n-1},0)$, respectively $f(X_1,\ldots,X_{n-1},1)$, in
$k[X_1,\ldots,X_{n-1}]$. 

We aim to study the quotient ring
$$\quotient{k[X_1,\ldots,X_n]}{(f,\partial_1 f,\cdots,\partial_{n} f)}$$ 
and its associated inertia forms of degree 0 with respect to  $\mm:=(X_1,\ldots,X_n)$. 
The geometric interpretation of the {\it generic case} over the commutative ring $k$ is the following. Let $d$ be an integer greater or equal to 2. We suppose that 
$$f(X_1,\ldots,X_n)=\sum_{|\alpha|=d} U_\alpha X^\alpha$$
and denote ${}_kA:=k[U_\alpha \,|\,  |\alpha|=d]$, ${}_kC:= {}_kA[X_1,\ldots,X_n]$ and
$${}_kB:= \quotient{{}_kC}{(f,\partial_1 f,\cdots,\partial_{n} f)}$$
The closed image of the canonical projection $\pi$ of $\Proj( _kB)$ to $\Spec(_kA)$ is defined by the ideal $H^0_\mm(_kB)_0$; roughly speaking, it parameterizes all the homogeneous forms of degree $d$ with coefficients in $k$ whose zero locus has a singular point.

\subsection{Regularity of certain sequences}
We suppose that we are in the generic case over the commutative ring $k$. We begin with two technical results. Given a sequence of elements $r_1,\ldots,r_s$ in a ring $R$, we will denote by $H_i(r_1,\ldots,r_s;R)$ the $i^{\textrm{th}}$ homology group of the Koszul complex associated to this sequence.

\begin{lem}\label{lem:Hi>2=0} For all $i\geq 2$ we have $H_i(f,\partial_1 f,\ldots,\partial_{n} f;_kC)=0$.
\end{lem}
\begin{proof} Let us emphasize some coefficients of $f$ by rewriting it as
$$f(X_1,\ldots,X_n)=g(X_1,\ldots,X_n)+\sum_{i=1}^{n}\EE_i X_iX_n^{d-1}$$
where $g \in {}_kC$. Then, it appears that the sequence $(\partial_1 f,\ldots,\partial_{n-1} f,f)$ is, in this order, regular in the ring ${}_kC_{X_n}$. Indeed, the quotient by $\partial_1 f$ amounts to express $\EE_1$ in $k[U_\alpha \,|\, U_\alpha\neq \EE_1,\ldots,\EE_n][X_1,\ldots,X_n]_{X_n}$, then the quotient by $\partial_2 f$ amounts to express $\EE_2$ and so on, to end with the quotient by $f$ that amounts to express $\EE_n$. From this property and the well known properties of the Koszul complex, it follows that 
\begin{equation*}
H_i(f,\partial_1 f,\ldots,\partial_{n} f;{}_kC)_{X_n}=0 \text{ for all } i\geq 2.
\end{equation*}
But we can argue similarly by choosing another variable $X_j$ instead of $X_n$ and therefore we actually deduce that
\begin{equation}\label{Hi>2=0}
H_i(f,\partial_1 f,\ldots,\partial_{n} f;{}_kC)_{X_j}=0 \text{ for all } i\geq 2 \text{ and } j=1,\ldots,n.
\end{equation}

Now, the consideration of the two spectral sequences 
\begin{align*} 
'E^{p,q}_1  =  H^q_\mm(K^\bullet(f,\partial_1 f,\ldots,\partial_{n} f;{}_kC)) \Longrightarrow E^n=H^n_\mm(K^\bullet(f,\partial_1 f,\ldots,\partial_{n} f;{}_kC)) \\
 ''E^{p,q}_2  =  H^p_\mm(H^q(f,\partial_1 f,\ldots,\partial_{n} f;{}_kC)) \Longrightarrow E^n=H^n_\mm(K^\bullet(f,\partial_1 f,\ldots,\partial_{n} f;{}_kC))
\end{align*}
shows that for all $i\geq 2$ we have $H_i(f,\partial_1 f,\ldots,\partial_{n} f;{}_kC)=0$, as claimed.
\end{proof}

\begin{prop}\label{prop:regular-sequences} The two following statements hold:
 \begin{itemize}
  \item[\rm(i)] For all  $i\in \{1,\ldots,n\}$ the sequence 
$(f,\partial_1 f,\ldots,\widehat{\partial_i f},\ldots,\partial_{n} f)$
is regular in the ring ${}_kC$.
\item[\rm (ii)] If $d$ is a nonzero divisor in $k$ then the sequence
$(\partial_1 f,\ldots,\partial_n f)$
is regular in the ring ${}_kC$.
 \end{itemize}
\end{prop}
\begin{proof} We prove (i) in the case $i=n$ to not overload the notation; the other cases can be treated similarly. For simplicity, we set $$K_\bullet:=K_\bullet(f,\partial_1 f,\ldots,\partial_n f;{}_kC), \ \ 
L_\bullet:=K(f,\partial_1 f,\ldots,\partial_{n-1} f;{}_kC).$$

Since $K_\bullet=L_\bullet\otimes_{{}_kC} K_\bullet(\partial_nf;{}_kC)$, we deduce, using the two spectral sequences associated to the two filtrations of a double complex having only two rows, that we have an exact sequence
$$0 \rightarrow H_0(\partial_nf;H_2(L_\bullet))\rightarrow H_2(K_\bullet) \rightarrow H_1(\partial_nf;H_1(L_\bullet)) \rightarrow 0.$$
But by Lemma \ref{lem:Hi>2=0}, we know that $H_2(K_\bullet)=0$; it follows that $\partial_n f$ is a nonzero divisor in $H_1(L)$. The homology of $L_\bullet$ is annihilated by the ideal generated by $(f,\partial_1 f,\ldots,\partial_{n-1} f)$. So, by the Euler identity we deduce that $X_n\partial_nf$ annihilates $H_1(L)$. But since we have just proved that $\partial_n f$ is a nonzero divisor in $H_1(L)$ we obtain $X_nH_1(L)=0$.

Denoting $\bar{f}(X_1,\ldots,X_{n-1}):=f(X_1,\ldots,X_{n-1},0) \in {}_kA[X_1,\ldots,X_{n-1}]$, we have the exact sequence of complexes 
$$0\rightarrow L_\bullet \xrightarrow{\times X_n} L_\bullet \rightarrow L_\bullet/X_nL_\bullet \rightarrow 0$$
where the complex $L_\bullet/X_nL_\bullet$ is nothing but the Koszul complex 
$$L_\bullet/X_nL_\bullet=K_\bullet(\bar{f},\partial_1\bar{f},\ldots,\partial_{n-1}\bar{f};{}_kA[X_1,\ldots,X_{n-1}]).$$
It follows that $H_2(L_\bullet/X_nL_\bullet)=0$ and hence the long exact sequence of homology
$$\cdots \rightarrow H_2(L_\bullet/X_nL_\bullet) \rightarrow H_1(L_\bullet) \xrightarrow{\times X_n} H_1(L_\bullet) \rightarrow H_1(L_\bullet/X_nL_\bullet) \rightarrow \cdots$$
shows that $X_n$ is a nonzero divisor in $H_1(L)$. This, with the equality $X_nH_1(L_\bullet)=0$ obtained above,  implies that $H_1(L_\bullet)=0$ which means that $(f,\partial_1 f,\ldots,\partial_{n-1} f)$ is a regular sequence in ${}_kC$.

Setting $M_\bullet:=K_\bullet(\partial_1 f, \ldots, \partial_n f;{}_kC)$, we will prove the point (ii) by showing that $H_1(M_\bullet)=0$. Since $K_\bullet=M_\bullet\otimes_{{}_kC} K_\bullet(f;{}_kC)$ we have the two  exact sequences
\begin{align}
0 \rightarrow H_0(f;H_2(M_\bullet))\rightarrow H_2(K_\bullet) \rightarrow H_1(f;H_1(M_\bullet)) \rightarrow 0  \label{Mseq1} \\ 
0 \rightarrow H_0(f;H_1(M_\bullet))\rightarrow H_1(K_\bullet) \rightarrow H_1(f;H_1(M_\bullet)) \rightarrow 0. \label{Mseq2} 
\end{align}
First, by \eqref{Hi>2=0} we know that $H_2(K_\bullet)=0$ and hence the exact sequence \eqref{Mseq1} shows that $H_1(f,H_1(M))=0$, that is to say that $f$ is a nonzero divisor in $H_1(M)$. But the Euler identity implies that $df$ annihilates $H_1(M_\bullet)$, so  $dH_1(M)=0$. 
Second, from the exact sequence of complexes
$$0\rightarrow K_\bullet \xrightarrow{\times d} K_\bullet \rightarrow K_\bullet/dK_\bullet \rightarrow 0$$ we get the long exact sequence
$$\cdots \rightarrow H_2(K_\bullet/dK_\bullet) \rightarrow H_1(K_\bullet) \xrightarrow{\times d} H_1(K_\bullet) \rightarrow \cdots$$
which shows, since $ H_2(K_\bullet/dK_\bullet)=0$, that $d$ is a nonzero divisor in $H_1(K_\bullet)$. 

Finally, the exact sequence \eqref{Mseq2} combined with the two facts $dH_1(M)=0$ and $d$ is a nonzero divisor in $H_1(K_\bullet)$, implies that  $H_0(f;H_1(M_\bullet))=0$, that is to say that the multiplication map $\times f:H_1(M_\bullet)\rightarrow H_1(M_\bullet)$ is surjective. It follows that, by composition, for any integer $m\geq 1$ the multiplication map $\times f^m:H_1(M_\bullet)\rightarrow H_1(M_\bullet)$ is also surjective. But $H_1(M)$ is a $\ZZ$-graded module and $f$ has degree $d$ for this graduation, so we have, for any $\nu \in\ZZ$ and $m\in \NN^*$, a surjective map
$$H_1(M_\bullet)_{\nu-dm} \xrightarrow{\times f^m} H_1(M)_\nu.$$
As $H_1(M)_\mu=0$ for $\mu \ll 0$ we finally get, by choosing $m \gg 0$, that $H_1(M)_\nu=0$ for all $\nu\in \ZZ$.
\end{proof}

\begin{cor}\label{ResD} For all $i\in\{1,\ldots,n\}$, the resultant  
$$\Res(\partial_1
     f,\ldots,\widehat{\partial_i f},\ldots,\partial_{n-1}f,f) \in {}_kA$$ is a primitive polynomial, hence nonzero divisor, in ${}_k A$.
   \end{cor}
   \begin{proof} This result is a consequence of Proposition \ref{prop:regular-sequences} and \cite[Proposition 3.12.4.2]{J92}. The last claim is obtained by observing that this resultant is a nonzero divisor in ${}_{\ZZ/p\ZZ} A$ for all integer $p$, which implies that it is a  primitive polynomial in ${}_\ZZ A$, hence in ${}_k A$. 
   \end{proof}

\subsection{Definition of the discriminant}

\begin{lem}\label{cd1:form1}
Let $k$ be a commutative ring and $f\in k[X_1,\ldots,X_n]$ be a
homogeneous polynomial of degree $d\geq 2$. Then, we have the
following equality in $k$:
$$d^{(d-1)^{n-1}}\Res(\partial_1 f,\ldots,\partial_{n-1}f,f)=
\Res(\partial_1f,\ldots,\partial_nf)\Res(\partial_1\overline{f},\ldots,\partial_{n-1}\overline{f}).$$
   \end{lem}
   \begin{proof} On the one hand we have, using the homogeneity of
     the resultant, 
     $$\Res(\partial_1 f, \ldots, \partial_{n-1}f,
     df)=d^{(d-1)^{n-1}}\Res(\partial_1 f,\ldots,\partial_{n-1}f,f),$$
     and on the other hand we have, using successively \cite[\S
     5.9]{J91}, \cite[\S 5.7]{J91} and \cite[Lemma 4.8.9]{J91},
     \begin{align*}
       \Res(\partial_1 f, \ldots, \partial_{n-1}f,df) & =
       \Res(\partial_1 f, \ldots, \partial_{n-1}f,X_n\partial_nf) \\
       & = \Res(\partial_1f,\ldots,\partial_nf)\Res(\partial_1 f,\ldots ,\partial_{n-1} f,X_n)\\
       & =
       \Res(\partial_1f,\ldots,\partial_nf)\Res(\partial_1\overline{f},\ldots,\partial_{n-1}\overline{f}).
       \end{align*}
 Comparing these two computations we deduce the claimed equality.    
   \end{proof}

   \begin{prop}\label{cd1:form2} Let $f(X_1,\ldots,X_n)=\sum_{|\alpha|=d}U_{\alpha}
     X^\alpha$ be the generic homogeneous polynomial of degree $d\geq
     2$ over the integers. Then the
resultant $\Res(\partial_1f,\ldots,\partial_nf)$ is divisible  by
     $d^{a(n,d)}$ in the ring ${}_\ZZ A$ where
     $$a(n,d):=\frac{(d-1)^n-(-1)^n}{d} \in \mathbb{Z}.$$     
   \end{prop}
   \begin{proof} By Corollary \ref{ResD}, we know that 
     $\Res(\partial_1f,\ldots,\partial_{n-1}f,f)$ is a primitive
     polynomial in ${}_\ZZ A$. Denoting by $c(n,d)$ the content of
     $\Res(\partial_1f,\ldots,\partial_nf)$ for all $n,d\geq 2$, Lemma 
     \ref{cd1:form1} implies that
     $$c(n,d)c(n-1,d)=d^{(d-1)^{n-1}} \ \text{ for all }  n\geq 3
     \text{ and } d\geq 2$$
     and also that $c(2,d)=d^{d-2}=d^{a(2,d)}$ for all $d\geq 2$
     (just remark that we have $\Res(dUX_1^{d-1})=dU$). Therefore, we
     can proceed by induction on $n$ to prove the claimed result:
     assume that $c(n-1,d)=d^{a(n-1,d)}$, which is true for $n=3$, then
     $$c(n,d)=d^{(d-1)^{n-1}-a(n-1,d)}=d^{a(n,d)}$$
     since it is immediate to check that $a(n-1,d)+a(n,d)=(d-1)^{n-1}$.
   \end{proof}

   We are now ready to define the discriminant of a homogeneous
   polynomial of degree $d\geq 2$. 
   \begin{defn}\label{defdiscr} Let $f(X_1,\ldots,X_n)=\sum_{|\alpha|=d}U_{\alpha}
     X^\alpha \in {}_\ZZ A$ be the generic homogeneous polynomial of degree $d\geq
     2$. The \emph{discriminant} of $f$, that will
be denoted $\Disc(f)$, is the unique element in ${}_\ZZ A$ such
that
 \begin{equation}\label{defeq1}
   d^{a(n,d)}\Disc(f)=\Res(\partial_1 f,\ldots,\partial_n
     f).
     \end{equation}
     
     Let $R$ be a commutative ring and
     $g=\sum_{|\alpha|=d}u_{\alpha}X^{\alpha}$ be a homogeneous
     polynomial of degree $d\geq 2$ in $R[X_1,\ldots,X_n]$. Then we
     define the discriminant of $g$ as $\Disc(g):=\lambda(\Disc(f))$
     where $\lambda$ is the canonical (specialization) morphism
     $\lambda : {}_\ZZ A \rightarrow R: U_\alpha \mapsto u_{\alpha}$. 
   \end{defn}


\subsection{Formal properties}  Up to a nonzero integer constant factor, the
discriminant of a homogeneous polynomial corresponds to a resultant. Consequently,
most of its properties follow from the properties of the resultant.

\begin{prop}\label{firstfp} Let $k$ be a commutative ring and $f$ be a homogeneous
  polynomial in $k[X_1,\ldots,X_n]$ of degree $d\geq 2$. 
  \begin{itemize}
  \item[\rm (i)] For all $t\in k$, we have $\Disc(tf)=t^{n(d-1)^{n-1}}\Disc(f)$.
    \item[\rm (ii)]  For all $n\geq 2$, we have the equality in $k$  
   $$\Disc(f)\Disc(\overline{f})=\Res(\partial_1 f,\ldots,\partial_{n-1}f,f).$$ 
 \end{itemize} 
\end{prop}
\begin{proof} To
  prove (i), we use the 
homogeneity of the resultant: one obtains
$$\Res(t \partial_1 f, \ldots, t \partial_{n}f)=
t^{\sum_{i=1}^n(d-1)^{n-1}}\Res(\partial_1f,\ldots,\partial_{n}f).$$
 To prove (ii), we first assume that we are in the generic
 case, that is to say that $f=\sum_{|\alpha|=d}U_\alpha X^{\alpha}$ and
 $k={}_\ZZ A:=\ZZ[U_\alpha \,|\, |\alpha|=d]$.  Using the notation of
 Proposition \ref{cd1:form2}, we have $a(n,d)+a(n-1,d)=d^{(d-1)^{n-1}}$ for all $n\geq
 3$ and $d\geq 2$. Moreover, from Definition \ref{defdiscr}, we deduce that
  $$\Res(\partial_1 f, \ldots, \partial_n f)\Res(\partial_1
  \overline{f},\ldots,\partial_{n-1}\overline{f})=d^{(d-1)^{n-1}}\Disc(f)\Disc(\overline{f}).$$
Now, comparing with Lemma \ref{cd1:form1}, we get the claimed formula in
${}_\ZZ A$ and then over any commutative ring $k$ by specialization.  
\end{proof}

\begin{rem}\label{unidisc} In the case where $n=2$ and $d$ is a nonzero divisor of $k$
  (equivalently $\mathrm{char}(k)$ does not divide $d$), the point (ii) recovers a well known formula:  
set  $f:=U_0X_1^d+U_1X_1^{d-1}X_2+\cdots+U_dX_2^d$ for simplicity, then
\begin{equation*}
U_0\Disc(f)=\Res(\partial_1f,f)=\Res(f,\partial_1f).
\end{equation*}
This follows from Definition \ref{defdiscr} since we have
$$d^{a(1,d)}\Disc(\overline{f})=d\Disc(U_0X_1^d)=\Res(dU_0X_1^{d-1})=dU_0$$ 
in ${}_k A$. 
\end{rem}

\begin{cor}\label{cd1:nzd}Let $f(X_1,\ldots,X_n)=\sum_{|\alpha|=d}U_{\alpha}
     X^\alpha \in {}_kA$ be the generic homogeneous polynomial of degree $d\geq
     2$ over the commutative ring $k$. Then $\Disc(f)$ is a primitive polynomial, hence nonzero divisor, in ${}_k A$.
\end{cor}
\begin{proof} The first claim is a combination of both Corollary \ref{ResD} and Proposition \ref{firstfp}, (ii). To prove the second claim 
 we can argue as in the proof of Corollary \ref{ResD}.
\end{proof}

We continue with some particular examples.

\begin{exmp}\label{ex:disc1} Let $h(X_1,\ldots,X_{n-1})=\sum_{|\alpha|} V_\alpha X^\alpha$ be the generic homogeneous polynomial of degree $d\geq 2$ in the variables $X_1,\ldots,X_{n-1}$ over the commutative ring $k$ and consider the homogeneous polynomial 
  $$g(X_1,\ldots,X_n)=UX_n^d+h(X_1,\ldots,X_{n-1})\in k[U, V_\alpha \,|\, |\alpha|=d][X_1,\ldots,X_n].$$
  Then, we have
  $$\Disc(g)=d^{(d-1)^{n-1}+(-1)^n}U^{(d-1)^{n-1}}\Disc(h)^{d-1}.$$
\end{exmp}
\begin{proof} Notice that without loss of generality, it is enough to prove this formula in the case $k=\ZZ$. Now, since $\partial_n g=dUX_n^{d-1}$ and $\partial_i g=\partial_i h$ for $i=1,\ldots,n-1$, we deduce that
	$$\Res(\partial_1 g,\ldots,\partial_n g)= (dU)^{(d-1)^{n-1}} \Res(\partial_1 h,\ldots, \partial_{n-1} h)^{d-1}.$$
	Therefore, from the definition of the discriminant we get
	$$d^{a(n,d)}\Disc(g)=(dU)^{(d-1)^{n-1}} d^{(d-1)a(n-1,d)}\Disc(h)^{d-1}$$
	and the claimed formula follows from a straightforward computation.
\end{proof}

\begin{exmp}[\mbox{\cite{Bou}}] Consider the homogeneous polynomial of degree $d\geq 2$
  $$g(X_1,\ldots,X_n)=A_1X_1^d+\cdots+A_nX_n^d \in \ZZ[A_1,\ldots,A_n][X_1,\ldots,X_n].$$
  Then, its discriminant consists of only one monomial; more precisely,
  $$\Disc(g)=d^{n(d-1)^{n-1}-a(n,d)}(A_1A_2\ldots A_n)^{(d-1)^{n-1}}
  \in \ZZ[A_1,\ldots,A_n].$$
\end{exmp}
\begin{proof}
Indeed, since $\partial_ig=dA_iX_i^{d-1}$ for all $i=1,\ldots,n$, from the 
classical properties of the resultant we get
\begin{align}
\Res(\partial_1g,\ldots,\partial_ng) &= 
d^{n(d-1)^{n-1}}\Res(A_1X_1^{d-1},\cdots,A_nX_n^{d-1}) \notag \\ 
&=  d^{n(d-1)^{n-1}}(A_1A_2\ldots A_n)^{(d-1)^{n-1}}. \notag  
\end{align}
The claimed result follows by comparing this equality with
\eqref{defeq1}.
\end{proof}
\begin{exmp}[\mbox{\cite{Bou}}]\label{ex:disc3}  Consider the homogeneous polynomial of degree $d\geq 2$
  $$g(X_1,\ldots,X_n)=X_1^d+UX_1X_2^{d-1}+X_2X_3^{d-1}+\cdots+X_{n-1}X_n^{d-1} \in
  \ZZ[U][X_1,\ldots,X_n].$$
  Then, its discriminant contains only one monomial modulo
  $d$. More precisely,
  $$\Disc(g)=U^{(d-1)^{n-1}+(-1)^n} \mod (d) \in \ZZ[U].$$
\end{exmp}
\begin{proof}
To prove this formula, we proceed by induction on the number $n$ of
variables. So, assume first that $n=2$. We have $g=X_1+UX_1X_2^{d-1}$
and we easily compute in the ring $\ZZ[U]$
\begin{align}
  \Res(\partial_1 g,\partial_2 g) &=
  \Res(dX_1^{d-1}+UX_2^{d-1},(d-1)UX_1X_2^{d-2}) \notag \\
  &=(d-1)^{d-1}U^{d-1}\Res(UX_2^{d-1},X_1)\Res(dX_1^{d-1},X_2^{d-2}) \notag \\
  &= (-1)^{d-1}(d-1)^{d-1}d^{d-2}U^{d}.
\end{align}
From \eqref{defeq1} and since $a(2,d)=d-2$, we deduce that
$$\Disc(g) =(-1)^{d-1}(d-1)^{d-1} U^d= U^d \mod (d).$$

Now, fix the integer $n > 2$ and suppose that the claimed formula is
proved at the step $n-1$. Again, an easy computation of resultants in
$\ZZ[U]$  yields
\begin{align}
\lefteqn{\Res(\partial_1g,\ldots,\partial_{n-1}g,g)} \\
 & = 
\Res(UX_2^{d-1},-UX_1X_2^{d-2}+X_3^{d-1},  \ldots ,-X_{n-2}X_{n-1}^{d-2}+X_n^{d-1},g) \mod (d)
\notag \\
&= U^{d(d-1)^{n-2}}\Res(X_2^{d-1},X_3^{d-1},\ldots,X_n^{d-1},X_1^{d}) \mod (d)
\notag \\
&= U^{d(d-1)^{n-2}} \mod (d).\notag
\end{align}
By Proposition \ref{firstfp}, (ii), it follows that, in $\ZZ[U]$,
\begin{align}
U^{d(d-1)^{n-2}} &=
\Disc(\overline{g})\Disc(g) \mod (d) \notag \\
&= U^{(d-1)^{n-2}+(-1)^{n-2}}\Disc(g) \mod(d). \notag
\end{align}
We deduce that
$$\Disc(g)=U^{d(d-1)^{n-2}-(d-1)^{n-2}-(-1)^{n-2}}=U^{(d-1)^{n-1}+(-1)^{n-1}} \mod (d)
\in \ZZ[U].$$
\end{proof}

Next, we provide two formulas that encapsulate the behavior of the discriminant under a linear change of coordinates and under a general base change formula. 

\begin{prop}\label{prop:discinv} Let $k$ be a commutative ring and  $f$
  be a homogeneous polynomial of degree $d\geq 2$
  in $k[X_1,\ldots,X_n]$. Being given a matrix $\varphi=[c_{i,j}]_{1\leq
        i,j \leq n}$ with entries in $k$ and denoting
      $$f\circ
      \varphi(X_1,\ldots,X_n):=f\left(c_{1,1}X_1+\cdots+c_{1,n}X_n,\ldots,\sum_{j=1}^nc_{i,j}X_j,\ldots,\sum_{j=1}^nc_{n,j}X_n\right),$$ 
      we have
      $$\Disc(f\circ\varphi)=\det(\varphi)^{d(d-1)^{n-1}}\Disc(f).$$
\end{prop}
\begin{proof} By specialization, it is sufficient to prove this
  formula in the generic setting, that is to say with $f=\sum_{|\alpha|=d}U_\alpha X^{\alpha} \in {}_\ZZ A$. 
  Since $f\circ \varphi$ and $f$ have the same degree $d$
  as polynomials in the $X_i$'s, it is equivalent to prove that
  $$\Res(\partial_1(f\circ\varphi),\ldots,\partial_n(f\circ\varphi))=\det(\varphi)^{d(d-1)^{n-1}}\Res(\partial_1 f ,\ldots,\partial_n f).$$
   To do this, we remark, by basic differential calculus, that
$$\left[
  \partial_1(f\circ\varphi),\ldots,\partial_n(f\circ\varphi)
\right] = \left[
  \partial_1(f)\circ\varphi,\ldots,\partial_n(f)\circ\varphi
\right].\det(\varphi),$$
as matrices. Therefore, the covariance formula of the resultant
\cite[\S 5.11.2]{J91} shows that
$$\Res(\partial_1(f\circ\varphi),\ldots,\partial_n(f\circ\varphi))=\det(\varphi)^{(d-1)^{n-1}}\Res(
\partial_1(f)\circ\varphi,\ldots,\partial_n(f)\circ\varphi).$$ 
Moreover, the formula for linear change of coordinates for the resultant
\cite[\S 5.13.1]{J91} gives
$$\Res(\partial_1(f)\circ\varphi,\ldots,\partial_n(f)\circ\varphi)=\det(\varphi)^{(d-1)^n}
\Res(\partial_1 f,\ldots,\partial_n f),$$
and we conclude the proof by observing that $(d-1)^n+(d-1)^{n-1}=d(d-1)^{n-1}$.
\end{proof}

One consequence of this invariance property is the following generalization of the formula defining the discriminant given in Proposition \ref{firstfp}, (ii).
\begin{prop}
	Let $k$ be a commutative ring, let $f$ be a homogeneous
	  polynomial in $k[X_1,\ldots,X_n]$ of degree $d\geq 2$ and let $\varphi=[c_{i,j}]_{1\leq i\leq n, 1\leq j\leq n-1 }$ be a $n\times(n-1)$-matrix with coefficients in $k$. Then, we have
	\begin{equation*}
		\Disc(f)\Disc\left(f(\left[X_1,\ldots,X_{n-1}\right]\circ {}^t\varphi)\right)=
		\Res\left(f,\left[\partial_1 f,\ldots,\partial_n f\right]\circ \varphi\right).
	\end{equation*}
\end{prop}
\begin{proof} By specialization, it is sufficient to prove this equality for $f$ the generic homogeneous polynomial of degree $d$ over the integers and for $\varphi:=[V_{i,j}]_{1\leq i\leq n, 1\leq j\leq n-1 }$ a matrix of indeterminates. Adding another column of indeterminates to $\varphi$, we introduce the matrix $\psi:=[V_{i,j}]_{1\leq i,j\leq n}$.
	
	Now, consider the following resultant
	\begin{multline*}\label{eq:reslincomb1}
		\Omega:=\Res\left( f(\left[X_1,\ldots,X_{n}\right]\circ {}^t\psi),\right. \\
		\left.\left[\partial_1 f(\left[X_1,\ldots,X_{n}\right]\circ {}^t\psi),\ldots,\partial_n f(\left[X_1,\ldots,X_{n}\right]\circ {}^t\psi)\right]\circ \varphi    \right).
	\end{multline*}
On the one hand, by the invariance property of the resultant \cite[\S 5.13]{J91} we have
\begin{equation}\label{eq:reslincomb2}
	\Omega=\det(\psi)^{d(d-1)^{n-1}}\Res\left(f,\left[\partial_1 f,\ldots,\partial_n f\right]\circ \varphi\right).	
\end{equation}
On the other hand, since 
\begin{multline*}
	\left[\partial_1 f(\left[X_1,\ldots,X_{n}\right]\circ {}^t\psi),\ldots,\partial_n f(\left[X_1,\ldots,X_{n}\right]\circ {}^t\psi)\right]\circ \varphi= \\
	\left[\frac{\partial}{\partial X_1} (f(\left[X_1,\ldots,X_{n}\right]\circ {}^t\psi)),\ldots,\frac{\partial}{\partial X_n} (f(\left[X_1,\ldots,X_{n}\right]\circ {}^t\psi))\right]
\end{multline*}
	by the composition rule of the derivatives, we get from Proposition \ref{firstfp}, (ii) that
	\begin{align}\label{eq:reslincomb3} \notag
		\Omega &= \Disc(f(\left[X_1,\ldots,X_{n}\right]\circ {}^t\psi))\Disc(f(\left[X_1,\ldots,X_{n-1},0\right]\circ {}^t\psi))\\ \notag
		&=\Disc(f(\left[X_1,\ldots,X_{n}\right]\circ {}^t\psi))\Disc(f(\left[X_1,\ldots,X_{n-1}\right]\circ {}^t\varphi))\\ 
		&=\det(\psi)^{d(d-1)^{n-1}}\Disc(f)\Disc(f(\left[X_1,\ldots,X_{n-1}\right]\circ {}^t\varphi))
	\end{align}
where the last equality holds by invariance of the discriminant; see Proposition \ref{prop:discinv}. Finally, the claimed formula follows by comparing \eqref{eq:reslincomb2} and \eqref{eq:reslincomb3}, taking into account the fact that $\det(\psi)$ is a nonzero divisor in our generic setting.
\end{proof}

Now, we turn to the more general problem of the behavior of the discriminant under a general change of basis.
\begin{prop}\label{prop:basechgformula} 
	Let $k$ be a commutative ring, $f$
  be a homogeneous polynomial of degree $m\geq 2$ and $g_1,\ldots,g_n$ be homogeneous polynomials of degree $d\geq 1$. There exists a polynomial $K$ that depends on the coefficients of the polynomials $f,g_1,\ldots,g_n$ such that
$$
\Disc(f(g_1,\ldots,g_n))=\Disc(f)^{d^{n-1}}\Res(g_1,\ldots,g_n)^{m(m-1)^{n-1}}K(f,g_1,\ldots,g_n).
$$	
\end{prop}
\begin{proof} We prove the existence of $K$ in the universal setting over the integers so that the claimed result follows by specialization.
	From the equality of matrices
	\begin{multline*}
		 \left[ \begin{array}{ccc} \partial_{X_1}(f(\underline{g})) & \cdots & \partial_{X_n}(f(\underline{g})) \end{array} \right]
			= \\
			\left[ \begin{array}{ccc} \partial_{X_1}f(\underline{g}) & \cdots & \partial_{X_n}f(\underline{g}) \end{array}  \right]
		\left[
		\begin{array}{ccc}
			\partial_{X_1}g_1 & \cdots & \partial_{X_n}g_1 \\
			\vdots & & \vdots \\
			\partial_{X_1}g_n & \cdots & \partial_{X_n}g_n
		\end{array}
		\right]	
	\end{multline*}
	we deduce that for all $i=1,\ldots,n$ 
	\begin{equation}\label{eq:belong1}
	\partial_{X_1}(f(\underline{g})) \in ( \partial_{X_1}f(\underline{g}), \ldots , \partial_{X_n}f(\underline{g})).		
	\end{equation}
	Therefore, applying the divisibility property of the resultant \cite[\S 5.6]{J91}, we obtain that
	$$\Res\left( \partial_{X_1}f(\underline{g}), \ldots , \partial_{X_n}f(\underline{g}) \right)
	\textrm{ divides }
	\Res\left( \partial_{X_1}(f(\underline{g})), \ldots , \partial_{X_n}(f(\underline{g}))\right).$$
	On the one hand, using the base change formula of the resultant \cite[\S 5.12]{J91}, we have
	\begin{eqnarray*}
		\Res\left( \partial_{X_1}f(\underline{g}), \ldots , \partial_{X_n}f(\underline{g}) \right)
		& = & \Res(g_1,\ldots,g_n)^{(m-1)^n}\Res(\partial_{X_1}f,\ldots,\partial_{X_n}f) \\
	& = &  m^{a(m,d)d^{n-1}}\Disc(f)^{d^{n-1}} \Res(g_1,\ldots,g_n)^{(m-1)^n}
	\end{eqnarray*}
	and on the other hand 
	$$\Res\left( \partial_{X_1}(f(\underline{g})), \ldots , \partial_{X_n}(f(\underline{g}))\right)=
	(md)^{a(n,md)}\Disc(f(g_1,\ldots,g_n)).$$
	Therefore, since $\Disc(f)$ is a primitive polynomial, we deduce that $\Disc(f)^{d^{n-1}}$ divides $\Disc(f(\underline{g}))$.

Now, notice that we have $f(g_1,\ldots,g_n) \in (g_1,\ldots,g_n)^{m}$ and that for all $i=1,\ldots,n$ we have $\partial_{X_1}(f(\underline{g})) \in (g_1,\ldots,g_n)^{m-1}$ by using from \eqref{eq:belong1}. Using the generalized divisibility property of the resultant \cite[\S 6.2]{J91}, it follows that
$$\Res(g_1,\ldots,g_n)^{m(m-1)^{n-1}} \textrm{ divides } \Res(\partial_{X_1}(f(\underline{g})), \ldots , \partial_{X_{n-1}}(f(\underline{g})),f(\underline{g})).$$
But
$$\Res(\partial_{X_1}(f(\underline{g})), \ldots , \partial_{X_{n-1}}(f(\underline{g})),f(\underline{g})) = \Disc(f(\underline{g}))\Disc(\overline{f(\underline{g})})$$
and $\Res(g_1,\ldots,g_n)$ is an irreducible polynomial that depends on all the coefficients of all the polynomials $g_1,\ldots,g_n$. We deduce that 
	$$\Res(g_1,\ldots,g_n)^{m(m-1)^{n-1}} \textrm{ divides } \Disc(f(g_1,\ldots,g_n)).$$
Finally, since $\Res(g_1,\ldots,g_n)$ and $\Disc(f)$ are obviously coprime, the existence of the polynomial $K$ is proved.	
\end{proof}



\subsection{Inertia forms and the discriminant}

From its definition, it is clear
that the discriminant of a homogeneous polynomial $f\in
R[X_1,\ldots,X_n]$, where $R$ is a field, of degree $d\geq 2$ vanishes if and only if
$\partial_1f,\ldots,\partial_n f$ (and hence $f$ if $\mathrm{char}(k)$ does not 
divide $d$) have a non trivial common root in an algebraic extension
of $R$. The purpose of this section is to study the behavior of the discriminant when $R$, the coefficient ring of the homogeneous polynomial $f$, is not assumed to be a field.

\medskip

Let $d\geq 2$ be a fixed integer and consider the 
polynomial 
$$f(X_1,\ldots,X_n):=\sum_{|\alpha|=d}U_\alpha X^\alpha.$$
Let $k$ be a commutative ring and denote by
${}_kA:=k[U_\alpha\,|\,|\alpha|=d]$ the coefficient ring of $f$ over
$k$. Then $f\in {}_kA[X_1,\ldots,X_n]$; it is the
generic homogeneous polynomial of degree $d$ over $k$. Defining the ideals
of ${}_kC:={}_kA[X_1,\ldots,X_n]$
$$\Dc:= (f,\partial_1 f,\ldots,\partial_n f), \ \  \mm:=(X_1,\ldots,X_n),$$
we recall that $\pp:=\TF_\mm(\Dc)_0=H^0_\mm({}_kB)_0$ where ${}_kB$ is the quotient ring
${}_kC/\Dc$. This latter ideal is nothing but the defining ideal of the closed
subscheme of $\mathrm{Spec}({}_kA)$ obtained as the image of the canonical
projective morphism
$$ \mathrm{Proj}({}_kB) \rightarrow \mathrm{Spec}({}_kA).$$
In the sequel, our aim is to relate the discriminant of $f$ as defined in Definition \ref{defdiscr} with this ideal of inertia forms $\pp \subset {}_kA$.

\begin{prop}\label{Bdom} For $j\in \{1,\ldots,n\}$ we have an isomorphism of $k[X_1,\ldots,X_n]$-algebras
	\begin{equation}\label{BXi-iso}
	 {}_kB_{X_j} \xrightarrow{\sim} k[U_\alpha\,|\,|\alpha|=d,
	  \alpha_j < d-1][X_1,\ldots,X_n][X_j^{-1}].
	\end{equation}
In particular, for all $j\in \{1,\ldots,n\}$ the ring ${}_kB_{X_j}$ is a domain if $k$ is a domain.
\end{prop}
\begin{proof} Let $i$ be a fixed integer in $\{1,\ldots,n\}$.
  The Euler equality $df=\sum_{j=1}^{n}X_j\partial_jf$
  shows that, after localization by the variable $X_j$, we have
  $$\Dc_{X_j}=(\partial_1f,\ldots,\partial_{j-1}f,\partial_{j+1}f,\ldots,\partial_{n}f,f)\subset {}_kC_{X_j}.$$
  In order to emphasize some particular coefficients of the polynomial 
  $f$, let us rewrite it as
  $$f(X_1,\ldots,X_n)=\sum_{i=1}^{n}
  \EE_iX_iX_j^{d-1}+\sum_{|\alpha|=d,\alpha_j < d-1}U_{\alpha}X^{\alpha}.$$
Then, denoting by $Q(X_1,\ldots,X_n)$ the second term of the right
side of this equality, for all integer $i\in\{1,\ldots,n\}$ such that $i\neq j$ we
have
$$\partial_if(X_1,\ldots,X_n)=\EE_iX_j^{d-1} +
\partial_iQ(X_1,\ldots,X_n).$$
It follows that the following $k[X_1,\ldots,X_n]$-algebras morphism
\begin{eqnarray*}
  {}_kC_{X_j} & \longrightarrow & k[U_\alpha\,|\,|\alpha|=d,
  \alpha_j < d-1][X_1,\ldots,X_n][X_j^{-1}] \\
  \EE_i \ (i\neq j)  & \mapsto &
  -X_j^{-d+1}\partial_iQ \\
  \EE_j & \mapsto & -X_j^{-d}Q+\sum_{i\neq j, i=1}^n X_iX_j^{-d}\partial_iQ=-X_j^{-d}((1-d)Q+X_j\partial_jQ)
\end{eqnarray*}
has kernel $\Dc_{X_j}$ and therefore induces an isomorphism of
$k[X_1,\ldots,X_n]$-algebras
\begin{equation*}
 {}_kB_{X_j} \xrightarrow{\sim} k[U_\alpha\,|\,|\alpha|=d,
  \alpha_j < d-1][X_1,\ldots,X_n][X_j^{-1}].
\end{equation*}
\end{proof}

\begin{cor}\label{cor:TFprime} For all $i=1,\ldots,n$ we have
  $$\TF_\mm(\Dc)=\ker({}_kC \rightarrow {}_kB_{X_i})$$
  where ${}_kC \rightarrow {}_kB_{X_i}$ is the canonical map, so that 
	$$\TF_\mm(\Dc)_0=\ker({}_kA \rightarrow {}_kB_{X_i})=H^0_{(X_i)}(_kB)_0.$$
In particular, if $k$ is a domain then $\TF_\mm(\Dc)$ and $\pp$ are prime ideals.
\end{cor}
\begin{proof} Observe first that by definition we have 
$$\TF_\mm(\Dc)=\ker({}_kC \rightarrow \prod_{i=1}^n {}_kB_{X_i}).$$
The isomorphisms \eqref{BXi-iso} show that for any couple of  integers $(i,j)\in\{1,\ldots,n\}^2$ the variable $X_i$ is a nonzero divisor in ${}_kB_{X_j}$ and hence that the canonical map ${}_kB_{X_i} \rightarrow
  {}_kB_{X_iX_j}$ is injective. 
By considering the commutative diagrams,  for all couple $(i,j)\in \{1,\ldots,n\}^2$,
$$\xymatrix{
C \ar[d] \ar[r] & B_{X_i} \ar[d] \\
B_{X_j} \ar[r] & B_{X_iX_j} }$$
we obtain that $\TF_\mm(\Dc)=\ker({}_kC
  \rightarrow {}_kB_{X_i})$ for all $i=1,\ldots,n$. From here, assuming that $k$ is domain we deduce easily that $\TF_\mm(\Dc)$ is a prime ideal of ${}_kC$ and that $\pp=\TF_\mm(\Dc)_0$ is a prime ideal of ${}_kA$. 
\end{proof}

We now turn to the relation between the ideal of inertia forms $\TF_\mm(\Dc)$ and the discriminant of $f$.

\begin{thm} \label{thm:discTF}
	Let $R$ be a commutative ring and $f$ a homogeneous polynomial in $R[X_1,\ldots,X_n]_d$ with
	$d\geq 2$. Then, we have the following  inclusions of ideals in $R$:
	$$(\Disc(f)) \subset \TF_\mm((f,\partial_1 f,\ldots,\partial_n f))\cap R \subset \sqrt{(\Disc(f))}.$$
\end{thm}
\begin{proof} We first prove these inclusions in the generic case over the integers, that is to say with
	  $f=\sum_{|\alpha|=d}U_\alpha X^\alpha$ and
	  $R={}_\ZZ A=\ZZ[U_\alpha\,|\,|\alpha|=d]$. 

	  By definition of the discriminant, we have
	  $$d^{a(n,d)}\Disc(f)=\Res(\partial_1f,\ldots,\partial_nf) \ \text{
	    in } {}_\ZZ A.$$
	  But since $\Res(\partial_1f,\ldots,\partial_nf)$ is an inertia form
	  of the ideal $(\partial_1f,\ldots,\partial_nf)$ with respect to $\mm$, we
	  deduce that
	  $$d^{a(n,d)}\Disc(f) \in \TF_\mm(\Dc)_0$$
	which is a prime ideal ($\ZZ$ is a domain). Moreover, we claim that $d^{a(n,d)} \notin
	\TF_\mm(\Dc)_0$ because $$\TF_\mm(\Dc)_0 \cap \ZZ =(0).$$ 
	Indeed, this equality can be checked using any particular
	specialization of the coefficients $U_\alpha$; for instance, if we specialize $f$ to 
	$X_1^d$, then $\Dc$ specializes to the ideal $(X_1^d,dX_1^{d-1})$ in
	$\ZZ[X_1,\ldots,X_n]$ and clearly $\TF_\mm((X_1^d,dX_1^{d-1}))_0=(0)
	\subset \ZZ$. Finally, we deduce that $\Disc(f) \in \TF_\mm(\Dc)_0$.
	
	\medskip	
	
	We turn to the proof of the second inclusion, always in the generic case over the integers. 
	Suppose given $a \in H^0_\mm({}_\ZZ B)_0$
	and denote by ${}_\ZZ B'$ the quotient ring ${}_\ZZ C/(\partial_1f,\ldots,\partial_nf)$. 
	By the Euler identity, $da \in H^0_\mm({}_\ZZ B')_0$.  Since both ideals $H^0_\mm(_\ZZ B')_0$ and $(\Res(\partial_1f,\ldots,\partial_nf))$ of ${}_\ZZ A$ have the same radical, we deduce that there exists an integer $N$ such that $\Res(\partial_1f,\ldots,\partial_nf)$ divides $(da)^N$. Using \eqref{defeq}, there exists $a' \in {}_\ZZ A$ such that
$$d^Na^N=d^{a(n,d)}a'\Disc(f) \text{ in } {}_\ZZ A.$$
Taking the contents in the above equality, we deduce that 
$$a^N=\frac{a'}{C_k(a')}C_k(a)^N\Disc(f) \text{ in } {}_\ZZ A$$
and this proves that $\TF_\mm(\Dc)_0 \subset \sqrt{(\Disc(f))}$. 

\medskip

To conclude the proof, we first remark that the inclusion 
$$(\Disc(f)) \subset \TF_\mm((f,\partial_1 f,\ldots,\partial_n f))\cap R$$
 is, by specialization, an immediate consequence of the same inclusion in the generic case over the integers. The rest of the proof is a consequence of a base change property, exactly as in the proof of Proposition \ref{prop:resbasechange}.
\end{proof}

\begin{cor}\label{cor:discgeomreduced}
	Let $k$ be a domain and
	  $f=\sum_{|\alpha|=d}U_{\alpha}X^{\alpha}$ be the generic
	  homogeneous polynomial of degree $d\geq 2$ over $k$. Then, 
	 $\Disc(f)=c.P^r$ where $c$ is an invertible element in $k$, $r$ is a positive integer and $P$ is a prime polynomial that generates the ideal $\pp \subset {}_kA$.
\end{cor}
\begin{proof} Let us first assume that $k$ is a UFD. Theorem \ref{thm:discTF} implies that both ideals $\pp=\TF_\mm(\Dc)_0$ and $(\Disc(f))$ of ${}_k A$ have the same radical and Corollary \ref{cor:TFprime} shows that $\pp$ is a prime ideal. Therefore, we deduce immediately that $\Disc(f)=c.P^r$ as claimed.
	
	Now, assume that $k$ is a domain. Depending on its characteristic, it contains either $\ZZ$ or $\ZZ/p\ZZ$, $p$ a prime integer, that we will denote by $F$ in the sequel. Thus, we have an injective map $F\hookrightarrow k$ which is moreover flat (for $k$ is a torsion-free $F$-module). Therefore, the canonical exact sequence (see Corollary \ref{cor:TFprime}) 
$${}_F\TF_\mm(\Dc) \rightarrow {}_F C \rightarrow {}_F B_{X_n}$$ 
remains exact after tensorization by $k$ over $F$. Since ${}_F C\otimes_F k={}_kC$ and ${}_F B_{X_n}\otimes_F k = {}_k B_{X_n}$, this latter being an immediate consequence of \eqref{BXi-iso}, we deduce that
\begin{equation}\label{eq:TFiso}
	{}_k\TF_\mm(\Dc)={}_F\TF_\mm(\Dc)\otimes_F k.
\end{equation}
Since $F$ is a UFD, we know that ${}_F\Disc(f)=c.P^r$ where $c$ is an invertible element in $F$, $r$ is a positive integer and $P$ is a prime polynomial in ${}_F A$ that generates ${}_F\pp$. Now, considering the canonical specialization $\rho:{}_F A \rightarrow {}_kA$, we get
$$ {}_k\Disc(f)=\rho({}_F\Disc(f))=\rho(c).\rho(P)^r,$$
where the first equality follows from the definition of the discriminant. But by \eqref{eq:TFiso}, $\rho(P)$ generates ${}_k\pp$ and since  ${}_k\pp$ is a prime ideal by Corollary \ref{cor:TFprime}, we deduce that $\rho(P)$ is a prime polynomial in ${}_k A$. To conclude, observe that $\rho(c)$ is clearly an invertible element in $k$ because $F$ is contained in $k$.	
\end{proof}
\begin{rem} From the proof of the above corollary we see that the only dependence of $r$ on $k$ is the characteristic of $k$, for $F$ only depends on this characteristic.	
\end{rem}

With this property, we can explore the behavior of the discriminant in some particular cases. Here are two such examples.

\begin{prop}\label{prop:DiscIrredZZ} The universal discriminant over the integers is a prime polynomial in ${}_\ZZ A$ that generates the ideal ${}_\ZZ \pp$.
\end{prop}
\begin{proof} By Corollary \ref{cor:discgeomreduced}, there exists an irreducible polynomial $P\in {}_\ZZ A$ that generates ${}_\ZZ \pp$ and an integer $r\geq 1$ such that ${}_\ZZ\Disc(f)=\pm P^r$. In order to prove that $r=1$ we will use two specializations.
	
	First, consider the specialization that sends $f$ to $UX_n^d+f(X_1,\ldots,X_{n-1},0)$ where $U$ denotes, for simplicity, the coefficient of $X_n^d$ of $f$. By Example \ref{ex:disc1}, we get that ${}_\ZZ\Disc(f)$ specializes to
	  $$d^{(d-1)^{n-1}+(-1)^n}U^{(d-1)^{n-1}}\Disc(f(X_1,\ldots,X_{n-1},0))^{d-1} \in {}_\ZZ A.$$
Since $U$ is an irreducible polynomial in ${}_\ZZ A$ and $U$ does not divide the discriminant $\Disc(f(X_1,\ldots,X_{n-1},0))$ (this latter actually does not depend on $U$), we deduce that $r$ divides $(d-1)^{n-1}$.

Second, consider the specialization that sends $f$ to the polynomial $$g\in \ZZ/d\ZZ[U][X_1,\ldots,X_n]$$ given in Example \ref{ex:disc3}. We have seen that $\Disc(f)$ specializes to $U^{(d-1)^{n-1}+(-1)^n}$. 
It follows
%
%
that $r$ divides $(d-1)^{n-1}+(-1)^n$.

Finally, we have shown that $r$ divides two consecutive and positive integers, namely $(d-1)^{n-1}$ and $(d-1)^{n-1}+(-1)^n$. Therefore, $r$ must be equal to $1$. 	
\end{proof}

\begin{prop}\label{prop:d=2} Let $k$ be a domain and
  $f=\sum_{|\alpha|=2}U_{\alpha}X^{\alpha}$ be the generic
  homogeneous polynomial of degree $2$ over $k$. If $\mathrm{char}(k)\neq 2$ or $n$ is odd, then $\Disc(f)$ is a prime polynomial in ${}_k A$ that generates $\pp$. Otherwise, if $\mathrm{char}(k)= 2$ and $n$ is even, then $\Disc(f)=P^2$ where $P$ is a prime polynomial that generates $\pp$.
\end{prop}
\begin{proof} As explained in the proof of Corollary \ref{cor:discgeomreduced}, it is enough to prove this proposition under the assumption that $k$ is a UFD. So let us assume hereafter that this is the case.
	
	By Corollary \ref{cor:discgeomreduced}, there exists an irreducible polynomial $P\in {}_k A$ that generates ${}_k \pp$, an integer $r\geq 1$ and $c$ an invertible element in $k$ such that ${}_k\Disc(f)=c.P^r$. Depending on the characteristic of $k$ and the parity of $n$ we will prove that $r$ is equal to 1 or 2.

	Rewriting  $f(X_1,\ldots,X_n)$ as $f=\sum_{0\leq i \leq j \leq n} A_{i,j}X_iX_j$ (so that ${}_kA$ is now the polynomial ring $k[A_{i,j},0\leq i \leq j \leq n]$), for all $i\in\{1,\ldots,n\}$ we have
	$$\partial_i f=A_{1,i}X_1 + \cdots+ A_{i-1,i}X_{i-1}+2A_{i,i}X_i+A_{i,i+1}X_{i+1}+\cdots+A_{i,n}X_n$$ 
	in ${}_kA[X_1,\ldots,X_n]$.
	Then, Definition \ref{defdiscr} implies that
	\begin{align}\label{cd1:d=2}
	 \left|\begin{array}{ccccc}
	2A_{1,1} & A_{1,2} & \cdots & A_{1,n-1} & A_{1,n} \\ 
	A_{1,2} & 2A_{2,2} &  & × & A_{2,n} \\ 
	\vdots & × & \ddots & × & \vdots \\ 
	A_{1,n-1} & × & × & 2A_{n-1,n-1} & A_{n-1,n} \\ 
	A_{1,n} & A_{2,n} & \cdots & A_{n-1,n} & 2A_{n,n}
	\end{array}\right|=
	\begin{cases}
	2\,\Disc(f)=2 c.P^r & \text{if } n \text{ is odd} \\
	\Disc(f)=c.P^r & \text{if } n \text{ is even}
	\end{cases}
	\end{align}
	in the polynomial ring ${}_k A$. 
	
	\medskip
	
	Let us first assume that $\mathrm{char}(k)\neq 2$. Denote by $\rho$ the specialization that leaves invariant $A_{i,i}$ for all $i$ and sends $A_{i,j}$ to 0 for all $i\neq j$. The specialization of \eqref{cd1:d=2} by $\rho$ yields
	\begin{align*}
	2^nA_{1,1}A_{2,2}\ldots A_{n,n}=
	\begin{cases}
	2c.\rho(P)^r & \text{if } n \text{ is odd} \\
	c.\rho(P)^r & \text{if } n \text{ is even}
	\end{cases}
	\end{align*}
and from here we deduce that $r$ must be equal to 1.

\medskip

Now, assume that $\mathrm{char}(k)=2$ and that $n$ is even. Since $\mathrm{char}(k)=2$, the determinant in \eqref{cd1:d=2} can be seen as the determinant of a skew-symmetric matrix, and since $n$ is even it is known that it is equal to the square of its pfaffian. Therefore, \eqref{cd1:d=2} implies that $r\geq 2$.  

Consider the specialization $\varphi$ that leaves invariant $A_{1+2k,2+2k}$ for all integer $k=0,1,\ldots,(n-2)/2$ and that sends all the other variables $A_{i,j}$ to 0. The matrix in \eqref{cd1:d=2} then specializes by $\varphi$ to the block diagonal matrix
$$\mathrm{diag}\left(
\left[\begin{array}{cc}
	0 & A_{1,2} \\
	A_{1,2} & 0
\end{array}\right],
\left[\begin{array}{cc}
	0 & A_{3,4} \\
	A_{3,4} & 0
\end{array}\right],
\ldots,
\left[\begin{array}{cc}
	0 & A_{n-1,n} \\
	A_{n-1,n} & 0
\end{array}\right]
\right)
$$
and therefore \eqref{cd1:d=2} yields 
$$\prod_{k=0}^{(n-2)/2} {A_{1+2k,2+2k}}^2=c.\varphi(P)^r$$ 
This implies that $r\leq 2$ and hence we conclude that $r=2$ if $\mathrm{char}(k)=2$ and $n$ is even.
Then, to conclude observe that ${}_{\ZZ/2\ZZ}\Disc(f)$ is a square (necessarily $c=1$ in this case), so that we deduce that $c$ is actually a square in $k$ via the canonical specialization from $\ZZ/2\ZZ$ to $k$. It follows that ${}_k \Disc(f)=(uP)^2$ where $u^2=c$ and $u$ is an invertible element in $k$, and the claimed result follows as $uP$ is an irreducible element that generates $\pp$.
\medskip

Let us turn to the last case: $\mathrm{char}(k)=2$ and $n$ is odd. Consider the specialization $\phi$ that leaves invariant $A_{n-2,n}$, $A_{n-1,n}$ and $A_{1+2k,2+2k}$ for all $k=0,1,\ldots,(n-3)/2$, and that sends all the other variables $A_{i,j}$ to 0. In order to determine the image of ${}_k\Disc(f)$ by this specialization, we remark that we have the following commutative diagram of specializations
$$\xymatrix{
{}_\ZZ A \ar[r]^(.13){\phi} \ar[d] & \ZZ[A_{n-2,n},A_{n-1,n},A_{1+2k,2+2k} \,|\,k=0,1,\ldots,(n-3)/2 ] \ar[d] \\
{}_k A \ar[r]^(.13){\phi} & k [A_{n-2,n},A_{n-1,n},A_{1+2k,2+2k} \,|\,k=0,1,\ldots,(n-3)/2 ]
}$$
where the vertical arrows are induced by the ring morphism $\ZZ \rightarrow k$. So, we can first perform the specialization $\phi$ over the integers and then specialize to $k$. 

The matrix in \eqref{cd1:d=2} specializes by $\varphi$ to the block diagonal matrix
\begin{multline*}
	\mathrm{diag}\left(
	\left[\begin{array}{cc}
		0 & A_{1,2} \\
		A_{1,2} & 0
	\end{array}\right],
	\ldots,
	\left[\begin{array}{cc}
		0 & A_{n-4,n-3} \\
		A_{n-4,n-3} & 0
	\end{array}\right],\right. \\
	\left.\left[\begin{array}{ccc}
		0 & A_{n-2,n-1} & A_{n-2,n} \\
		A_{n-2,n-1} & 0 & A_{n-1,n} \\
		A_{n-2,n} & A_{n-1,n} & 0
	\end{array}\right]
	\right).	
\end{multline*}
Therefore, the specialization of \eqref{cd1:d=2} by $\phi$ over the integers yields the equality 
$$2A_{n-2,n-1}A_{n-2,n}A_{n-1,n}\prod_{k=0}^{(n-5)/2} -({A_{1+2k,2+2k}})^2=2\phi({}_\ZZ\Disc(f))$$
so that,
$$A_{n-2,n-1}A_{n-2,n}A_{n-1,n}\prod_{k=0}^{(n-5)/2} -({A_{1+2k,2+2k}})^2=\phi({}_\ZZ\Disc(f)).$$
Now, we specialize this equality to $k$ and we obtain
$$A_{n-2,n-1}A_{n-2,n}A_{n-1,n}\prod_{k=0}^{(n-5)/2} {A_{1+2k,2+2k}}^2=\phi({}_k\Disc(f))=c.\phi(P)^r.$$
From here, we deduce that $r$ must be equal to 1.
\end{proof}

Our next step is to prove that the conclusion of this proposition holds without restriction on the degree $d$. This is Theorem \ref{thm:disc-irred}. Notice that in the case $n=2$ we already know that such a result is valid by Theorem \ref{thm:2<>0} and Theorem \ref{thm:2=0} (see also \cite[\S 8.5]{ApJo}).

\subsection{Zariski weight of the discriminant}

Let $k$ be a commutative ring and consider the generic homogeneous polynomial in the variables $X_1,\ldots,X_n$ of degree $d\geq 2$
$$ f:=\sum_{|\alpha|=d} U_\alpha X^\alpha \in C:=A[X_1,\ldots,X_n]$$
where $A:=k[U_\alpha\,|\, |\alpha|=d]$. Define also the ideals $\mm:=(X_1,\ldots,X_n)$ and $\nn:=(X_1,\ldots,X_{n-1})$ of $C$ and rewrite the polynomial $f$ as $f=\sum_{t=0}^d f_{d-t}X_n^t$ where $f_l$ is the generic homogeneous polynomial of degree $l$ in $A[X_1,\ldots,X_{n-1}]$ for all $l=0,\ldots,d$. 

Now, fix an integer $\mu$ such that $0\leq \mu\leq d$ and define the polynomials
$$h:=\sum_{t=0}^\mu f_{d-t}X_n^t \in C_d \textrm{ and } g:=\sum_{t=\mu}^d f_{d-t}X_n^{t-\mu} \in C_{d-\mu}.$$

\begin{prop}\label{prop:redregularseq}
	For all integer $0\leq \mu \leq d$ the sequence $h,\partial_1 h,\ldots,\partial_{n-2}h$ is $C$-regular. Moreover, for all integer $1 \leq \mu \leq d$, the sequence $h,\partial_1 h,\ldots,\partial_{n-1}h$ is $C$-regular outside $V(\nn)$.
\end{prop}
\begin{proof} By Proposition \ref{prop:regular-sequences}, the sequence $f_d,\partial_1 f_d,\ldots,\partial_{n-2}f_d$ is $C$-regular. It follows that the sequence $X_n,h,\partial_1 h,\ldots,\partial_{n-2}h$ is also $C$-regular. Since all the elements of this sequence are homogeneous of positive degree, this sequence remains $C$-regular under any permutation of its elements. Therefore, $h,\partial_1 h,\ldots,\partial_{n-2}h,X_n$ is $C$-regular, in particular $h,\partial_1 h,\ldots,\partial_{n-2}h$ is $C$-regular. 
	
	\medskip
	
	To prove the second assertion, we have to prove that the sequence $h,\partial_1 h,\ldots$, $\partial_{n-1}h$ is $C_{X_j}$-regular for all $1\leq j\leq n-1$. Up to a permutation of the variables $X_1,\ldots,X_{n-1}$, one can assume that $j=n-1$. 
	
	For the sake of simplicity in the notation, we rename by $V_i$ the coefficient of the monomial $X_iX_{n-1}^{d-1}$ in $f_d$ for all $i=1,\ldots,n-1$ so that
	$$f_d=V_1X_1X_{n-1}^{d-1}+V_2X_2X_{n-1}^{d-1}+\cdots+V_{n-2}X_{n-2}X_{n-1}^{d-1}+V_{n-1} X_{n-1}^d+\cdots.$$
We also define the polynomial $v$ by the equality
$$ h=v+	V_1X_1X_{n-1}^{d-1}+V_2X_2X_{n-1}^{d-1}+\cdots+V_{n-2}X_{n-2}X_{n-1}^{d-1}+V_{n-1} X_{n-1}^d.$$
Now, perform the following successive specializations:
\begin{eqnarray}\label{eq:specializations}
V_{n-1} & \mapsto & \frac{-1}{X_{n-1}^d}(v+V_1X_1X_{n-1}^{d-1}+V_2X_2X_{n-1}^{d-1}+\cdots+V_{n-2}X_{n-2}X_{n-1}^{d-1}),\\ \nonumber
V_{i} & \mapsto & 	\frac{-1}{X_{n-1}^{d-1}}\partial_i v,  \ \ 1 \leq i \leq n-2. 
\end{eqnarray}
They successively annihilate $h, \partial_1 h,\ldots ,\partial_{n-1}h$ and we recover that $h,\partial_1 h,\ldots ,\partial_{n-2}h$ is a regular sequence (outside $V(\nn)$). In addition, \eqref{eq:specializations} yields an isomorphism
$$\quotient{C_{X_{n-1}}}{(h,\partial_1 h,\ldots ,\partial_{n-2}h)} \longrightarrow A'[X_1,\ldots,X_n][X_{n-1}^{-1}]$$
where $A':=k[U_{\alpha} \,|\, |\alpha|=d, U_{\alpha}\neq V_i \ \forall i\in \{1,\ldots,n-1\}]$. Therefore, it remains to prove that the image of $\partial_{n-1}h$ by the specializations \eqref{eq:specializations} is a nonzero divisor in $A'[X_1,\ldots,X_n][X_{n-1}^{-1}]$. For that purpose, we observe that the Euler identity implies that
$$X_1\partial_1 h+\cdots+X_{n-1}\partial_{n-1}h= \sum_{t=0}^\mu (d-t)f_{d-t}X_n^t=dh-\sum_{t=1}^\mu tf_{d-t}X_n^t.$$
But the polynomials $f_{d-t}$ for $1\leq t\leq \mu$ do not depend on the variables $V_1,V_2,\ldots$, $V_{n-1}$, so we deduce that $X_{n-1}\partial_{n-1}h$ is specialized to $-\sum_{t=1}^\mu tf_{d-t}X_n^t$ by \eqref{eq:specializations}. Assuming $\mu\geq 1$, the $k$-content of this polynomial contains the $k$-content of $f_{d-1}$ which is a primitive polynomial over $k$, and we conclude the proof by the Dedekind-Mertens Lemma.
\end{proof}

%

By definition, the polynomial $h \in C$ is homogeneous of degree $d$ with respect to the variables $X_1,\ldots,X_n$ and of valuation $d-\mu$ with respect to the variables $X_1,\ldots,X_{n-1}$. Therefore, for all $i=1,\ldots,n-1$, the polynomial $\partial_i h$ is of degree $d-1$ with respect to the variables $X_1,\ldots,X_n$ and of valuation $d-1-\mu$ with respect to the variables $X_1,\ldots,X_{n-1}$. We will denote by $\Red(h,\partial_1 h,\partial_2 h,\ldots,\partial_{n-1}h)$ the reduced resultant of $h,\partial_1 h,\partial_2 h,\ldots,\partial_{n-1}h$ with respect to these degrees and weights. It is well defined for all $\mu$ such that $1\leq \mu \leq d-2$ (\cite{Zar37,PhDthesisRed}).

\begin{prop}\label{prop:Rednonzero} For all $1\leq \mu\leq d-2$ the reduced resultant $$\Red(h,\partial_1 h,\partial_2 h,\ldots,\partial_{n-1}h)$$
	 is a primitive polynomial, hence a nonzero divisor, in $A$.
\end{prop}
\begin{proof} The reduced resultant is a nonzero divisor by Proposition \ref{prop:redregularseq} and the Poisson formula (\cite[Theorem 5.1 and Theorem 5.2]{Zar37}, \cite[Chapter IV]{PhDthesisRed}). Then, we deduce that it is primitive over the integers, hence over $k$, by applying the previous property with $k=\ZZ$ and $k=\ZZ/p\ZZ$ for all prime integer $p$.  
\end{proof}

 \begin{thm}\label{thm:discweight} 
Assume that the ring $A=k[U_\alpha \,|\, |\alpha|=d]$ is graded by the Zariski weight, i.e.~by setting $\mathrm{weight}(c):=0$ for all $c\in k$ and 
$\mathrm{weight}(U_\alpha):=\max(\alpha_n-\mu,0)$. Then, 
the discriminant $\Disc(f) \in A$ is of valuation $(d-\mu)(d-1-\mu)^{n-1}$. Moreover, its isobaric part $H$ of weight  $(d-\mu)(d-1-\mu)^{n-1}$ satisfies the equality 
$$\Disc(g)\Disc(\bar{g})\Red(h,\partial_1 h,\ldots,\partial_{n-1}h)=H.\Disc(\bar{f}) \in A$$
where $\Disc(\bar{g})=\Disc(f_{d-\mu})$, $\Disc(\bar{f})=\Disc(f_d)$ and $\Red(h,\partial_1 h,\ldots,\partial_{n-1}h)$ are all isobaric polynomials of zero weight.
 \end{thm}
\begin{proof} 
	Let $f_0:=\sum_{|\alpha|=d}V_{0,\alpha}X^\alpha$ and $f_i:=\sum_{|\alpha|=d-1}V_{i,\alpha}X^\alpha$ for $i=1,\ldots,n-1$ be generic homogeneous polynomials of degree $d,d-1,\ldots,d-1$ respectively and let $\varphi_0,\varphi_1,\ldots,\varphi_{n-1}$ be their generic specialization of degree $d,d-1,\ldots,d-1$ and of valuation $d-\mu,d-\mu-1,\ldots,d-\mu-1$ respectively. Notice that we consider here the canonical grading of $k[V_{i,\alpha} \forall i,\alpha]$, so that
$$f_0=\sum_{|\alpha|=d}V_{0,\alpha}X^\alpha, \ \ f_i=\sum_{|\alpha|=d-1}V_{i,\alpha}X^\alpha, \ \ \varphi_0=\mathop{\sum_{|\alpha|=d}}_{\alpha_n\leq \mu} V_{0,\alpha}X^\alpha, \ \ \varphi_i=\mathop{\sum_{|\alpha|=d-1}}_{\alpha_n\leq \mu}V_{0,\alpha}X^\alpha$$
for all $i=1,\ldots,n-1$. Moreover, we also define the polynomials
$$g_0:= \mathop{\sum_{|\alpha|=d}}_{\alpha_n\geq \mu} V_{0,\alpha}X^\alpha /X_n^\mu, \ \ g_i:=\mathop{\sum_{|\alpha|=d-1}}_{\alpha_n\geq \mu}V_{0,\alpha}X^\alpha/X_n^\mu \ \in \ k[V_{i,\alpha},\forall i,\alpha][X_1,\ldots,X_n]$$
for all $i=1,\ldots,n-1$. 
	
	Now, consider the grading of $k[V_{i,\alpha} \, \forall i,\alpha]$ defined in this theorem, namely 
	$$\mathrm{weight}(V_{i,\alpha}):=\max(\alpha_n-\mu,0) \textrm{ for all } i=0,\ldots,n-1.$$ 
	Then, by definition of the reduced resultant of $\varphi_0,\varphi_1,\ldots,\varphi_{n-1}$, we have the equality
	\begin{multline}\label{eq:redresgen}
		\Res(f_0,\ldots,f_{n-1})=\Res(g_0,\ldots,g_n)\Red(\varphi_0,\ldots,\varphi_n) + \\ \textrm{terms of weight} > (d-\mu)(d-1-\mu)^{n-1}	
	\end{multline}
	
Denote by $\rho$ the specialization from $k[V_{i,\alpha} \, \forall i,\alpha]$ to $A=k[U_\alpha\,|\, |\alpha|=d]$ (and also, by abusing notation, its canonical extension to polynomial rings) which is such that $\rho(f_0)=f$ and $\rho(f_i)=\partial_i f$ for all $i=1,\ldots,n-1$. It is easy to check that $\rho(g_0)=g$, $\rho(\varphi_0)=h$ and that $\rho(g_i)=\partial_i g$, $\rho(\varphi_i)=\partial_i h$ for all $î=1,\ldots,n-1$. 	Moreover, $\rho$ is isobaric with respect to the Zariski grading of $k[V_{i,\alpha}\, \forall i,\alpha]$ and $A$ because each variable $U_{\alpha}$ has the same Zariski weight in $f$ and $\partial_{1}f,\ldots,\partial_{n-1}f$. Therefore, the specialization of \eqref{eq:redresgen} yields the equality
\begin{multline*}
\Res(f,\partial_1f,\ldots,\partial_{n-1}f)=\Res(g,\partial_1g,\ldots,\partial_{n-1}g)\Red(h,\partial_1h,\ldots,\partial_{n-1}h) \\ 
+ \textrm{terms of weight} > (d-\mu)(d-1-\mu)^{n-1}.	
\end{multline*}
By Proposition \ref{firstfp}, we deduce that
\begin{multline*}
	\Disc(f)\Disc(\bar{f})=\Disc(g)\Disc(\bar{g}) \Red(h,\partial_1h,\ldots,\partial_{n-1}h)
	+ \\ \textrm{terms of weight} > (d-\mu)(d-1-\mu)^{n-1}.	
\end{multline*}
But $\Disc(g)\neq 0$, $\Disc(\bar{g})\neq 0$ and by Proposition \ref{prop:Rednonzero} $\Red(h,\partial_1h,\ldots,\partial_{n-1}h)\neq 0$. Since $\Disc(\bar{f})$, $\Disc(\bar{g})$ and $\Red(h,\partial_1h,\ldots,\partial_{n-1}h)\neq 0$ have null Zariski weight and $\Disc(g)$ is isobaric of Zariski weight $(d-\mu)(d-1-\mu)^{n-1}$, we deduce that $\Disc(f)\in A$ is of valuation $(d-\mu)(d-1-\mu)^{n-1}$ with respect to Zariski weight as claimed.

Pushing further the computations, we see that 
$$\Disc(\bar{f})=\Disc(\bar{h}) \textrm{ divides } \Red(h,\partial_1h,\ldots,\partial_{n-1}h)$$
 and hence we deduce the formula of the theorem. To see this property, notice that the reduced resultant is a reduced inertia form, that is to say that there exists an integer $N$ such that
$$ (X_1,\ldots,X_{n-1})^N \Red(h,\partial_1h,\ldots,\partial_{n-1}h) \subset (h,\partial_1h,\ldots,\partial_{n-1}h).$$
Specializing $X_n$ to 0, we get
$$ (X_1,\ldots,X_{n-1})^N \Red(h,\partial_1h,\ldots,\partial_{n-1}h) \subset (\bar{h},\partial_1{\bar{h}},\ldots,\partial_{n-1}\bar{h})\subset A[X_1,\ldots,X_{n-1}]$$
from we deduce the claimed property by Proposition \ref{prop:DiscIrredZZ}.
\end{proof}

We are now ready to extend Proposition \ref{prop:d=2} to the generic homogeneous polynomial of arbitrary degree $d\geq 2$.

\begin{thm}\label{thm:disc-irred}
	Let $k$ be a domain and
	  $f=\sum_{|\alpha|=d}U_{\alpha}X^{\alpha}$ be the generic
	  homogeneous polynomial of degree $d\geq 2$ over $k$. If $\mathrm{char}(k)\neq 2$ or $n$ is odd, then $\Disc(f)$ is a prime polynomial in ${}_k A$ that generates $\pp$. Otherwise, if $\mathrm{char}(k)= 2$ and $n$ is even, then $\Disc(f)=P^2$ where $P$ is a prime polynomial that generates $\pp$.
\end{thm}
\begin{proof} By Corollary \ref{cor:discgeomreduced}, there exists an invertible element $c$ in $k$, a prime polynomial $P$ that generates $\pp$ and an integer $r$ such that $\Disc(f)=c.P^r$. 
	
	Now, grading $A$ with the Zariski weight, for all integer $1\leq \mu\leq d-2$ Theorem \ref{thm:discweight} shows that
	$$\Disc(f)=Q_\mu(f).\Disc(g)+ \textrm{ terms of weight } > (d-\mu)(d-\mu-1)^{n-1}$$ 
	where $Q_\mu(f)$ has weight zero and $\Disc(g)$ is isobaric of weight $(d-\mu)(d-\mu-1)^{n-1}$. Let $P_s$ be the isobaric part of smallest weight $s$ of $P$. Then, we deduce that for all integer $1\leq \mu\leq d-2$
	$$ Q_\mu(f).\Disc(g) = c.(P_s)^r.$$
In particular, if $\mu=d-2$ then $g$ is the generic homogeneous polynomial in $X_1,\ldots,X_n$ of degree 2. But by Proposition \ref{prop:d=2} we know that $\Disc(g)$ is prime if $n$ is odd or $2\neq 0$ in $k$, and that it is equal to the square of a prime polynomial otherwise. We deduce that $r=1$ in the first case and that necessarily $r\leq 2$ in the second case.

\medskip

Assume now that $2=0$ in $k$ and $n$ is even. We have just seen that $r\in \{1,2\}$. We claim that in this case, the canonical projection $\Proj(B)\rightarrow \Spec(A)$ is not birational onto its image $\Spec(A/\pp)$. This implies that $r$ cannot be equal to 1, so $r=2$ and $\Disc(f)=c.P^2$. Then, to conclude observe that ${}_{\ZZ/2\ZZ}\Disc(f)$ is a square (necessarily $c=1$ in this case), so that we deduce that $c$ is actually a square in $k$ via the canonical specialization from $\ZZ/2\ZZ$ to $k$. It follows that ${}_k \Disc(f)=(uP)^2$ where $u^2=c$ and $u$ is an invertible element in $k$, and the claimed result follows as $uP$ is an irreducible element that generates $\pp$.

To prove that $\Proj(B)\rightarrow \Spec(A)$ is not birational, we examine the module of relative differentials $\Omega_{B_{(X_n)}/A}$. In the following section, we will prove in Lemma \ref{lem:Hessnonzero} that it is isomorphic to the cokernel of a Hessian matrix. Moreover, under the assumptions that $2=0$ in $k$ and $n$ is even it turns out that the determinant of this Hessian matrix is equal to zero (see the beginning of Section \ref{sec:inertHessain} below). Consequently, the projection $\Proj(B)\rightarrow \Spec(A)$ can not be birational.
\end{proof}

\subsection{Inertia forms and the Hessian}\label{sec:inertHessain}

Let $k$ be a commutative ring. Given a polynomial $f \in k[X_1,\ldots,X_n]$, we will denote by $\Hess(f)$, and call it the Hessian of $f$, the determinant of the (symmetric) matrix
$$\HH(f):=\left( \frac{\partial ^2 f}{\partial X_i \partial X_j}
\right)_{1\leq i,j \leq n}.$$

When $2=0$ in $k$, the elements on the diagonal of $\HH(f)$ all vanish and $\HH(f)$ is then a skew-symmetric matrix. Consequently, $\Hess(f)=0$ if $n$ is odd and $\Hess(f)$ is the square of a polynomial (its Pfaffian) if $n$ is even. Regarding this behavior, the case where $f$ is a generic polynomial of degree 2 is particularly instructive.

\begin{lem}\label{lem:Hessnonzero} Set $A:=k[U_{i,j} \,|\, 1\leq i<j\leq n]$ and let 
	$$f:=\sum_{1\leq i\leq j\leq n}U_{i,j}X_iX_j \in A[X_1,\ldots,X_n]$$ 
be the generic homogeneous polynomial of degree $2$ over the ring $k$. If $n$ is even or if $2$ is a nonzero divisor in $k$ then $\Hess(f)$ is a nonzero divisor in $A$. 
\end{lem}
\begin{proof} If $n$ is even, the monomial $U_{1,2}^2U_{3,4}^2\ldots U_{n-1,n}^2$ appears in $\Hess(f)$ with a coefficient $\pm 1$ (to see it, one can for instance specialize all the other variables to zero). We deduce that the $k$-content of $\Hess(f)$ is equal to $k$ and therefore that $\Hess(f)$ is a nonzero divisor in $A$ by  Dedekind-Mertens Lemma.
	
	Now, assume that $n$ is odd and that $2$ is a nonzero divisor in $k$. By specializing $U_{1,j}$ to 0 for all $1<j\leq n$, $\Hess(f)$ specializes to 
	$2U_{1,1}\Hess(g)$
	where $g=\sum_{2\leq i \leq j\leq n} U_{i,j}X_iX_j$. But since $n-1$ is even, $\Hess(g)$ is a nonzero divisor in $A$ and it follows that $\Hess(f)$ is also a nonzero divisor.
\end{proof}

\begin{prop}\label{prop:HessianProp} Set $A:=k[U_{\alpha}\,|\, |\alpha|=d]$ and let 
	$$f:=\sum_{|\alpha|=d}U_{\alpha}X^{\alpha} \in A[X_1,\ldots,X_n]$$
	 be the generic homogeneous polynomial of degree $d$ over the ring $k$. If $n$ is odd or if $2$ is a nonzero divisor in $k$ then the determinant
	\begin{equation}\label{eq:D}
		\det  \left( \frac{\partial ^2 f}{\partial X_i \partial X_j}
		\right)_{1\leq i,j \leq n-1}		
	\end{equation}
is a nonzero divisor in the quotient ring $\quotient{A[X_1,\ldots,X_n]}{\TF_\mm\left(\Dc \right)}$. 
\end{prop}
\begin{proof} The case $n=1$ being trivially correct, we assume that $n\geq 2$. We first prove the claimed result under the assumption that $k$ is a domain. In this case, $\TF_\mm\left(\Dc \right)$ is a prime ideal by Corollary \ref{cor:TFprime} and hence we have to show that
	\begin{equation}\label{eq:detTFm}
		\det  \left( \frac{\partial ^2 f}{\partial X_i \partial X_j}	\right)_{1\leq i,j \leq n-1} \notin \TF_\mm\left(\Dc \right)
	\end{equation}
But it is enough to exhibit a particular specialization for which this property holds. So consider the specialization the sends $f$ to the polynomial
	$$ h:= \left(\sum_{1\leq i\leq j\leq n-1} U_{i,j} X_iX_j \right)X_n^{d-2} \in k[U_{i,j}\,|\, 1\leq i\leq j \leq n-1][X_1,\ldots,X_n].$$
	Denoting $g:=\sum_{1\leq i\leq j\leq n-1} U_{i,j} X_iX_j$, we have
	$$\det  \left( \frac{\partial ^2 h}{\partial X_i \partial X_j}
	\right)_{1\leq i,j \leq n-1} = \Hess(g)X_n^{(d-2)(n-1)}.$$
	Therefore, specializing further the variable $X_n$ to 1, we see that to prove \eqref{eq:detTFm} it is sufficient to prove that 
	\begin{multline*}
		\Hess(g) \notin (g,\partial_1 g,\ldots,\partial_{n-1}g,(d-2)g) \\ =(g,\partial_1 g,\ldots,\partial_{n-1}g)\subset k[U_{i,j}\,|\, { 1\leq i,j \leq n-1}][X_1,\ldots,X_{n-1}].		
	\end{multline*}
	But this holds because the ideal $(g,\partial_1 g,\ldots,\partial_{n-1}g)$ is nonzero and is contained in the ideal $(X_1,\ldots,X_{n-1})$, whereas $\Hess(g)$ belongs to $k[U_{i,j}\,|\, { 1\leq i,j \leq n-1}]$ and is nonzero by Lemma \ref{lem:Hessnonzero}.
	
	\medskip
	
	We now turn to the proof in the case $k$ is an arbitrary commutative ring. Let $D$ stands for the determinant \eqref{eq:D}. We begin with the case where $n$ is odd. By \eqref{BXi-iso}, ${}_\ZZ B_{X_n}$ is a free abelian group. Moreover, from what we have just proved under the assumption that $k$ is a domain, we deduce that the multiplication by $D$ in ${}_\ZZ B_{X_n}$ and ${}_{\ZZ/p\ZZ} B_{X_n}$, $p$ a prime integer, are all injective maps. Denoting by ${}_\ZZ Q$ the quotient abelian group of the multiplication by $D$ in ${}_\ZZ B_{X_n}$, that is to say we have he exact sequence of abelian groups
\begin{equation*}
	0 \rightarrow {}_\ZZ B_{X_n} \xrightarrow{\times D} {}_\ZZ B_{X_n} \rightarrow {}_\ZZ Q \rightarrow 0,
\end{equation*}	
we deduce that ${}_\ZZ Q$ is torsion free (for $\mathrm{Tor}^\ZZ_1(\ZZ/p\ZZ,{}_\ZZ Q)=0$ for all prime integer $p$) and hence is flat. By a classical property of flatness we obtain that $\mathrm{Tor}_1^{\ZZ}({}_\ZZ Q,k)=0$ and therefore that the multiplication by $D$ in ${}_k B_{X_n}$ is an injective map, i.e.~ $D$ is a nonzero divisor in ${}_k B_{X_n}$. Finally, since 
\begin{equation}\label{eq:TFker}
	\TF_\mm(\Dc)=\ker({}_k C \rightarrow {}_k B_{X_n})	
\end{equation}
by Corollary \ref{cor:TFprime}, it follows that $D$ is a nonzero divisor in $\quotient{{}_k C}{\TF_\mm(\Dc)}$.

We can proceed similarly to prove the claimed result in the case where $n$ is even. The multiplication by $D$ in ${}_\ZZ B_{X_n}$ and ${}_{\ZZ/p\ZZ} B_{X_n}$, $p$ a prime but odd integer, are all injective maps. It follows that after inversion of 2 we obtain the exact sequence 
\begin{equation*}
	0 \rightarrow {}_{\ZZ[\frac{1}{2}]} B_{X_n} \xrightarrow{\times D} {}_{\ZZ[\frac{1}{2}]} B_{X_n} \rightarrow {}_{\ZZ[\frac{1}{2}]} Q \rightarrow 0
\end{equation*}
where the $\ZZ[\frac{1}{2}]$-module ${}_{\ZZ[\frac{1}{2}]} Q$ is torsion free and is hence flat. Consequently, if 2 is a unit in $k$ we immediately deduce by tensorization by $k$ over $\ZZ[\frac{1}{2}]$ that the multiplication by $D$ in ${}_k B_{X_n}$ is an injective map. Now, if $2$ is a nonzero divisor in $k$ then $k$ can be embedded in $k[\frac{1}{2}]$. This induces the inclusion of ${}_k B_{X_n}$ in ${}_{k[\frac{1}{2}]} B_{X_n}$. But we have just proved that $D$ is a nonzero divisor in ${}_{k[\frac{1}{2}]} B_{X_n}$, so we deduce that it is also a nonzero divisor in ${}_k B_{X_n}$ and hence also a nonzero divisor in $\quotient{{}_k C}{\TF_\mm(\Dc)}$ by \eqref{eq:TFker}.
\end{proof}

\begin{thm} Set $A:=k[U_{\alpha}\,|\,|\alpha|=d]$ and let $$f:=\sum_{|\alpha|=d}U_{\alpha}X^{\alpha} \in A[X_1,\ldots,X_n]$$
	 be the generic homogeneous polynomial of degree $d$ over $k$. If  
	$n$ is odd or if $2$ is a nonzero divisor in $k$ then
	$$\TF_\mm (f,\partial_1f, \ldots, \partial_n f)\cap A \subset (\partial_1 \tilde{f},\ldots, \partial_{n-1}\tilde{f})^2 + (\tilde{f}),$$
	where, for all polynomial $P(X_1,\ldots,X_n)$, the notation $\tilde{P}$ stands for the polynomial $P(X_1,\ldots,X_{n-1},1)$.
\end{thm}

\begin{proof} Let $a \in \TF_\mm (f,\partial_1f, \ldots, \partial_n f)\cap A$. There exists an integer $N$ such that $X_n^{N-1}a$ belongs to the ideal $(f, \partial_1f,\ldots,\partial_n f)$. Moreover, using the Euler identity $df=\sum_{i=1}^n X_i\partial_if$, we obtain that
	$X_n^Na$ belongs to the ideal $(f,\partial_1f,\ldots,\partial_{n-1}f)$ and therefore that there exist polynomials $P_1,\ldots,P_{n-1}$ and $Q$ in $A[X_1,\ldots,X_n]$ such that
	\begin{equation}\label{eq:HessianSquare}
	X_n^N a=P_1 \partial_1f+\cdots+P_{n-1}\partial_{n-1}f+Qf.	
	\end{equation}
By applying the derivation $\partial_j(-)$ for all $j=1,\ldots,n-1$, we obtain the following equalities: 
$$ 
\forall j \in \{ 1,\ldots,n-1 \}, \hspace{.5cm} \sum_{i=1}^{n-1} P_i 
\frac{\partial^2 f}{\partial X_i\partial X_j}=0 
\textrm{ mod } (f,\partial_1 f,\ldots,\partial_{n-1}f).
$$
By Cramer's rules, it follows that for all $i=1,\ldots,n-1$ we have
$$P_i . \det  \left( \frac{\partial ^2 f}{\partial X_i \partial X_j}
\right)_{1\leq i,j \leq n-1}  \in (f,\partial_1f,\ldots,\partial_{n}f) \subset \TF_\mm (f,\partial_1f, \ldots, \partial_n f).$$
But by Proposition \ref{prop:HessianProp}, the determinant
$$\det  \left( \frac{\partial ^2 f}{\partial X_i \partial X_j}
\right)_{1\leq i,j \leq n-1}$$
is not a zero divisor in the quotient ring of $A[X_1,\ldots,X_n]$ by the inertia form ideal $\TF_\mm (f,\partial_1f, \ldots, \partial_n f)$. Therefore, we deduce that $P_i \in \TF_\mm (f,\partial_1f, \ldots, \partial_n f)$ for all $i=1,\ldots,n-1$ and hence, using again Euler identity, that
$$\tilde{P}_i \in (\tilde{f},\partial_1 \tilde{f},\ldots, \partial_{n-1} \tilde{f}).$$
Coming back to the definition \eqref{eq:HessianSquare} of the $P_i$'s, the claimed result is proved.
 \end{proof}

An immediate consequence of this theorem is the 

\begin{cor} For any commutative ring $k$ and any homogeneous polynomial $f \in k[X_1,\ldots,X_n]$, we have
	$$\Disc(f) \in (\partial_1 \tilde{f},\ldots, \partial_{n-1}\tilde{f})^2 + (\tilde{f}).$$
\end{cor}

We end this paragraph with the computation of the module of relative differentials $\Omega_{B_{(X_n)}/A}$ induced by the canonical inclusion $A\rightarrow B_{(X_n)}$. 

\begin{lem}\label{lem:OBA}  For any commutative ring $k$, the module $\Omega_{B_{(X_n)}/A}$ of relative differential of $B_{(X_n)}$ over $A$ is isomorphic to the cokernel of the map 
	$$\bigoplus_{i=1}^{n-1}\frac{A[X_1,\ldots,X_{n-1}]}{(\tilde{f},\partial_1\tilde{f},\ldots,\partial_{n-1}\tilde{f})} \xrightarrow{\Hess(\tilde{f})} 
	\bigoplus_{i=1}^{n-1}\frac{A[X_1,\ldots,X_{n-1}]}{(\tilde{f},\partial_1\tilde{f},\ldots,\partial_{n-1}\tilde{f})}$$
	whose matrix in the canonical basis is given by the Hessian matrix $\HH(\tilde{f})$.
\end{lem}
\begin{proof} By definition of $B$, it is clear that 
	$$B_{(X_n)}\simeq A[X_1,\ldots,X_{n-1}]/(\tilde{f},\partial_1\tilde{f},\ldots,\partial_{n-1}\tilde{f}).$$ 
	We need to introduce some notation. We can decompose $f$ as a sum
	$$f=f_d+f_{d-1}X_n+\cdots+f_{d-2}X_n^{2}+f_{1}X_n^{d-1}+f_{0}X_n^{d}$$
where the $f_i$'s are homogeneous polynomials in $X_1,\ldots,X_{n-1}$ of degree $d-i$. 	Wet $h:=f-f_{1}X_n^{d-1}-f_{0}X_n^{d}$ and we rename the coefficients $U_{\alpha}$, $\alpha_n\geq d-1$, of $f$ by setting
$$ f_{1}=\EE_1X_1+\EE_2X_2±\cdots+\EE_{n-1}X_{n-1},\ \ f_{0}=\EE_n.$$
Setting $D:=k[X_1,\ldots,X_{n-1}][U_{\alpha}\,|\, \alpha_n\leq d-2]$, we define a $k$-linear map $\lambda$ from $B_{(X_n)}$ to $D$ as follows:
\begin{eqnarray}\label{eq:BD}
	X_i & \mapsto & X_i, \ \ i=1,\ldots,n-1 \\ \nonumber
	U_{\alpha} & \mapsto & U_{\alpha}, \ \ \alpha_n\leq d-2 \\ \nonumber
	\EE_i & \mapsto & -\partial_i\tilde{h}, \ \ i=1,\ldots,n-1 \\ \nonumber
	\EE_n & \mapsto & -\tilde{h}+\sum_{i=1}^{n-1}X_i\partial_i\tilde{h}
\end{eqnarray}
It is clear that $\lambda$ is surjective. Moreover, observe that $\tilde{f}=\tilde{h}+\EE_n+\sum_{i=1}^{n-1}\EE_iX_i$, so that $\partial_i\tilde{f}=\partial_i\tilde{h}+\EE_i$ for all $i=1,\ldots,n-1$, and hence we deduce that $\lambda$ is an isomorphism.

Now, $B_{(X_n)}$ is an $A$-algebra by the canonical inclusion of $A$ in $B_{(X_n)}$. Using the isomorphism $\lambda$, we get that $\Omega_{B_{(X_n)}/A}\simeq \Omega_{D/A}$ and $A\rightarrow D$ is given by \eqref{eq:BD} (without the $X_i$'s that have been removed). 
Setting $\bar{A}=k[U_{\alpha} \,|\, \alpha_n\leq d-2]$, so that $A=\bar{A}[\EE_1,\ldots,\EE_{n}]$,
we get maps of rings $\bar{A}\rightarrow A \rightarrow D$ and the relative cotangent sequence 
$$ D\otimes_A \Omega_{A/\bar{A}} \xrightarrow{can} \Omega_{D/\bar{A}} \rightarrow \Omega_{D/A} \rightarrow 0$$
which is exact. Since $\Omega_{A/\bar{A}}\simeq \oplus_{i=1}^n A \dd \EE_i$ and $\Omega_{D/\bar{A}}\simeq \oplus_{i=1}^{n-1}D\dd X_i$, the map $can$ in this sequence can be represented by a matrix in the basis $d\EE_1,\ldots,d\EE_n$ and $\dd X_1,\ldots \dd X_{n-1}$ respectively. By straightforward computations, we get
$$ can(\dd \EE_i)= - \left( \sum_{j=1}^{n-1} \frac{\partial^2 \tilde{h}}{\partial X_i \partial X_j} \dd X_j \right)=
- \left( \sum_{j=1}^{n-1} \frac{\partial^2 \tilde{f}}{\partial X_i \partial X_j} \dd X_j \right),  \hspace{.2cm} i=1,\ldots,n-1,$$
and
\begin{multline*}
	 can(\dd \EE_n)= \sum_{i=1}^{n-1}\left(  \sum_{j=1}^{n-1} X_j\frac{\partial^2 \tilde{h}}{\partial X_i \partial X_j}   \right)\dd X_i=\sum_{i=1}^{n-1}\left(  \sum_{j=1}^{n-1} X_j\frac{\partial^2 \tilde{f}}{\partial X_i \partial X_j}   \right)\dd X_i \\ 
	=-\sum_{j=1}^{n-1}X_j 
	can(\dd \EE_j)	
\end{multline*}
so that the first $n-1$ columns of this matrix corresponds to $-\Hess(\tilde{f})$ and its last column is the span of the $n-1$ first ones. Therefore,  the image of $can$ is isomorphic to the image of the map $D^{n-1}\rightarrow D^{n-1}$ defined by the matrix $-\Hess(\tilde{f})$, and the claimed result follows.
\end{proof}

The computation done in this lemma shows that the unramified points of $\Proj(B)$ over $\Spec(A)$ are the non-degenerated quadratic points, that is to say the points where the Hessian of $\tilde{f}$ does not vanish. We used it at the end of the proof of Theorem \ref{thm:disc-irred} to show that the canonical projection of $\Proj(B)$ over $\Spec(A)$ is not birational if $\mathrm{char}(k)=2$ and $n$ is even under the assumption that $k$ is a domain. If $n$ is odd or $2$ is a nonzero divisor in $k$ then this projection is birational (without assuming that $k$ is a domain). The purpose of the next section is to prove this fact by providing an explicit blowup structure to $\Proj(B)$.

\subsection{Effective blow-up structure}

For the sake of simplicity in the text, we introduce a particular notation for some coefficients $U_\alpha$ of the generic homogeneous polynomial $f\in {}_k A$ of degree $d\geq 2$~:
$$f(X_1,\ldots,X_n)=\EE_1 X_1X_n^{d-1}+\EE_2 X_2X_n^{d-1}+\cdots +\EE_{n-1} X_{n-1}X_n^{d-1}+\EE_n X_n^d + \cdots$$
Moreover, we introduce $n-1$ polynomials
$$g_i(X_1,\ldots,X_n)=\sum_{|\beta|=d-1} V_{i\beta}X^\beta, \ \ i=1,\ldots,n-1$$
and define the coefficient ring 
$${}_kA'={}_k A[V_{i\beta}\,|\, 1\leq i \leq n-1, |\beta|=d-1]$$ 
so that $f$ and $g_1,\ldots,g_{n-1}$
belong to ${}_k A'[X_1,\ldots,X_n]$. For the sake of simplicity, we will omit the subscript $k$ in the notation whenever there is no possible confusion. 

\medskip

The resultant $S:=\Res(\partial_1f,\ldots,\partial_{n-1}f,f) \in A$ can be obtained by specialization of the resultant $R:=\Res(g_1,\ldots,g_{n-1},f) \in A'$. More precisely, for all integer $i=1,\ldots,n-1$ we have
$$\frac{\partial{f}}{\partial X_i}=\sum_{|\alpha|=d, \alpha_i\geq 1} \alpha_i U_{\alpha} \frac{X^\alpha}{X_i} = \sum_{|\beta|=d-1}(\beta_i+1)U_{\beta+e_i}X^\beta$$
where $e_i$ stands for the multi-index such that $X^{e_i}=X_i$ for all $i=1,\ldots,n-1$. Thus, we define the specialization
\begin{eqnarray*}
	\rho : A' & \rightarrow & A \\
	V_{i\beta} &  \mapsto & (\beta_i+1)U_{\beta+e_i}, \ \ i=1,\ldots,n-1 \\
	U_{\beta} & \mapsto & U_{\beta}
\end{eqnarray*}
so that $\rho(R)=S$. Notice that we also have $\rho( \partial R/\partial \EE_n)=\partial S/\partial \EE_n$. Now, set $D:=\Disc(f)\in A$ and recall that $\bar{f}(X_1,\ldots,X_{n-1}):=f(X_1,\ldots,X_{n-1},0)$.  
\begin{prop} There exist polynomials $\Delta_1(f),\ldots,\Delta_n(f) \in {}_\ZZ A$ such that 
$$\Disc(\bar{f})\Delta_i(f)=\rho\left(\frac{\partial R}{\partial \EE_i}\right) \in {}_\ZZ A.$$	
For any commutative ring $k$, we define the polynomials $\Delta_1(f),\ldots,\Delta_n(f) \in {}_kA$ by change of basis $\ZZ\rightarrow k$. 

Moreover, 
\begin{equation}\label{eq:Dn=partial}
	\Delta_n(f)=\frac{\partial D}{\partial \EE_n} \in {}_kA
\end{equation}
and for all $1\leq i,j \leq n$ we have
\begin{equation}\label{eq:DeltaTF}
\Delta_i(f)X_j -\Delta_j(f)X_i\in \TF_\mm(\partial_1 f, \ldots,\partial_{n-1}f, f) \subset {}_kA[X_1,\ldots,X_n].	
\end{equation}
\end{prop} 
\begin{proof} We begin by proving the claim about $\Delta_n(f)$. For that purpose, introduce a new indeterminate $T$. By Taylor expansion we have
	$$\Res(g_1,\ldots,g_{n-1},f+TX_n^d)-R=T\frac{\partial R}{\partial \EE_n} \mod (T^2) \in {}_kA'[T].$$
Applying the specialization $\rho$ and the definition of the discriminant, we obtain
\begin{equation}\label{eq:Deltan}
	\Disc(\bar{f})(\Disc(f+TX_n^d)-\Disc(f))=T\rho\left(\frac{\partial R}{\partial \EE_n} \right) \mod (T^2) \in 	{}_kA[T].	
\end{equation}
But the Taylor expansion also yields the equality
\begin{equation}\label{eq:DeltanDisc}
	\Disc(f+TX_n^d)-\Disc(f)=T \frac{\partial D}{\partial \EE_n}  \mod (T^2) \in 	{}_kA[T].	
\end{equation}
Therefore, combining \eqref{eq:Deltan} and \eqref{eq:DeltanDisc} we deduce that 
\begin{equation}\label{eq:Deltann}
 \Disc(\bar{f}) \, \frac{\partial D}{\partial \EE_n} =\rho\left(\frac{\partial R}{\partial \EE_n} \right)= \frac{\partial S}{\partial \EE_n}  \in {}_kA
\end{equation}
so that the claim $\Delta_n(f)= \partial D / \partial \EE_n$ in ${}_k A$ is proved since $\Disc(\bar{f})$ is a nonzero divisor by Corollary \ref{cd1:nzd}.

Now, we turn to the polynomials $\Delta_1(f),\ldots,\Delta_{n-1}(f)$ and hence we assume that $k=\ZZ$. From \cite[Lemme 4.6.1]{J91}, we know that for all multi-index $\alpha$ such that $|\alpha|=d$ we have
$$ X_n^{d}\frac{\partial R}{\partial U_\alpha} - X^\alpha  \frac{\partial R}{\partial \EE_n} \in \TF_\mm(g_1,\ldots,g_n).$$
Moreover, \cite[Lemma 4.6.6]{J91} then shows that the specialization of $X_i$ by $\partial R/\partial \EE_i$ for all $i=1,\ldots,n$ yields
$$\left(\frac{\partial R}{\partial \EE_n}\right)^d \frac{\partial R}{\partial U_\alpha} - \left(\frac{\partial R}{\partial \EE_1}\right)^{\alpha_1}\cdots \left(\frac{\partial R}{\partial \EE_n}\right)^{\alpha_n} \frac{\partial R}{\partial \EE_n} \in R.{}_\ZZ A'$$
By the properties of the resultant, $R$ is irreducible, $\partial R/\partial \EE_n\neq 0$ and $\partial R/\partial \EE_n\notin R.{}_\ZZ A'$ so we deduce that
$$\left(\frac{\partial R}{\partial \EE_n}\right)^{d-1} \frac{\partial R}{\partial U_\alpha} - \left(\frac{\partial R}{\partial \EE_1}\right)^{\alpha_1}\cdots \left(\frac{\partial R}{\partial \EE_n}\right)^{\alpha_n} \in R.{}_\ZZ A'.$$
Taking suitable choices for the multi-index $\alpha$, we finally obtain that for all integer $i=1,\ldots,n-1$
\begin{equation}\label{eq:Di1}
	\left(\frac{\partial R}{\partial \EE_n}\right)^{d-1} \frac{\partial R}{\partial U_\alpha} - \left(\frac{\partial R}{\partial \EE_i}\right)^{d} \in R.{}_\ZZ A'.
\end{equation}
Now, since $\Disc(\bar{f})$ divides $S=\rho(R)$, by definition of the discriminant and divides $\rho(\partial R/\partial \EE_n)$ by \eqref{eq:Deltann}, we deduce that it also divides $\rho(\partial R/\partial \EE_i)^d$ for all integer $i=1,\ldots,n-1$ by specialization of \eqref{eq:Di1} under $\rho$. But $\Disc(\bar{f})$ is irreducible in ${}_\ZZ A$, so we finally deduce that $\Disc(\bar{f})$ divides $\rho(\partial R/\partial \EE_i)$ for all $i=1,\ldots,n-1$ and hence the existence of the polynomials $\Delta_1(f),\ldots,\Delta_{n-1}(f) \in {}_\ZZ A$.

It remains to prove \eqref{eq:DeltaTF}. Recall from \cite[Lemma 4.6.1, (4.6.3)]{J91} that for all $1\leq i,j \leq n$ we have
$$\frac{\partial R}{\partial \EE_i}X_j-\frac{\partial R}{\partial \EE_j}X_i \in \TF_\mm(g_1,\ldots,g_{n-1},f) \subset {}_kA'[X_1,\ldots,X_n].$$
Applying the specialization $\rho$, we deduce that
$$\Disc(\bar{f})\left(\Delta_i(f)X_j-\Delta_j(f)X_i\right) \in \TF_\mm(\Dc).$$
Therefore, we deduce that \eqref{eq:DeltaTF} holds if $k=\ZZ$ because in this case $\Disc(\bar{f})$ is irreducible and does not divide $\Disc(f)$, hence does not belong to the prime ideal $\TF_\mm(\Dc)$. Finally, \eqref{eq:DeltaTF} holds for any $k$ by change of basis $\ZZ\rightarrow k$.
\end{proof}

We are now ready to define a map from ${}_kC$ to a Rees algebra. Recall that $\pp:=\TF_\mm({}_k\Dc)_0 \subset {}_k A$ and denote by $\bar{\Delta}_i$ the image of $\Delta_i(f)$ by the canonical map $A\rightarrow A/\pp$ for all $i=1,\ldots,n$. Introducing a new indeterminate $T$, we define the $A$-algebra morphism
\begin{eqnarray*}
	\varphi : {}_k C={}_k A[X_1,\ldots,X_n] & \rightarrow & \Rees_{A/\pp}(\bar{\Delta}_1,\ldots,\bar{\Delta}_n) \subset A/\pp[T] \\
	h=\sum_{\nu\in\NN} h_\nu(X_1,\ldots,X_n) & \mapsto & \sum_{\nu\in\NN} h_\nu(\bar{\Delta}_1,\ldots,\bar{\Delta}_n)T^\nu
\end{eqnarray*}
where the notation $h_\nu$ stands for the homogeneous part of degree $\nu$ of $h\in {}_kC$. Notice that it is a graded and surjective map.

\begin{lem} With the above notation, $\varphi$ vanishes on $\TF_\mm(\Dc)$.
\end{lem}
\begin{proof} Since $\varphi$ is graded, it is sufficient to check the claimed property on graded parts. 
	Let $h\in C_\nu$. By using \eqref{eq:DeltaTF}, we obtain that
	\begin{equation}\label{eq:birat1}
		X_n^\nu h(\Delta_1(f),\ldots,\Delta_n(f))-\Delta_n(f)^\nu h(X_1,\ldots,X_n) \in \TF_\mm(\Dc).
	\end{equation}
	It follows that if $h\in C_\nu\cap \TF_\mm(\Dc)$ then $X_n^\nu h(\Delta_1(f),\ldots,\Delta_n(f)) \in \TF_\mm(\Dc)$. Similarly, we get that 
	$X_j^\nu h(\Delta_1(f),\ldots,\Delta_n(f)) \in \TF_\mm(\Dc)$ for all $j=1,\ldots,n$ and consequently, we deduce that
	$$h(\Delta_1(f),\ldots,\Delta_n(f)) \in \TF_\mm(\Dc)_0=\pp \subset {}_k A,$$
	hence $h_\nu(\bar{\Delta}_1,\ldots,\bar{\Delta}_n)=0 \in A/\pp$.
\end{proof}

As a consequence of this lemma, the morphism $\varphi$ induces
\begin{equation*}
	\bar{\varphi} : {}_k C/\TF_\mm(\Dc)={}_kB/H^0_\mm(B)  \rightarrow  \Rees_{A/\pp}(\bar{\Delta}_1,\ldots,\bar{\Delta}_n)
\end{equation*}
From a geometric point of view, $\bar{\varphi}$ defines a map from a blow-up variety to the discriminant variety. Below, we will prove that this map is an isomorphism under suitable assumptions. As a consequence, it will follow that the scheme morphism 
$$\Proj(B)=\Proj(B/H^0_\mm(B)) \rightarrow \Spec(A/\pp)$$ 
is birational since $\bar{\varphi}$ identifies $\Proj(B)$ to the blow-up of $\Spec(A/\pp)$ along the ideal $(\bar{\Delta}_1,\ldots,\bar{\Delta}_n)$.

\begin{lem}\label{lem:birat} Assume that $k$ is a domain and that  $n$ is odd or  $2\neq 0$ in $k$. Let $a \in \pp=\TF_\mm(\Dc)_0$. If $\partial a /\partial \EE_n=0$ then $a\in \TF_\mm(\Dc^2)_0$. In particular, if $\partial a /\partial \EE_n=0$ then $\partial a /\partial U_\alpha \in \pp$ for all multi-index $\alpha$ such that $|\alpha|=d$.
\end{lem}
\begin{proof} Let $a \in \pp=\TF_\mm(\Dc)_0$ such that $\partial a /\partial \EE_n=0$. by Corollary \ref{cor:TFprime}, there exits an integer $N$ and polynomials $P_1,\ldots,P_{n-1},Q \in{}_kA[X_1,\ldots,X_n]$ such that
	\begin{equation}\label{eq:TF1}
		X_n^Na=P_1\partial_1f+\partial_2f+\cdots+\partial_{n-1}f+Qf.
	\end{equation}
Since $\partial_if$ does not depend on $\EE_n$ for all $1\leq i\leq n-1$, by derivation with respect to $\EE_n$ we get
\begin{equation*}
	0=X_n^N\frac{\partial a}{\partial \EE_n}=\frac{\partial P_1}{\partial \EE_n}\partial_1f+\cdots+
	\frac{\partial P_{n-1}}{\partial \EE_n}\partial_{n-1}f+\frac{\partial Q}{\partial \EE_n}f+ QX_n^d.
\end{equation*}
It follows immediately that
$$X_n^dQ \in (f,\partial_1f,\ldots,\partial_{n-1}f)$$
and hence, by comparing with \eqref{eq:TF1}, we deduce that there exits polynomials $L_1,\ldots,L_{n-1},M \in{}_kA[X_1,\ldots,X_n]$ such that
\begin{equation}\label{eq:TF2}
	X_n^{N+d}a= L_1\partial_1f+L_2\partial_2 f+\cdots+L_{n-1}\partial_{n-1}f+Mf^2.
\end{equation}
Computing the derivatives with respect to $X_j$ for all $1\leq j\leq n-1$, we get the equalities
$$0=\sum_{i=1}^{n-1}L_i\frac{\partial^2f}{\partial X_i\partial X_j} + 
\sum_{i=1}^{n-1}\frac{\partial L_i}{\partial X_j}\frac{\partial f}{\partial X_i} +
2Mf\frac{\partial f}{\partial X_j}+\frac{\partial M}{\partial X_j}f^2, \ 1\leq j\leq n-1.$$
Hence, for all $1\leq j\leq n-1$ we have
$$\sum_{i=1}^{n-1}L_i\frac{\partial^2f}{\partial X_i\partial X_j} \in (f^2,\partial_1f,\ldots,\partial_{n-1}f)$$
and Cramer's rules show that for all $1\leq l\leq n-1$  we have
$$ \det \left( \frac{\partial^2f}{\partial X_i\partial X_j}  \right)_{1\leq i,j\leq n-1} L_l \in (f^2,\partial_1f,\ldots,\partial_{n-1}f).$$
Finally, by comparison with \eqref{eq:TF2} we obtain
$$ X_n^{N+d}a\det \left( \frac{\partial^2f}{\partial X_i\partial X_j}  \right)_{1\leq i,j\leq n-1} \in (f,\partial_1 f,\partial_2f,\ldots,\partial_{n-1}f)^2.$$
In other words, using the notation $\tilde{f}(X_1,\ldots,X_{n-1}):=f(X_1,\ldots,X_{n-1},1)$, we obtained that
\begin{equation}\label{eq:TF3}
	\Hess(\tilde{f}).a \in (\tilde{f},\partial_1 \tilde{f},\ldots,\partial_{n-1} \tilde{f})^2.
\end{equation}

Now, Proposition \ref{prop:HessianProp} implies that $\Hess(\tilde{f})$ is a nonzero divisor in the quotient ring 
${}_k\tilde{B}:={}_kA[X_1,\ldots,X_{n-1}]/(\tilde{f},\partial_1 \tilde{f},\ldots,\partial_{n-1} \tilde{f})$. Moreover, Proposition \ref{prop:regular-sequences}, (ii) shows that $\tilde{f},\partial_1 \tilde{f},\ldots,\partial_{n-1} \tilde{f}$ is a regular sequence in ${}_kA[X_1,\ldots,X_{n-1}]$ and hence 
$$\quotient{(\tilde{f},\partial_1 \tilde{f},\ldots,\partial_{n-1} \tilde{f})}{(\tilde{f},\partial_1 \tilde{f},\ldots,\partial_{n-1} \tilde{f})^2}$$
is a free $\tilde{B}$-module. Therefore, this and \eqref{eq:TF3} show that
$$ a \in (\tilde{f},\partial_1 \tilde{f},\ldots,\partial_{n-1} \tilde{f})^2.$$
Finally, using Corollary \ref{cor:TFprime} we conclude that
$a \in \TF_{(X_n)}(\Dc^2)=\TF_\mm(\Dc^2)$.
\end{proof}

\begin{cor}\label{cor:DEndiv} If $n$ is odd or if $2$ is a nonzero divisor in $k$ then $\partial D/\partial \EE_n$ is not a zero divisor in the quotient ring $\quotient{{}_kC}{\TF_\mm(\Dc)}$.
\end{cor}
\begin{proof} We first assume that $k$ is a domain. Then, observe that we can assume without loss of generality that $k$ is actually a field by extension to the fraction field of $k$. Now, if $\partial D/\partial \EE_n\neq 0$ then Lemma \ref{lem:birat} implies that $D$ divides $\partial D /\partial U_\alpha$ for all multi-index $\alpha$ such that $|\alpha|=d$ and hence that $\partial D /\partial U_\alpha=0$ for all $\alpha$ such that $|\alpha|=d$ by inspecting the degrees. If $k$ has characteristic zero then we deduce that $D=0$, a contradiction with Theorem \ref{thm:disc-irred}. If $k$ has characteristic $p>0$, then passing to the algebraic closure of $k$ (which is a perfect field) we get that $D$ must be some polynomial raised to the power $p$, again a contradiction with Theorem \ref{thm:disc-irred}.

It remains to prove that the claimed property holds for an arbitrary ring $k$, knowing that it is valid for a domain. To do this, we can proceed exactly as in the proof of Proposition \ref{prop:HessianProp}.
\end{proof}

We are now ready to prove the main result of this section.

\begin{thm} If $n$ is odd or $2$ is a nonzero divisor in $k$, then $\bar{\varphi}$ is an isomorphism.
\end{thm}
\begin{proof}	
	Since $\bar{\varphi}$ is graded and surjective, it is sufficient to show that it is injective on graded parts. So let $h \in C_\nu$ and assume that
	$h(\Delta_1(f),\ldots,\Delta_n(f))\in \pp$. Then, \eqref{eq:birat1} and \eqref{eq:Dn=partial} shows that 
	\begin{equation}\label{eq:DEn}
		\left(\frac{\partial D}{\partial \EE_n}  \right)^\nu h(X_1,\ldots,X_n) \in \TF_\mm(\Dc).		
	\end{equation}
But by Corollary \ref{cor:DEndiv}, $\partial D/\partial \EE_n$ is not a zero divisor in the quotient ring $\quotient{{}_kC}{\TF_\mm(\Dc)}$. Therefore \eqref{eq:DEn} implies that $h(X_1,\ldots,X_n) \in \TF_\mm(\Dc)$ and from here we deduce that $\bar{\varphi}$ is an isomorphism. 	
\end{proof}


\def\cprime{$'$}

\appendix

\section*{Appendix - Two formulas of F.~Mertens}\label{mertens}

\setcounter{equation}{0}
\renewcommand{\theequation}{A.\arabic{equation}}

In this appendix, 
we give rigorous proofs of two outstanding formulas that were given by Frantz Mertens around 1890 in its study of the resultant of homogeneous polynomials \cite{Mertens1886}.

\medskip

Let $R$ be a commutative ring and suppose given $n\geq 1$ homogeneous polynomials $f_1,\ldots,f_{n}$ in $R[X_1,\ldots,X_n]$ with positive degree $d_1,\ldots,d_{n}$ respectively, such that $\prod_{i=1}^n d_i>1$. Introducing news indeterminates $U_1,\ldots,U_n$, we define
$$\theta(U_1,\ldots,U_n):=\Res(f_1,\ldots,f_{n-1},\sum_{i=1}^n U_iX_i) \in R[U_1,\ldots,U_n]$$
and $\theta_i(U_1,\ldots,U_n):=\partial \theta/\partial U_i \in R[U_1,\ldots,U_n]$ for all $i=1,\ldots,n$. In addition, let $V_1,\ldots, V_n$, $W_1,\ldots,W_n$, $X,Y$ be a collection of some other new indeterminates and consider the ring morphisms
\begin{eqnarray*}
	\rho : R[U_1,\ldots,U_n] & \rightarrow & R[V_1,\ldots,V_n,W_1,\ldots,W_n][X_1,\ldots,X_n] \\
	U_i & \mapsto & V_i(\sum_{j=1}^nW_jX_j)-W_i(\sum_{j=1}^nV_jX_j).
\end{eqnarray*}
and 
\begin{eqnarray*}
	\overline{\rho} : R[U_1,\ldots,U_n] & \rightarrow & R[V_1,\ldots,V_n,W_1,\ldots,W_n][X,Y] \\
	U_i & \mapsto & V_iX+W_iY 
\end{eqnarray*}

\noindent {\bf First Mertens' formula:} 
$$\Res_{X,Y}(\overline{\rho}(\theta),\overline{\rho}(f_n(\theta_1,\ldots,\theta_n))) =(-1)^{d_1\ldots d_n}\Disc_{X,Y}(\overline{\rho}(\theta))^{d_n}\Res(f_1,\ldots,f_n).$$

\medskip

\noindent {\bf Second Mertens' formula:} 
$$\Res(f_1,\ldots,f_{n-1},\rho(f_n(\theta_1,\ldots,\theta_n))) =(-1)^{d_1\ldots d_n}\Disc_{X,Y}(\overline{\rho}(\theta))^{d_n}\Res(f_1,\ldots,f_n).$$

Notice that the subscript $X,Y$ is written to emphasize that the discriminant, or the resultant, is taken with respect to these two variables.

\medskip

\begin{proof} We begin with the proof of the first formula and then we will deduce the second formula form the first one. Observe that we can assume that $R$ is actually the universal ring of coefficients of the polynomials $f_1,\ldots,f_n$ that we will denote by $A$.
	
\medskip
	
	From definition, $\theta$ is an inertia form of the polynomials  $f_1,\ldots,f_{n-1},\sum_{i=1}^n U_iX_i$ with respect to $(X_1,\ldots,X_n)$: there exists an integer, say $N$, and polynomials $h_1,\ldots,h_{n-1},h$ in the polynomial ring $A[U_1,\ldots,U_n][X_1,\ldots,X_n]$  such that
	$$X_n^N\theta=h_1f_1+\cdots+h_{n-1}f_{n-1}+h(\sum_{i=1}^n U_iX_i).$$
A simple computation then shows that $X_i\theta_j-X_j\theta_i$ is an inertia form of the same polynomials for all couple $(i,j)$. By successive iterations, we deduce that $X_n^{d_n}f_n(\theta_1,\ldots,\theta_n)-\theta_n^{d_n}f_n(X_1,\ldots,X_n)$ is also such an inertia form. Finally, we obtain that $f_n(\theta_1,\ldots,\theta_n)$  is an inertia forms of the polynomials  $$f_1,\ldots,f_{n-1},f_n,\sum_{i=1}^n U_iX_i$$ with respect to $(X_1,\ldots,X_n)$. Obviously, the same holds for $\theta$. 

Set $R:=\Res_{X,Y}(\overline{\rho}(\theta),\overline{\rho}(f_n(\theta_1,\ldots,\theta_n)))$. There exists an integer $N_1$ such that
$$X^{N_1}R \in (\overline{\rho}(\theta),\overline{\rho}(f_n(\theta_1,\ldots,\theta_n))) \subset A[V_1,\ldots,V_n,W_1,\ldots,W_n][X,Y]$$
and therefore we deduce that there exists an integer  $N_2$ such that
\begin{multline*}
	X^{N_1}X_n^{N_2}R \in (f_1,\ldots,f_n,\overline{\rho}(\sum_{i=1}^n U_iX_i)) \\ \subset A[V_1,\ldots,V_n,W_1,\ldots,W_n][X,Y][X_1,\ldots,X_n].	
\end{multline*}
Now, specializing $X$ to $\sum_{i=1}^nW_iX_i$ and $Y$ to $-\sum_{i=1}^nV_iX_i$ we obtain that
$$(\sum_{i=1}^nW_iX_i)^{N_1}X_n^{N_2}R \in (f_1,\ldots,f_n) \subset A[V_1,\ldots,V_n,W_1,\ldots,W_n][X_1,\ldots,X_n].$$
In other words, $(\sum_{i=1}^nW_iX_i)^{N_1} R$ is an inertia form of the polynomials $f_1,\ldots,f_n$ with respect to $(X_1,\ldots,X_n)$. Moreover, since $\sum_{i=1}^nW_iX_i$ is obviously not such an inertia form, we deduce that $R$ is. Consequently, there exists $$M \in A[V_1,\ldots,V_n,W_1,\ldots,W_n]$$ such that
\begin{equation}\label{eq:eq1mertens}
R:=\Res_{X,Y}(\overline{\rho}(\theta),\overline{\rho}(f_n(\theta_1,\ldots,\theta_n)))=M \Res(f_1,\ldots,f_n).	
\end{equation}

Looking at this equation, we see that both $R$ and $\Res(f_1,\ldots,f_n)$	are homogeneous with respect to the coefficients of the polynomial $f_n$ of the same degree $d_1\ldots d_{n-1}$. Therefore, $M$ must be independent of these coefficients, but it could depend on the degree $d_n$. To emphasize this property, we use the notation $M(f_1,\ldots,f_{n-1},d_n)$. If we specialize $f_n$ to $X_n^{d_n}$ in \eqref{eq:eq1mertens}, we obtain
$$ \Res_{X,Y}(\overline{\rho}(\theta),\overline{\rho}(\theta_n))^{d_n}=M(f_1,\ldots,f_{n-1},d_n) \Res(f_1,\ldots,f_{n-1},X_n)^{d_n}.$$
But on the other hand, form the definition of $M$, we have
 	$$ \Res_{X,Y}(\overline{\rho}(\theta),\overline{\rho}(\theta_n))=M(f_1,\ldots,f_{n-1},1) \Res(f_1,\ldots,f_{n-1},X_n).$$
By comparison, it follows that $M(f_1,\ldots,f_{n-1},d_n)=M(f_1,\ldots,f_{n-1},1)^{d_n}$ and hence it remains to determine $M(f_1,\ldots,f_{n-1},1)$. For that purpose, noticing  that
$\partial \overline{\rho}(\theta)/\partial Y= \sum_{i=1}^n W_i \overline{\rho}(\theta_i)$, we choose to specialize $f_n$ to the linear form $\sum_{i=1}^n W_iX_i$.  We obtain
$$ \Res_{X,Y}(\overline{\rho}(\theta),\frac{\partial \overline{\rho}(\theta)}{\partial Y})=M(f_1,\ldots,f_{n-1},1) \Res(f_1,\ldots,f_{n-1},\sum_{i=1}^n W_iX_i).$$
Now, by definition of $\Disc_{X,Y}(\overline{\rho}(\theta))$, we have
\begin{align*}
	\Res_{X,Y}(\overline{\rho}(\theta),\frac{\partial \overline{\rho}(\theta)}{\partial Y}) &= \Disc_{X,Y}(\overline{\rho}(\theta)) \Res_{X,Y}(\overline{\rho}(\theta),X) \\
	&= \Disc_{X,Y}(\overline{\rho}(\theta)). \overline{\rho}(\theta)(0,-1) \\
	&= \Disc_{X,Y}(\overline{\rho}(\theta)) \Res(f_1,\ldots,f_{n-1},-\sum_{i=1}^n W_iX_i) \\
	&= (-1)^{d_1\ldots d_{n-1}} \Disc_{X,Y}(\overline{\rho}(\theta)) \Res(f_1,\ldots,f_{n-1},\sum_{i=1}^n W_iX_i).
\end{align*}
It follows that $M(f_1,\ldots,f_{n-1},1)=(-1)^{d_1\ldots d_{n-1}} \Disc_{X,Y}(\overline{\rho}(\theta))$ and the first formula is proved.

\medskip

We turn to the proof of the second formula. For the sake of simplicity, define $$h:=\rho(f_n(\theta_1,\ldots,\theta_n)) \in A[V_1,\ldots,V_n,W_1,\ldots,W_n][X_1,\ldots,X_n]$$
and denote by $d_h$ its degree with respect to the variables $X_1,\ldots,X_n$. It is not hard to check that 
$d_h=d_n(d_1\ldots d_{n-1}-1)$ which is a positive integer by our assumption $\prod_{i=1}^n d_i>1$.  

 By applying Mertens' first formula, we obtain the equality
\begin{align}\label{eq:eq2mertens}
	\Res_{X,Y}(\overline{\theta},\overline{\rho}(h(\theta_1,\ldots,\theta_n))) & = (-1)^{d_1\ldots d_{n-1}d_h}\Disc_{X,Y}(\overline{\rho}(\theta))^{d_h}\Res(f_1,\ldots,f_{n-1},h) \\ \nonumber
	& = \Disc_{X,Y}(\overline{\rho}(\theta))^{d_h}\Res(f_1,\ldots,f_{n-1},h)
\end{align}	
From the definitions we have
\begin{align*}
	\overline{\rho}(h(\theta_1,\ldots,\theta_n)) & = \overline{\rho}\left( 
	\left[f_n(\theta_1,\ldots,\theta_n)\right](\ldots,V_i(\sum_{j=1}^nW_j\theta_j)-W_i(\sum_{j=1}^nV_j\theta_j),\ldots)
	\right)\\
	& = 	
	\left[f_n(\theta_1,\ldots,\theta_n)\right](\ldots,V_i(\sum_{j=1}^nW_j\overline{\rho}(\theta_j))-W_i(\sum_{j=1}^nV_j\overline{\rho}(\theta_j)),\ldots)\\
	& = 	
	\left[f_n(\theta_1,\ldots,\theta_n)\right](\ldots,V_i\frac{\partial \overline{\rho}(\theta)}{\partial Y}-W_i\frac{\partial \overline{\rho}(\theta)}{\partial X},\ldots).
\end{align*}
Thus, if we define $$F(X,Y):=\overline{\rho}(f_n(\theta_1,\ldots,\theta_n))=\left[f_n(\theta_1,\ldots,\theta_n)\right](\ldots,V_iX+W_iY,\ldots),$$ 
then
$$\overline{\rho}(h(\theta_1,\ldots,\theta_n)) = F\left(\frac{\partial \overline{\rho}(\theta)}{\partial Y},-\frac{\partial \overline{\rho}(\theta)}{\partial X}\right)= (-1)^{d_h}F\left(-\frac{\partial \overline{\rho}(\theta)}{\partial Y},\frac{\partial \overline{\rho}(\theta)}{\partial X}\right)$$
where the last equality holds because $\deg(F)=d_h$. Now, 
from Proposition \ref{F-J}, recall that 
$$\Res_{X,Y}(\overline{\rho}(\theta),F(-\frac{\partial \overline{\rho}(\theta)}{\partial Y},\frac{\partial \overline{\rho}(\theta)}{\partial X}))=\Res_{X,Y}(\overline{\rho}(\theta),F(X,Y)) \Disc_{X,Y}(\overline{\rho}(\theta))^{d_h}.$$
Therefore, we have (observe that $(-1)^{d_1\ldots d_{n-1}d_h}=1$)
$$	\Res_{X,Y}(\overline{\rho}(\theta),\overline{\rho}(h(\theta_1,\ldots,\theta_n))) =
\Res_{X,Y}(\overline{\rho}(\theta),F(X,Y)) \Disc_{X,Y}(\overline{\rho}(\theta))^{d_h}
$$
and using again the first Mertens' formula for $\Res_{X,Y}(\overline{\rho}(\theta),F(X,Y))$, we obtain
\begin{multline}\label{eq:eq3mertens}
	\Res_{X,Y}(\overline{\rho}(\theta),\overline{\rho}(h(\theta_1,\ldots,\theta_n))) = \\
	(-1)^{d_1\ldots d_n}\Disc_{X,Y}(\overline{\rho}(\theta))^{d_n}\Res(f_1,\ldots,f_n) \Disc_{X,Y}(\overline{\rho}(\theta))^{d_h}.	
\end{multline}
Now, the comparison of the equations \eqref{eq:eq2mertens} and \eqref{eq:eq3mertens} yields
\begin{multline*}
	\Disc_{X,Y}(\overline{\rho}(\theta))^{d_h}\Res(f_1,\ldots,f_{n-1},h) \\
			=(-1)^{d_1\ldots d_n}\Disc_{X,Y}(\overline{\rho}(\theta))^{d_n}\Res(f_1,\ldots,f_n) \Disc_{X,Y}(\overline{\rho}(\theta))^{d_h}.	
\end{multline*}
We conclude the proof by observing that $\Disc(\overline{\rho}(\theta)) \in A[V_1,\ldots,V_n,W_1,\ldots,W_n]$ is nonzero, a fact that we show in the following lemma.
\end{proof}

\noindent {\bf Lemma A} {\it
$\Disc_{X,Y}(\overline{\rho}(\theta))$ is nonzero in $A[V_1,\ldots,V_n,W_1,\ldots,W_n]$, where $A$ is the universal ring of the coefficients of the polynomials $f_1,\ldots,f_{n-1}$.	
}

\medskip

\begin{proof} We exhibit a specialization for which $\Disc(\overline{\rho}(\theta))$ is easily seen to be nonzero. We start by specializing each polynomial $f_i$, $i=1,\ldots,n-1$ to the product of $d_i$ generic linear forms 
	$$l_{i,j}:=U_{i,j,1}X_1+U_{i,j,1}X_2+\cdots+U_{i,j,n}X_n=\sum_{r=1}^{d_i}U_{i,j,r}X_r, \ \ i=1,\ldots,n, \ j=1,\ldots,d_i.$$
Set $A'=\ZZ[U_{i,j,r}: i=1,\ldots,n, \ j=1,\ldots,d_i, \ r=1,\ldots,n]$. After this specialization, we get
$$ 
\theta=\prod_{\stackrel{1\leq j_i\leq d_i}{i=1,\ldots n-1}} \det(l_{1,j_1},l_{2,j_2},\ldots,l_{n-1,j_n-1},U_1X_1+\ldots,U_nX_n) 
$$
in $A'[U_1,\ldots,U_n]$. For each $(n-1)$-uple $\lambda:=(j_1,\ldots,j_{n-1})$ in the above product we denote by $\Delta_\lambda(U_1,\ldots,U_n)$ the corresponding determinant. We deduce that
$$
\frac{\partial \overline{\rho}(\theta)}{\partial Y}= \sum_{\lambda} \left( 
\Delta_\lambda(W_1,\ldots,W_n) \prod_{\mu, \ \mu\neq \lambda} \overline{\rho}(\Delta_\mu)
\right).
$$
Now, on the one hand we have (the resultant and the discriminant are taken with respect to $X,Y$)
\begin{multline*}
	\Res\left(\overline{\rho}(\theta),\frac{\partial \overline{\rho}(\theta)}{\partial Y}\right)=\\
	\Disc(\overline{\rho}(\theta)). \overline{\rho}(\theta)(0,-1)=(-1)^{d_1\ldots d_{n-1}}\Disc(\overline{\rho}(\theta))\prod_\lambda \Delta_\lambda(W_1,\ldots,W_n),	
\end{multline*}
and on the other hand
\begin{align*}
\Res\left(\overline{\rho}(\theta),\frac{\partial \overline{\rho}(\theta)}{\partial Y}\right) & =
	\Res\left( \prod_\lambda \overline{\rho}(\Delta_\lambda), \sum_{\lambda} \left( 
	\Delta_\lambda(W_1,\ldots,W_n) \prod_{\mu, \ \mu\neq \lambda} \overline{\rho}(\Delta_\mu)\right)\right) \\
	&= \prod_\lambda \Res\left(\overline{\rho}(\Delta_\lambda), \sum_{\omega} \left( 
	\Delta_{\omega}(W_1,\ldots,W_n) \prod_{\mu, \ \mu \neq \omega} \overline{\rho}(\Delta_\mu)\right)\right) \\
	&= \prod_\lambda \Res\left(\overline{\rho}(\Delta_\lambda), 
	\Delta_\lambda(W_1,\ldots,W_n) \prod_{\mu, \ \mu\neq \lambda} \overline{\rho}(\Delta_\mu)\right) \\
	&= \left(\prod_\lambda \Delta_\lambda(W_1,\ldots,W_n)\right) 
	\prod_{\stackrel{\lambda, \mu}{\lambda\neq \mu}} \Res\left(\overline{\rho}(\Delta_\lambda), 
	 \overline{\rho}(\Delta_\mu)\right).
\end{align*}
Therefore, choosing an order for the $(n-1)$-uples $\lambda$, we deduce that
$$ \Disc(\overline{\rho}(\theta))=(-1)^{\frac{N^2+N}{2}}\prod_{\lambda<\mu} \Res\left(\overline{\rho}(\Delta_\lambda), 
 \overline{\rho}(\Delta_\mu)\right)^2$$
with $N=d_1\ldots d_{n-1}$. Moreover, for any $(n-1)$-uple $\lambda$, it is easy to see that $$\overline{\rho}(\Delta_\lambda)=\Delta_\lambda(V_1,\ldots,V_n)X+\Delta_\lambda(W_1,\ldots,W_n)Y.$$ 
It follows that in $A'[V_1,\ldots,V_n,W_1,\ldots,W_n]$ we have the equality
\begin{multline}\label{eq:discrhotheta}
	\Disc(\overline{\rho}(\theta))= (-1)^{\frac{N^2+N}{2}}\times \\ \prod_{\lambda<\mu} \left(
	\Delta_\lambda(V_1,\ldots,V_n) \Delta_\mu(W_1,\ldots,W_n)-\Delta_\lambda(W_1,\ldots,W_n)\Delta_\mu(V_1,\ldots,V_n)
	\right)^2.	
\end{multline}
To finish the proof, we specialize a little more our polynomials $f_1,\ldots,f_{n-1}$ by specializing each linear form $l_{i,j}$ to 
$X_i-U_{i,j,n}X_n$. Then, it is not hard to check that 
\begin{equation}\label{eq:eq4mertens}
\Delta_{\lambda=(j_1,\ldots,j_{n-1})}(V_1,\ldots,V_n)=U_{1,j_1}V_1+U_{2,j_2}V_2+\cdots+U_{n-1,j_{n-1}}V_{n-1}+V_n
\end{equation}
and hence that $\Delta_{\lambda}(0,\ldots,0,1)=1$. Therefore, we deduce that for any couple $(\lambda,\mu)$ such that $\lambda \neq \mu$ we have
\begin{multline*}
	\Delta_\lambda(0,\ldots,0,1) \Delta_\mu(W_1,\ldots,W_n)-\Delta_\lambda(W_1,\ldots,W_n)\Delta_\mu(0,\ldots,0,1)=\\ \Delta_\mu(W_1,\ldots,W_n)-\Delta_\lambda(W_1,\ldots,W_n)
\end{multline*}
and this quantity is clearly nonzero in view of \eqref{eq:eq4mertens}.
\end{proof}

\nocite{*}
\end{document}